\documentclass[reqno]{amsart}
\usepackage{amssymb}
\usepackage{amsfonts}
\usepackage{amsmath}
\usepackage{amssymb, amsmath}
\usepackage{graphicx}
\usepackage{hyperref}

\usepackage{color}

\newcommand{\newcom}{\newcommand}

\newcom{\al}{\alpha}
\newcom{\be}{\beta}
\newcom{\eps}{\epsilon}
\newcom{\veps}{\varepsilon}
\newcom{\ga}{\gamma}
\newcom{\Ga}{\Gamma}
\newcom{\ka}{\kappa}
\newcom{\Lam}{\Lambda}
\newcom{\lam}{\lambda}
\newcom{\Om}{\Omega}
\newcom{\om}{\omega}
\newcom{\Si}{\Sigma}
\newcom{\si}{\sigma}
\newcom{\tht}{\theta}
\newcom{\dtri}{\nabla}
\newcom{\tri}{\triangle}
\newcom{\oo}{\infty}
\newcom{\vphi}{\varphi}
\newcom{\cB}{{\mathcal B}}
\newcom{\cC}{{\mathcal C}}
\newcom{\cD}{{\mathcal D}}
\newcom{\cF}{{\mathcal F}}
\newcom{\cL}{{\mathcal L}}
\newcom{\cM}{{\mathcal M}}
\newcom{\cP}{{\mathcal P}}
\newcom{\cS}{{\mathcal S}}
\newcom{\cQ}{{\mathcal Q}}
\newcom{\cT}{{\mathcal T}}
\newcom{\cY}{{\mathcal Y}}
\newcom{\cZ}{{\mathcal Z}}
\newcom{\R}{\mathbb R}
\newcom{\T}{\mathbb T}
\newcom{\N}{\mathbb N}
\newcom{\Z}{\mathbb Z}
\newcom{\C}{\mathbb C}
\newcom{\E}{\mathbb E}

\newcommand{\FP}{\mathbf{P}}

\newcommand{\De}{\Delta}

\newcom{\f}{\frac}
\newcom{\di}{\displaystyle\int}
\newcom{\ds}{\displaystyle\sum}
\newcom{\dl}{\displaystyle\lim}
\newcom{\ov}{\overline}
\newcom{\sset}{\subset}
\newcom{\wt}{\widetilde}
\newcom{\pa}{\partial}
\newcom{\p}{\partial}
\newcom\na{\nabla}
\newcom{\suml}{\sum\limits}
\newcom{\supl}{\sup\limits}
\newcom{\intl}{\int\limits}
\newcom{\infl}{\inf\limits}
\newcom{\disp}{\displaystyle}
\newcom{\non}{\nonumber}
\newcom{\no}{\noindent}
\newcom{\QED}{$\square$}

\newtheorem{athm}{\bf \t}[section]
\newenvironment{thm} [1] {\def\t{#1}\begin{athm} \bf \rm} {\end{athm}}
\newcom{\bthm}{\begin{thm}}
\newcom{\ethm}{\end{thm}}

\newtheorem{theorem}{Theorem}[section]
\newtheorem{lemma}{Lemma}[section]

\newtheorem{corollary}{Corollary}[section]

\newcom{\beq}{\begin{equation}}
\newcom{\eeq}{\end{equation}}
\newcom{\ben}{\begin{eqnarray}}
\newcom{\een}{\end{eqnarray}}
\newcom{\beno}{\begin{eqnarray*}}
\newcom{\eeno}{\end{eqnarray*}}
\newcom{\bali}{\begin{aligned}}
\newcom{\eali}{\end{aligned}}

\numberwithin{equation}{section}

\begin{document}

\title[Two fluid Euler-Maxwell system]
{Darcy's law and diffusion for a two-fluid Euler-Maxwell  system with dissipation}

\author{Renjun Duan}
\address{(RJD)
Department of Mathematics, The Chinese University of Hong Kong, Shatin, Hong Kong}
\email{rjduan@math.cuhk.edu.hk}

\author{Qingqing Liu}
\address{(QQL)
The Hubei Key Laboratory of Mathematical Physics, School of
Mathematics and Statistics, Central China Normal University, Wuhan,
430079, P. R. China} \email{shuxueliuqingqing@126.com}

\author{Changjiang Zhu}
\address{(CJZ)
The Hubei Key Laboratory of Mathematical Physics, School of
Mathematics and Statistics, Central China Normal University, Wuhan,
430079, P. R. China} \email{cjzhu@mail.ccnu.edu.cn}




\begin{abstract}
This paper is concerned with the large-time behavior of solutions to the Cauchy problem on the two-fluid Euler-Maxwell system with dissipation when initial data are around a constant equilibrium state. The main goal is the rigorous justification of diffusion phenomena in fluid plasma at the linear level.  Precisely, motivated by the classical Darcy's law for the nonconductive fluid, we first give a heuristic derivation of the asymptotic equations of the Euler-Maxwell system in large time. It turns out that both the density and the magnetic field tend time-asymptotically to the diffusion equations with diffusive coefficients explicitly determined by given physical parameters. Then, in terms of the Fourier energy method, we analyze the linear dissipative structure of the system, which implies the almost exponential time-decay property of solutions over the high-frequency domain. The key part of the paper is the spectral analysis of the linearized system, exactly capturing the diffusive feature of solutions over the low-frequency domain. Finally, under some conditions on initial data, we show the convergence  of the  densities and the magnetic field to the corresponding linear diffusion waves with the rate $(1+t)^{-5/4}$  in $L^2$ norm and also the convergence of the velocities and the electric field to the corresponding asymptotic profiles given in the sense of the geneneralized Darcy's law with the faster rate  $(1+t)^{-7/4}$  in $L^2$ norm.  Thus, this work can be also regarded as the mathematical proof of the Darcy's law in the context of collisional fluid plasma.
\end{abstract}

\maketitle
\thispagestyle{empty}

\tableofcontents

\section{Introduction}

It is generally believed that the Darcy's law governs the motion of the inviscid flow with frictional damping \cite{Liu} or the slow viscous flow \cite{Lions} in large time. It is quite nontrivial to mathematically justify the large-time behavior of solutions to those relative physical systems, particularly in the case when vacuum appears, cf.~\cite{DP,HMP, MM,Pa}. Besides, there are also some results, for instance, see \cite{Im} and references therein, to discuss the modified Darcy's law for conducting porous media. In the paper, we attempt to give a rigorous proof of Darcy's laws and diffusion phenomena in the context of collisional fluid plasma whenever the densities of fluids are close to non-vacuum states.

In a weakly ionised gas with a small enough ionisation fraction, charged particles will interact primarily by means of elastic collisions with neutral atoms rather than with other charged particles, cf.~\cite[Chapter 12.4]{GR}. In such situation, the motion of fluid plasmas consisting of ions $(\al=i)$, electrons $(\al=e)$ and neutral atoms is generally governed by the two-fluid Euler-Maxwell system in three space dimensions
\begin{eqnarray}\label{1.1}
&&\left\{\begin{aligned}
&\partial_t n_{\alpha}+\nabla\cdot(n_\alpha u_\alpha)=0,\\
&m_{\alpha}n_{\alpha}(\partial_t u_{\alpha}+u_{\alpha} \cdot \nabla
u_\alpha)+\nabla
p_\alpha(n_\alpha)\\
&\qquad\qquad\qquad\qquad=q_{\alpha}n_{\alpha}\left(E+\frac{u_{\alpha}}{c}\times B\right)-\nu_\alpha m_{\alpha} n_{\alpha}u_\alpha,\\
&\partial_t E-c\nabla\times B=-4\pi\sum_{\alpha=i,e}q_{\alpha}n_{\alpha}u_{\alpha},\\
&\partial_t B+c\nabla \times E=0,\\
&\nabla \cdot E=4\pi\sum_{\alpha=i,e}q_{\alpha}n_{\alpha}, \ \
\nabla \cdot B=0.
\end{aligned}\right.
\end{eqnarray}
Here the unknowns are $n_{\alpha}=n_{\alpha}(t,x)\geq 0 $ and $
u_{\alpha}=u_{\alpha}(t,x)\in \mathbb{R}^{3}$ with $\al=i,e$, denoting the densities and velocities of the $\al$-species respectively, and also $ E=E(t,x)\in
\mathbb{R}^{3}$ and $ B=B(t,x)\in \mathbb{R}^{3}$, denoting the self-consistent electron and  magnetic fields respectively, for $t>0$ and $x\in \R^3$. For the $\al$-species, $p_\al(\cdot)$ depending only on the density is the pressure function which is smooth and satisfies $p_\al'(n)>0$ for $n>0$, and for simplicity we assume in the paper that the fluid is isothermal and hence $p_\al(n)=T_\al n$ for the constant temperature $T_\al>0$. Constants $m_\al>0$, $q_\al$, $\nu_\al>0$, $c>0$ stand for the mass, charge and collision frequency of $\al$-species and the speed of light, respectively.
The constant $4\pi$ appearing in the system is related to the spatial dimension. Notice $q_{e}=-e$ and $q_{i}=Ze$ in the general physical situation,
where $e>0$ is the electronic charge and $Z\geq 1$ is an positive integer. Without loss of generality,
we may assume  $Z=1$ through the paper, since it can be normalized to be unit under the transformation \begin{eqnarray*}
\tilde{n}_{i}=Zn_{i},\ \ \tilde{q}_{i}=e,\ \
\tilde{m}_{i}=\frac{m_{i}}{Z},\ \
\tilde{p}_{i}(n)=p_{i}\left(\frac{n}{Z}\right).
\end{eqnarray*}
%
%
Initial data are given by
\begin{eqnarray}\label{1.2}
[n_{\alpha},u_{\alpha},E,B]|_{t=0}=[n_{\alpha 0},u_{\alpha
0},E_{0},B_0], 
\end{eqnarray}
with the compatibility condition
\begin{eqnarray}\label{1.3}
\nabla \cdot E_0=4\pi \sum_{\alpha=i,e}q_{\alpha}n_{\alpha 0}, \ \
\nabla \cdot B_0=0.
\end{eqnarray}

The paper is concerned with the large-time asymptotic behavior of solutions to the Cauchy problem on the two-fluid Euler-Maxwell system with collisions whenever initial data are close to a  constant equilibrium state $[n_\al=1,u_\alpha=0,E=0,B=0]$. Notice that collision terms play a key role in the analysis of the problem, see \cite{STW} for instance. For that purpose, we introduce the large-time asymptotic profile as follows.
Let $G_\mu(t,x)=(4\pi \mu t)^{-3/2}\exp \{-|x|^2/(4\mu t)\}$ be the heat kernel with the diffusion coefficient $\mu>0$. Let us define the ambipolar diffusive coefficient $\mu_1>0$ and the magnetic diffusive coefficient $\mu_2>0$ by
\begin{equation}
\label{def.heat}
\displaystyle \mu_1=\frac{T_{i}+T_{e}}{m_{i}\nu_{i}+m_{e}\nu_{e}},\quad
\mu_2=\frac{c^{2}m_{i}\nu_{i}m_{e}\nu_{e}}{4\pi e^{2}\left(
m_{i}\nu_{i}+m_{e}\nu_{e}\right)},
\end{equation}
respectively. Corresponding to given initial data \eqref{1.2}, we define the asymptotic profile $[\overline{n},\overline{u}_\alpha,\overline{E},\overline{B}]$ by
\begin{eqnarray}
&&\overline{n}=\sum_{\al=i,e} \frac{m_{\al}\nu_{\al}}{m_{i}\nu_{i}+m_{e}\nu_{e}}G_{\mu_1}(t,\cdot)\ast (n_{\al 0}-1),\label{def.pp1}\\
&&\overline{B}=G_{\mu_2}(t,\cdot)\ast B_0,\label{def.pp2}
\end{eqnarray}
and
\begin{eqnarray}
&&\overline{u}_{\alpha}(t,x)=-\frac{T_{i}+T_{e}}{m_{i}\nu_{i}+m_{e}\nu_{e}}\nabla
\overline{n}(t,x)+\frac{c}{4\pi q_\al}\frac{
m_{e}\nu_{e}}{m_{i}\nu_{i}+m_{e}\nu_{e}}\nabla \times
\overline{B}(t,x),\quad \al=i,e,\label{def.pp3}\\
&&\overline{E}(t,x)=\frac{T_{i}m_{e}\nu_{e}-T_{e}m_{i}\nu_{i}}{e(m_{i}\nu_{i}+m_{e}\nu_{e})}\nabla
\overline{n}(t,x)+\frac{c}{4\pi
e^{2}}\frac{m_{i}\nu_{i}m_{e}\nu_{e}}{
m_{i}\nu_{i}+m_{e}\nu_{e}}\nabla\times
\overline{B}(t,x).\label{def.pp4}
\end{eqnarray}
The main result of the paper is stated as follows.

\begin{theorem}\label{thm.main}
There are constants $\eps>0$ and $C>0$ such that if
\begin{equation}
\label{thm.ma1}
\sum_{\al=i,e} \|[n_{\alpha0}-1,u_{\alpha 0}]\|_{H^{11}\cap L^1} +\|[E_0,B_0]\|_{H^{11}\cap L^1}<\eps,
\end{equation}
then the Cauchy problem \eqref{1.1}, \eqref{1.2}, \eqref{1.3} admits a unique global solution
\begin{equation}
\label{thm.ma3}
n_\alpha-1,u_\alpha, E,B\in C([0,\infty);H^5(\R^3)),
\end{equation}
satisfying
\begin{equation}
\label{thm.ma4}
\sum_{\al=i,e} \|n_\al-1-\overline{n}\|_{L^2} +\|B-\overline{B}\|_{L^2}\leq C (1+t)^{-\frac{5}{4}},
\end{equation}
and
\begin{equation}
\label{thm.ma5}
\sum_{\al=i,e}\|u_\al-\overline{u}_\al\|_{L^2} +\|E-\overline{E}\|_{L^2}\leq C (1+t)^{-\frac{7}{4}},
\end{equation}
for all $t\geq 0$.
\end{theorem}

We give a few remarks on Theorem \ref{thm.main}.  First of all, from the proof later on, under the assumption \eqref{thm.ma1} the solution to the Cauchy problem \eqref{1.1}, \eqref{1.2}, \eqref{1.3} around the constant equilibrium state enjoys the time-decay property
\begin{equation*}
\sum_{\al=i,e}\|[n_\al-1,u_\al]\|_{L^2}+\|[E,B]\|_{L^2}\leq C (1+t)^{-\frac{3}{4}},
\end{equation*}
where the time-decay rate must be optimal for general initial data with $B_0\neq 0$ due to those results from the spectral analysis given in Section \ref{sec4}; see Corollary \ref{corollary.decayL} for instance.  On the other hand the large-time asymptotic profile also satisfies
\begin{equation*}
\|\overline{n}\|_{L^2}+\|\overline{B}\|_{L^2}\leq C (1+t)^{-\frac{3}{4}},\quad \sum_{\al=i,e}\|\overline{u}_\al\|_{L^2}+\|\overline{E}\|_{L^2}\leq C (1+t)^{-\frac{5}{4}},
\end{equation*}
which are  also optimal in terms of the definition \eqref{def.pp1}, \eqref{def.pp2}, \eqref{def.pp3}, \eqref{def.pp4} of $[\overline{n},\overline{u}_\al,\overline{E},\overline{B}]$. Therefore it is nontrivial to obtain the faster time-decay rates \eqref{thm.ma4} and \eqref{thm.ma5}, and this  also assures that $[1+\overline{n},\overline{u}_\al,\overline{E},\overline{B}]$ indeed can be regarded as the more accurate large-time asymptotic profile for solutions to the Cauchy problem under consideration, compared to the trivial constant equilibrium state. Notice that $\overline{n}$ and $\overline{B}$ are diffusion waves by  \eqref{def.pp1} and \eqref{def.pp2} as well as \eqref{def.heat}, and $\overline{u}_\al$ and $\overline{E}$ are defined in terms of those two diffusion waves by \eqref{def.pp3} and \eqref{def.pp4}. From the heuristic derivation of the large-time asymptotic profiles in the next section, we see that the asymptotic profiles can be solved from the asymptotic equations obtained in a formal way in the sense of the Darcy's law. In the case without any electromagnetic field, there have been extensive mathematical studies of the large-time behavior for the damped Euler system basing on the Darcy's law; see \cite{GHL, HL, MM, Pa} and reference therein. However few rigorous results are known for such physical law in the context of two-fluid plasma with collisions. This work can be regarded to some extent as the generalisation of the Darcy's law for the classical non-conductive fluid to the plasma fluid under the influence of the self-consistent electromagnetic field.

Second, by \eqref{thm.ma1} and \eqref{thm.ma3}, there is a discrepancy between regularities of initial data and the solution. The reason for the discrepancy is that the Euler-Maxwell with collisions is of the regularity-loss type, which is essentially induced by the fact that eigenvalues of the linearized system may tend asymptotically to the imaginary axis as the frequency goes to infinity; see \cite{Duan} in the one-fluid case. There has been a general theory developed in \cite{UDK} in terms of the Fourier energy method to study the decay structure of general symmetric hyperbolic systems with partial relaxations of the regularity-loss type. The main feature of time-decay properties for such regularity-loss system is that solutions over the high-frequency domain can still gain the enough time-decay rate by compensating enough regularity of initial data.

Third, the key point of Theorem \ref{thm.main} is to present the convergence in $L^2$ norm  of the solution $[n_\al,u_\al, E,B]$ to the profile $[1+\overline{n},\overline{u}_\al,\overline{E},\overline{B}]$ if initial data approach the constant steady state in the sense of \eqref{thm.ma1}. 
There could be several direct generalisations of the current result. As in \cite{UK}, it can be expected to obtain the convergence rates for the derivatives up to to some order. In general, the higher the order of derivatives is, the faster they decay in time. Another possible approach for obtaining the global existence and convergence of solutions to the constant steady state is to introduce as in \cite{GW} the negative Sobolev space basing on the pure energy method together with the functional interpolation inequalities, where the advantage is that both $L^2$ norms of the higher derivatives and $L^1$ norms of the zero-order are not necessarily small. However, it seems still nontrivial to apply such method to obtain the large-time asymptotic behaviour \eqref{thm.ma4} and \eqref{thm.ma5}.

The final remark is concerned with the nonlinear diffusion of the two fluid Euler-Maxwell system with collisions. In fact, the current work is done at the linearized level. Even for the general pressure functions $P_\al$ $(\al=i,e)$, by using the same formal derivation as in Section \ref{sec2}, the density satisfies the nonlinear heat equation
$$
\pa_t \overline{n} -\De P(\overline{n})=0,
$$
where $P(\cdot)$ is in connection with  $P_\al$ $(\al=i,e)$ as well as other physical parameters. The nonlinear heat equation above is also a type of the porous medium equation. Thus, it would be interesting and challenging to further investigate the asymptotic stability of the nonlinear diffusion waves, cf.~\cite{HL}. We hope to report it in the future study.

To prove Theorem \ref{thm.main} we need to carry out the spectral analysis of the linearized system around the constant steady state; see \cite{LK}, for instance. In fact, the solution can be written as the sum of the fluid part and the electromagnetic part in the form of
\begin{eqnarray*}
&& \left[
  \begin{array}{c}
   \rho_{\alpha}(t,x)\\
   u_{\alpha}(t,x)\\
   E(t,x)\\
   B(t,x)\\
 \end{array}\right]
 =
 \left[
  \begin{array}{c}
   \rho_{\alpha}(t,x)\\
   u_{\alpha,\parallel}(t,x)\\
   E_{\parallel}(t,x)\\
   0\\
 \end{array}\right]+
\left[
  \begin{array}{c}
   0 \\
   u_{\alpha,\perp}(t,x)\\
   E_{\perp}(t,x)\\
   B(t,x)\\
 \end{array}\right].
\end{eqnarray*}
However it seems difficult to give an explicit representation of
solutions to two eigenvalue problems due to the high phase
dimensions under consideration. The main idea is to obtain the
asymptotic expansions of solutions to the linearized system as the
frequency $|k|\to 0$ (cf.~\cite{Duan1}, \cite{IK}); see Section \ref{sec4}. One trick to deal with
the electromagnetic part is to first reduce the system to the
high-order ODE of the magnetic field $B$ only, then study the
asymptotic expansion of $B$ as $|k|\to 0$, and finally apply the
Fourier energy method to estimate the other two components
$u_{\al,\perp}$ and $E_\perp$(cf.~\cite{RSY} and reference therein);
see Lemma \ref{errorUE}.  For $|k|\to \infty$, it can be directly
treated by the Fourier energy method since the linearized solution
operator in the Fourier space behaves like
$$
\exp\left\{-\frac{\lambda|k|^2}{(1+|k|^2)^2}t\right\},
$$
which leads to the almost exponential time-decay depending on regularity of initial data; see Section \ref{sec3}. In the mean time, we find that the large-time behavior of solutions to the two-fluid Euler-Maxwell system \eqref{1.1} is governed by the following two subsystems
\begin{equation*}
\left\{\begin{array}{l}
 \displaystyle \pa_t n +\nabla \cdot (n u_{\parallel})=0,  \\[3mm]
\displaystyle      \nabla   P_\al (n)=q_\al n E_\parallel -\nu_\al m_\al n u_\parallel, \quad \al=i,e,
\end{array}\right.
\end{equation*}
and
\begin{equation*}
\left\{\begin{array}{l}
 \displaystyle q_\al E_\perp -\nu_\al m_\al u_{\al,\perp}=0,\quad \al=i,e,\\[3mm]
\displaystyle  -c \nabla \times B=-4\pi n \sum_{\al=i,e} q_\al u_{\al,\perp},\\[3mm]
 \displaystyle \pa_t B +c \nabla \times E_\perp=0.
\end{array}\right.
\end{equation*}
For more details see Section \ref{sec2} and Section \ref{sec4}.

Finally we would mention the following works related to the paper: some derivations and numerical computations of the relative models \cite{BCD, BMP,DDS,T}, global existence and large-time behavior for the damped Euler-Maxwell system \cite{CJW, Duan,DLZ,P,PWG,TW,UK, USK, WFL,X}, global existence in the non-damping case \cite{GM,GIP,LP}, and asymptotic limits under small parameters \cite{HP, PW-08,PWG}.

The rest of the paper is organised as follows. In Section \ref{sec2}, we give the heuristic derivation of diffusion waves motivated by the classical Darcy's law. In Section \ref{sec3} we reformulate the Cauchy problem on the Euler-Maxwell system around the constant steady state, and study the decay structure of the linearized homogeneous system by the Fourier energy method. In Section \ref{sec4}, we present the spectral analysis of the linearized system by three parts. The fist part is for the fluid, the second one for the electromagnetic field, and the third one for the extra time-decay of solutions with special initial data. The result in the third part accounts for estimating the inhomogeneous source terms. In Section \ref{sec5}, we first prove the global existence of solutions by the energy method, show the time asymptotic rate of solutions around the constant states and then obtain the main result concerning the time asymptotic rate around linear diffusion waves.

\medskip
\noindent{\it Notations.} Let us introduce some notations for the use throughout this paper.
$C$ denotes some positive (generally large) constant and $ \lambda$
denotes some positive (generally small) constant, where both $C$ and
$ \lambda$ may take different values in different places. For two
quantities $a$ and $b$, $a\sim b$ means $\lambda a \leq  b \leq
\frac{1}{\lambda} a $ for a generic constant $0<\lambda<1$. For any
integer$m\geq 0$, we use $H^{m}$, $\dot{H}^{m}$ to denote the usual
Sobolev space $H^{m}(\mathbb{R}^{3})$ and the corresponding
$m$-order homogeneous Sobolev space, respectively. Set $L^{2}=H^{m}$
when $m = 0$. For simplicity, the norm of $ H^{m}$ is denoted by
$\|\cdot\|_{m} $ with $\|\cdot \|=\|\cdot\|_{0}$. We use $
\langle\cdot, \cdot \rangle$ to denote the inner product over the
Hilbert space $ L^{2}(\mathbb{R}^{3})$, i.e.
\begin{eqnarray*}
\langle f,g \rangle=\int_{\mathbb{R}^{3}} f(x)g(x)dx,\ \ \ \  f =
f(x),\ \  g = g(x)\in L^2(\mathbb{R}^{3}).
\end{eqnarray*}
 For a multi-index $\alpha =
[\alpha_1, \alpha_2, \alpha_3]$, we denote $\partial^{\alpha} =
\partial^{\alpha_{1}}_ {x_1}\partial^{\alpha_{2}}_ {x_2} \partial^{\alpha_{3}}_ {x_3} $.
The length of $ \alpha$ is $|\alpha| = \alpha_1 + \alpha_2 +
\alpha_3$. For simplicity, we also
set $\partial_{j}=\partial_{x_{j}}$ for $j = 1, 2, 3$.


\section{Heuristic derivation of diffusion waves}\label{sec2}


In this section we would provide a heuristic derivation of the large-time asymptotic equations of the densities, velocities and the electromagnetic field. Indeed, both the densities and the magnetic field satisfy the diffusion equations with different diffusion coefficients in terms of those physical parameters appearing in the system, and the velocities and the electric field are defined by the densities and the magnetic field according to the Darcy's law.

\subsection{Diffusion of densities}
We first give a formal derivation of the large-time asymptotic equations of densities and velocities. Assume the quasineutral condition
\begin{eqnarray}\label{de.qnc}
n_{i}=n_{e}=n(t,x),\quad u_{i}=u_{e}=u(t,x),
\end{eqnarray}
and also assume that the background magnetic field is a constant vector, for instance, $B=(0,0,|B|)$ is constant along $x_{3}$-direction. Note that $|B|$ here is not necessarily assumed to be zero.  

We start from the
asymptotic momentum equations for $\al=i$ and $e$:
\begin{eqnarray}
\label{de.mie}
\nabla
p_\alpha(n)=q_{\alpha}n\left(E+\frac{u}{c}\times
B\right)-\nu_\alpha m_{\alpha} n u.
\end{eqnarray}
Along $B$, \eqref{de.mie} reduces to
\begin{eqnarray*}
\nabla p_\alpha(n)\cdot B=q_{\alpha}n E\cdot
B-\nu_\alpha m_{\alpha} n u \cdot B.
\end{eqnarray*}
i.e.,
\begin{eqnarray*}
\left\{
\begin{aligned}
&\nabla p_i(n)\cdot B=en E\cdot B-\nu_i m_{i} n u \cdot
B,\\
&\nabla p_e(n)\cdot B=-en E\cdot B-\nu_e m_{e} nu \cdot
B.
\end{aligned}
\right.
\end{eqnarray*}
It can be further written in the matrix form:
\begin{eqnarray*}
\left(\begin{array}{c}
\nabla p_i\cdot B\\
\nabla p_e\cdot B
\end{array} \right)=
\left(\begin{array}{cc}
  en\ &\ -\nu_{i}m_{i} n \\
-en  \ &\ -\nu_{e}m_{e} n
\end{array} \right)
\left(
\begin{array}{c}
E\cdot B\\
u\cdot B
\end{array}
\right).
\end{eqnarray*}
One can solve $E\cdot B$ and $u\cdot B$  as
\begin{eqnarray*}
\begin{aligned}
&u\cdot B=-\frac{1}{\nu_{e}m_{e}+\nu_{i}m_{i}}\frac{(\nabla
p_{i}+\nabla p_{e})\cdot B}{n},\\
&E\cdot B=\frac{1}{n} \frac{\left(\frac{\nabla
p_{i}}{m_{i}\nu_{i}}-\frac{\nabla p_{e}}{m_{e}\nu_{e}}\right)\cdot
B}{\frac{e}{\nu_{i}m_{i}}+\frac{e}{\nu_{e}m_{e}}}.
\end{aligned}
\end{eqnarray*}
Notice that since $B=(0,0,|B|)$ is along the $x_3$-direction, then
\begin{eqnarray}\label{u3E3}
\begin{aligned}
&u_{3}=-\frac{1}{\nu_{e}m_{e}+\nu_{i}m_{i}}\frac{\partial_{3}(
p_{i}(n)+  p_{e}(n)) }{n},\\
&E_{3}=\frac{1}{\frac{e}{\nu_{i}m_{i}}+\frac{e}{\nu_{e}m_{e}}}
\frac{\partial_{3}\left(\frac{ p_{i}(n)}{m_{i}\nu_{i}}-\frac{
p_{e}(n)}{m_{e}\nu_{e}}\right)}{n}.
\end{aligned}
\end{eqnarray}
Along the $x_1x_2$-plane normal to $B$, noticing $u\times B=(u_{2}|B|,-u_{1}|B|,0)$,
%
%
 \eqref{de.mie} reduces to
\begin{eqnarray*}
\left\{
\begin{aligned}
\partial_{1}p_{\alpha}=q_{\alpha}n\left(E_{1}+\frac{|B|}{c}u_{2}\right)-\nu_\alpha m_{\alpha}
nu_{1},\\
\partial_{2}p_{\alpha}=q_{\alpha}n\left(E_{2}-\frac{|B|}{c}u_{1}\right)-\nu_\alpha m_{\alpha}
nu_{2},
\end{aligned}
\right.
\end{eqnarray*}
i.e.,
\begin{eqnarray}\label{uErelation}
\left(\begin{array}{cc}
-m_{\alpha}\nu_{\alpha}n\ &\ \frac{q_{\alpha}n}{c}|B|\\
-\frac{q_{\alpha}n}{c}|B|\ &\
-m_{\alpha}\nu_{\alpha}n
\end{array} \right)\left(\begin{array}{c}
u_{1}\\
u_{2}
\end{array} \right)+q_{\alpha}n
\left(\begin{array}{c}
E_{1}\\
E_{2}
\end{array} \right)=
\left(\begin{array}{c}
\partial_{1}p_{\alpha}\\
\partial_{2}p_{\alpha}
\end{array} \right).
\end{eqnarray}
This implies
\begin{eqnarray}\label{udetermin1}
\left(\begin{array}{c}
u_{1}\\
u_{2}
\end{array} \right)=\left(\begin{array}{cc}
-m_{\alpha}\nu_{\alpha}n\ &\ \frac{q_{\alpha}n}{c}|B|\\
-\frac{q_{\alpha}n}{c}|B|\ &\
-m_{\alpha}\nu_{\alpha}n
\end{array} \right)^{-1}\left[
-q_{\alpha}n \left(\begin{array}{c}
E_{1}\\
E_{2}
\end{array}
\right)+\left(\begin{array}{c}
\partial_{1}p_{\alpha}\\
\partial_{2}p_{\alpha}
\end{array} \right)
\right],
\end{eqnarray}
for $\al=i$ and $e$. We denote
\begin{eqnarray*}
A_{\alpha}:=\left(\begin{array}{cc}
-m_{\alpha}\nu_{\alpha}n\ &\ \frac{q_{\alpha}n}{c}|B|\\
-\frac{q_{\alpha}n}{c}|B|\ &\
-m_{\alpha}\nu_{\alpha}n
\end{array} \right).
\end{eqnarray*}
Then letting the right-hand terms of  \eqref{udetermin1} be equal for $\alpha=i$ and $e$ further implies
\begin{eqnarray*}
-q_{i}nA_{i}^{-1} \left(\begin{array}{c}
E_{1}\\
E_{2}
\end{array}
\right)+A_{i}^{-1}\left(\begin{array}{c}
\partial_{1}p_{i}\\
\partial_{2}p_{i}
\end{array} \right)=-q_{e}nA_{e}^{-1} \left(\begin{array}{c}
E_{1}\\
E_{2}
\end{array}
\right)+A_{e}^{-1}\left(\begin{array}{c}
\partial_{1}p_{e}\\
\partial_{2}p_{e}
\end{array} \right).
\end{eqnarray*}
Due to the isothermal assumption $p_{\alpha}(n)=T_{\alpha}n$, one has
\begin{eqnarray*}
(q_{i}nA_{i}^{-1}-q_{e}nA_{e}^{-1}) \left(\begin{array}{c}
E_{1}\\
E_{2}
\end{array}
\right)=(T_{i}A_{i}^{-1}-T_{e}A_{e}^{-1})\left(\begin{array}{c}
\partial_{1}n\\
\partial_{2}n
\end{array} \right).
\end{eqnarray*}
Therefore,
\begin{eqnarray}\label{Edetermin}
 \left(\begin{array}{c}
E_{1}\\
E_{2}
\end{array}
\right)=(q_{i}nA_{i}^{-1}-q_{e}nA_{e}^{-1})^{-1}(T_{i}A_{i}^{-1}-T_{e}A_{e}^{-1})
\left(\begin{array}{c}
\partial_{1}n\\
\partial_{2}n
\end{array} \right).
\end{eqnarray}
Plugging \eqref{Edetermin} back into
\eqref{udetermin1} gives
\begin{eqnarray}\label{udetermin2}
&&\begin{aligned} \left(\begin{array}{c}
u_{1}\\
u_{2}
\end{array}\right)=&\left[
-q_{i}nA_{i}^{-1}(q_{i}nA_{i}^{-1}-q_{e}nA_{e}^{-1})^{-1}(T_{i}A_{i}^{-1}-T_{e}A_{e}^{-1})
+T_{i}A_{i}^{-1} \right]\left(\begin{array}{c}
\partial_{1}n\\
\partial_{2}n
\end{array} \right).
\end{aligned}
\end{eqnarray}
Let us give an explicit computation of the coefficient
matrix in \eqref{udetermin2}:
\begin{equation}
\label{def.G}
G=-q_{i}nA_{i}^{-1}(q_{i}nA_{i}^{-1}-q_{e}nA_{e}^{-1})^{-1}(T_{i}A_{i}^{-1}-T_{e}A_{e}^{-1})
+T_{i}A_{i}^{-1}.
\end{equation}
Notice $q_{i}=e$, $q_{e}=-e$, and
\begin{eqnarray*}
A_{\alpha}^{-1}=\frac{1}{\det A_{\alpha}}\left(\begin{array}{cc}
-m_{\alpha}\nu_{\alpha}n\ &\ -\frac{q_{\alpha}n}{c}|B|\\
\frac{q_{\alpha}n}{c}|B|\ &\
-m_{\alpha}\nu_{\alpha}n
\end{array} \right)=\frac{1}{\det A_{\alpha}}A_{\alpha}^{T},
\end{eqnarray*}
where $\det
A_{\alpha}=n^2(m_{\alpha}^{2}\nu_{\alpha}^{2}+\frac{q_{\alpha}^{2}}{c^{2}}|B|^{2})$.
To cancel $n^2$ in $\det A_{\alpha}$,  we write
\begin{eqnarray*}
\begin{aligned}
n A_{\alpha}^{-1}=&\frac{n}{\det
A_{\alpha}}A_{\alpha}^{T}
 =&-
\left(\begin{array}{cc}
\frac{m_{\alpha}\nu_{\alpha}}{m_{\alpha}^{2}\nu_{\alpha}^{2}+\frac{1}{c^{2}}e^{2}|B|^{2}}\
&\
\frac{\frac{q_{\alpha}}{c}|B|}{m_{\alpha}^{2}\nu_{\alpha}^{2}+\frac{1}{c^{2}}e^{2}|B|^{2}}\\[3mm]
\frac{-\frac{q_{\alpha}}{c}|B|}{m_{\alpha}^{2}\nu_{\alpha}^{2}+\frac{1}{c^{2}}e^{2}|B|^{2}}\
&\
\frac{m_{\alpha}\nu_{\alpha}}{m_{\alpha}^{2}\nu_{\alpha}^{2}+\frac{1}{c^{2}}e^{2}|B|^{2}}
\end{array} \right):= -K_{\alpha}.
\end{aligned}
\end{eqnarray*}
Then one can compute \eqref{def.G} as
\begin{eqnarray}
G&=&-q_{i}nA_{i}^{-1}(q_{i}nA_{i}^{-1}-q_{e}nA_{e}^{-1})^{-1}(T_{i}A_{i}^{-1}-T_{e}A_{e}^{-1})
+T_{i}A_{i}^{-1}\notag\\
&=& -\frac{e n}{\det A_{i}}A_{i}^{T}\left[\frac{e n}{\det
A_{i}}A_{i}^{T}+\frac{e n}{\det
A_{e}}A_{e}^{T}\right]^{-1}\left[\frac{T_{i}}{\det
A_{i}}A_{i}^{T}-\frac{T_{e}}{\det
A_{e}}A_{e}^{T}\right]+\frac{T_{i}}{\det A_{i}}A_{i}^{T}\notag\\
&=&-\frac{e n}{\det A_{i}}A_{i}^{T}(-e M)^{-1}\frac{T_{i}}{\det
A_{i}}A_{i}^{T}+\frac{e n}{\det A_{i}}A_{i}^{T}(-e
M)^{-1}\frac{T_{e}}{\det A_{e}}A_{e}^{T}+(-e M)(-e M)^{-1}\frac{T_i}{\det A_{i}}A_{i}^{T}\notag\\
&=&\frac{e n}{\det A_{e}}A_{e}^{T}(-e M)^{-1}\frac{T_{i}}{\det
A_{i}}A_{i}^{T}+\frac{e n}{\det A_{i}}A_{i}^{T}(-e
M)^{-1}\frac{T_{e}}{\det A_{e}}A_{e}^{T}\notag\\
&=&-eK_{e}\frac{1}{-e}(K_{i}+K_{e})^{-1}\frac{T_{i}}{n}(-K_{i})+(-eK_{i})\frac{1}{-e}(K_{i}+K_{e})^{-1}\frac{T_{e}}{n}(-K_{e})\notag\\
&=&-\frac{T_{i}}{n}K_{e}(K_{i}+K_{e})^{-1}K_{i}-\frac{T_{e}}{n}K_{i}(K_{i}+K_{e})^{-1}K_{e},\label{de.Gf}
\end{eqnarray}
where $M=K_i+K_e$.
Denoting
$$
C_{i}=m_{i}^{2}\nu_{i}^{2}+\frac{1}{c^{2}}e^{2}|B|^{2},\quad
C_{e}=m_{e}^{2}\nu_{e}^{2}+\frac{1}{c^{2}}e^{2}|B|^{2},
$$
one has
\begin{eqnarray*}
\begin{aligned}
K_{i}+K_{e}=&\frac{1}{C_{i}}\left(\begin{array}{cc}
m_{i}\nu_{i} \ &\ \frac{e|B|}{c}\\[3mm]
-\frac{e|B|}{c}\ &\ m_{i}\nu_{i}
\end{array} \right)+\frac{1}{C_{e}}\left(\begin{array}{cc}
m_{e}\nu_{e} \ &\ -\frac{e|B|}{c}\\[3mm]
\frac{e|B|}{c}\ &\ m_{e}\nu_{e}
\end{array} \right)\\
=&\left(\begin{array}{cc}
\frac{m_{i}\nu_{i}}{C_{i}}+\frac{m_{e}\nu_{e}}{C_{e}}\ &\ \frac{e|B|}{c}\left(\frac{1}{C_{i}}-\frac{1}{C_{e}}\right)\\[3mm]
-\frac{e|B|}{c}\left(\frac{1}{C_{i}}-\frac{1}{C_{e}}\right)\ &\
\frac{m_{i}\nu_{i}}{C_{i}}+\frac{m_{e}\nu_{e}}{C_{e}}
\end{array} \right)\\
=&\left(\begin{array}{cc}
\frac{(m_{i}\nu_{i}+m_{e}\nu_{e})\left(m_{i}\nu_{i}m_{e}\nu_{e}+\frac{e^{2}|B|^{2}}{c^{2}}\right)}{C_{i}C_{e}}\
&
\ \frac{e|B|}{c}\frac{(m_{i}\nu_{i}+m_{e}\nu_{e})(m_{e}\nu_{e}-m_{i}\nu_{i})}{C_{i}C_{e}}\\[3mm]
-\frac{e|B|}{c}\frac{(m_{i}\nu_{i}+m_{e}\nu_{e})(m_{e}\nu_{e}-m_{i}\nu_{i})}{C_{i}C_{e}}\
&\
\frac{(m_{i}\nu_{i}+m_{e}\nu_{e})\left(m_{i}\nu_{i}m_{e}\nu_{e}+\frac{e^{2}|B|^{2}}{c^{2}}\right)}{C_{i}C_{e}}
\end{array} \right)\\
=&\frac{m_{i}\nu_{i}+m_{e}\nu_{e}}{C_{i}C_{e}}
\left(\begin{array}{cc}
m_{i}\nu_{i}m_{e}\nu_{e}+\frac{e^{2}|B|^{2}}{c^{2}}\ &\ \frac{e|B|}{c}\left(m_{e}\nu_{e}-m_{i}\nu_{i}\right)\\[3mm]
-\frac{e|B|}{c}\left(m_{e}\nu_{e}-m_{i}\nu_{i}\right)\ &\
m_{i}\nu_{i}m_{e}\nu_{e}+\frac{e^{2}|B|^{2}}{c^{2}}
\end{array} \right).
\end{aligned}
\end{eqnarray*}
Hence,
$$
\det(K_{i}+K_{e})=\frac{(m_{i}\nu_{i}+m_{e}\nu_{e})^{2}}{C_{i}C_{e}},
$$
where we have used the identity
$$
C_{i}C_{e}=m_{i}^{2}\nu_{i}^{2}m_{e}^{2}\nu_{e}^{2}
+\frac{e^{2}|B|^{2}}{c^{2}}(m_{i}^{2}\nu_{i}^{2}+m_{e}^{2}\nu_{e}^{2})
+\left(\frac{e^{2}|B|^{2}}{c^{2}}\right)^{2}.
$$
It is therefore straightforward to see
\begin{eqnarray*}
\begin{aligned}
\left(K_{i}+K_{e}\right)^{-1}=\frac{1}{m_{i}\nu_{i}+m_{e}\nu_{e}}\left(\begin{array}{cc}
m_{i}\nu_{i}m_{e}\nu_{e}+\frac{e^{2}|B|^{2}}{c^{2}}\ &\ -\frac{e|B|}{c}\left(m_{e}\nu_{e}-m_{i}\nu_{i}\right)\\[3mm]
\frac{e|B|}{c}\left(m_{e}\nu_{e}-m_{i}\nu_{i}\right)\ &\
m_{i}\nu_{i}m_{e}\nu_{e}+\frac{e^{2}|B|^{2}}{c^{2}}
\end{array} \right).
\end{aligned}
\end{eqnarray*}
After strenuous computations, one can verify that
\begin{eqnarray*}
\begin{aligned}
K_{e}\left(K_{i}+K_{e}\right)^{-1}K_{i}=K_{i}\left(K_{i}+K_{e}\right)^{-1}K_{e}=\frac{1}{m_{i}\nu_{i}+m_{e}\nu_{e}}\left(\begin{array}{cc}
1\ &\ 0\\[3mm]
0\ &\ 1
\end{array} \right).
\end{aligned}
\end{eqnarray*}
Here we have omitted the proof of the above identity for brevity. Plugging this identity into \eqref{de.Gf} yields that the coefficient matrix $G$ in \eqref{udetermin2} is given by
\begin{eqnarray}\label{compG}
\begin{aligned}
G=\frac{1}{n}\left(\begin{array}{cc}
-\frac{T_{i}+T_{e}}{m_{i}\nu_{i}+m_{e}\nu_{e}}\ &\ 0\\[3mm]
0\ &\ -\frac{T_{i}+T_{e}}{m_{i}\nu_{i}+m_{e}\nu_{e}}
\end{array} \right).
\end{aligned}
\end{eqnarray}
This together with \eqref{u3E3} imply that
\begin{eqnarray}
\label{de.upro}
\begin{aligned} n\left(\begin{array}{c}
u_{1}\\[3mm]
u_{2}\\[3mm]
u_{3}
\end{array}\right)=  \left(\begin{array}{ccc}
-\frac{T_{i}+T_{e}}{m_{i}\nu_{i}+m_{e}\nu_{e}}\ &\ 0\ & 0\ \\[3mm]
0\ &\ -\frac{T_{i}+T_{e}}{m_{i}\nu_{i}+m_{e}\nu_{e}}\ & 0\ \\[3mm]
0\ &\ 0 \ & -\frac{T_{i}+T_{e}}{m_{i}\nu_{i}+m_{e}\nu_{e}}\ \\
\end{array} \right)\left(\begin{array}{c}
\partial_{1}n\\[3mm]
\partial_{2}n \\[3mm]
\partial_{3}n
\end{array} \right).
\end{aligned}
\end{eqnarray}
Therefore, using the first equation of \eqref{1.1} for the conservation of mass under the quasineutral assumption \eqref{de.qnc},  we obtain that $n$ satisfies the diffusion equation
\begin{equation}
\label{de.ndiff}
\pa_t  n -\frac{T_{i}+T_{e}}{m_{i}\nu_{i}+m_{e}\nu_{e}}\De n=0.
\end{equation}

It remains to determine the components of $E$ normal to $B$, namely $E_1$ and $E_2$. In fact,
%
%
\eqref{uErelation} implies that
\begin{eqnarray*}
-q_{\alpha}n_{\alpha}\left(\begin{array}{c}
E_{1}\\
E_{2}
\end{array}
\right)=A_{\alpha}\left(\begin{array}{c}
u_{1}\\
u_{2}
\end{array}
\right)-T_{\alpha}\left(\begin{array}{c}
\partial_{1}n\\
\partial_{2}n
\end{array} \right),
\end{eqnarray*}
for $\al=i$ and $e$.
Taking $\al=i$ for instance and then using \eqref{udetermin2}, \eqref{def.G} and \eqref{compG}, one has
\begin{eqnarray*}
\begin{aligned} \left(\begin{array}{c}
E_{1}\\
E_{2}
\end{array}
\right)=&-\frac{1}{en}\left(\begin{array}{cc}
-m_{i}\nu_{i}n\ &\ \frac{en}{c}|B|\\
-\frac{en}{c}|B|\ &\ -m_{i}\nu_{i}n
\end{array} \right)G \left(\begin{array}{c}
\partial_{1}n\\
\partial_{2}n
\end{array} \right)+\frac{T_{i}}{en}\left(\begin{array}{c}
\partial_{1}n\\
\partial_{2}n
\end{array} \right)\\
=&\left(-\frac{1}{en}\left(\begin{array}{cc}
-m_{i}\nu_{i}n\ &\ \frac{en}{c}|B|\\
-\frac{en}{c}|B|\ &\ -m_{i}\nu_{i}n
\end{array} \right)G +\frac{T_{i}}{en}\left(\begin{array}{cc}
1\ &\ 0\\[3mm]
0\ &\ 1
\end{array} \right)\right)\left(\begin{array}{c}
\partial_{1}n\\[3mm]
\partial_{2}n
\end{array} \right)\\
=&\frac{1}{en}\left(\begin{array}{cc}
 \frac{T_{i}m_{e}\nu_{e}-T_{e}m_{i}\nu_{i}}{m_{i}\nu_{i}+m_{e}\nu_{e}}\ &\
 \frac{e|B|}{c}\frac{T_{i}+T_{e}}{m_{i}\nu_{i}+m_{e}\nu_{e}}\\[3mm]
-\frac{e|B|}{c}\frac{T_{i}+T_{e}}{m_{i}\nu_{i}+m_{e}\nu_{e}}\ &\
\frac{T_{i}m_{e}\nu_{e}-T_{e}m_{i}\nu_{i}}{m_{i}\nu_{i}+m_{e}\nu_{e}}
\end{array} \right)\left(\begin{array}{c}
\partial_{1}n\\[3mm]
\partial_{2}n
\end{array} \right).
\end{aligned}
\end{eqnarray*}
This together with \eqref{u3E3} imply that
\begin{eqnarray}
\label{de.Epro}
\begin{aligned} n\left(\begin{array}{c}
E_{1}\\[3mm]
E_{2}\\[3mm]
E_{3}
\end{array}\right)= \frac{1}{e}\left(\begin{array}{ccc}
 \frac{T_{i}m_{e}\nu_{e}-T_{e}m_{i}\nu_{i}}{m_{i}\nu_{i}+m_{e}\nu_{e}}\ &\
 \frac{e|B|}{c}\frac{T_{i}+T_{e}}{m_{i}\nu_{i}+m_{e}\nu_{e}}\ & \ 0\\[3mm]
-\frac{e|B|}{c}\frac{T_{i}+T_{e}}{m_{i}\nu_{i}+m_{e}\nu_{e}}\ &\
\frac{T_{i}m_{e}\nu_{e}-T_{e}m_{i}\nu_{i}}{m_{i}\nu_{i}+m_{e}\nu_{e}}\
&\ 0\\[3mm]\\
0\ &\ 0\ &\
\frac{T_{i}m_{e}\nu_{e}-T_{e}m_{i}\nu_{i}}{m_{i}\nu_{i}+m_{e}\nu_{e}}
\end{array} \right)\left(\begin{array}{c}
\partial_{1}n\\[3mm]
\partial_{2}n \\[3mm]
\partial_{3}n
\end{array} \right).
\end{aligned}
\end{eqnarray}
We point out that in the coefficient matrix on the right of \eqref{de.Epro}, the diagonal entries are equal and independent of $B$,  and the non-diagonal entries are skew-symmetric and linear in $|B|$.

\subsection{Diffusion of the magnetic field}

Notice that the large-time asymptotic profiles of $u$ and $E$ given by \eqref{de.upro} and \eqref{de.Epro} are along the gravitational direction of the diffusive density $n$ determined by \eqref{de.ndiff}. Let $u_{\al,\perp}$ $(\al=i,e)$ and $E_\perp$ be the asymptotic profiles along the direction normal to the gravitation of the density. Then by taking the background densities as $n_i=n_e=1$, we expect that the large-time profile of the magnetic field $B$ is governed by the following system
 \begin{equation*}
\left\{
  \begin{aligned}
  &-e{E}_{\perp}+m_{i}\nu_{i}{u}_{i,\perp}=0,\\
  &e{E}_{\perp}+m_{e}\nu_{e}{u}_{e,\perp}=0,\\
  &-c\nabla\times {B}+4\pi  (e{u}_{i,\perp}-e{u}_{e,\perp})=0,\\
  &\partial_t {B}+c\nabla\times{ E}_{\perp}=0.\\
\end{aligned}\right.
\end{equation*}
It is easy to obtain that $B$ satisfies the diffusion equation
\begin{equation*}
\partial_{t} {B}-\dfrac{c^{2}m_{i}\nu_{i}m_{e}\nu_{e}}{4\pi e^{2}(m_{i}\nu_{i}+m_{e}\nu_{e})}\Delta B=0,
\end{equation*}
and  $u_{\al,\perp}$ $(\al=i,e)$ and $E_\perp$ are given by
\begin{equation*}
\left\{
  \begin{aligned}
  &{u}_{i,\perp}=\frac{e}{m_{i}\nu_{i}}{E}_{\perp}=\frac{c}{4\pi e}\dfrac{m_{e}\nu_{e}}{m_{i}\nu_{i}+m_{e}\nu_{e}}\nabla\times {B},\\
  &{u}_{e,\perp}=-\frac{e}{m_{e}\nu_{e}}{E}_{\perp}=-\frac{c}{4\pi e}\dfrac{ m_{i}\nu_{i}}{m_{i}\nu_{i}+m_{e}\nu_{e}}\nabla\times {B},\\
  &{E}_{\perp}=\frac{c}{4\pi e^{2}}\dfrac{m_{i}\nu_{i}m_{e}\nu_{e}}{m_{i}\nu_{i}+m_{e}\nu_{e}}\nabla\times {B}.\\
\end{aligned}\right.
\end{equation*}


\section{Decay property of linearized system}\label{sec3}

In this section, we study the time-decay property of solutions to the linearized  system basing on the  Fourier energy method. The result of this part is similar to the case of one-fluid in \cite{Duan}, and also similar to the study of two-species kinetic Vlasov-Maxwell-Boltzmann system in \cite{DS}. The main motivation to present this part is to understand the linear dissipative structure of such complex system as in \cite{UDK} in terms of the direct energy method and also provide a clue to the more delicate spectral analysis to be given later on.  Notice that the key estimate \eqref{s4.j} in this section will be used to deal with the time-decay property of solutions over the high-frequency domain in the next sections.

\subsection{Reformulation of the problem}

We assume that the steady state of the Euler-Maxwell system \eqref{1.1} is trivial, taking the form of
$$
n_\al=1,\ u_\al=0,\ E=B=0.
$$
Before constructing the more accurate large-time asymptotic profile around the trivial steady state, we first consider the linearized system around the above constant state. For that, let us set $\rho_{\alpha}=n_{\alpha}-1$ for $\al=i$ and $e$. Then
$U:=[\rho_{\alpha},u_{\alpha},E,B]$ satisfies
\begin{equation}\label{2.5}
\left\{
  \begin{aligned}
  &\partial_t \rho_\alpha+\nabla\cdot u_\alpha=g_{1\alpha},\\
  &m_{\alpha}\partial_t u_{\alpha}+T_{\alpha}\nabla \rho_{\alpha}-q_{\alpha}E+m_{\alpha}\nu_\alpha u_\alpha=g_{2\alpha},\\
  &\partial_t E-c\nabla\times B+4\pi\sum_{\alpha=i,e}q_{\alpha}u_{\alpha}=g_{3},\\
  &\partial_t B+c\nabla \times E=0,\\
  &\nabla \cdot E=4\pi\sum_{\alpha=i,e}q_{\alpha}\rho_{\alpha}, \ \  \nabla
\cdot B=0.
\end{aligned}\right.
\end{equation}
Initial data are given by
\begin{eqnarray}\label{NI}
[\rho_{\alpha},u_{\alpha},E,B]|_{t=0}=[n_{\alpha 0}-1,u_{\alpha
0},E_{0},B_0],
\end{eqnarray}
with the compatible condition
\begin{eqnarray}\label{NC}
\nabla \cdot E_0=4\pi\sum_{\alpha=i,e}q_{\alpha}\rho_{\alpha 0}, \ \
\nabla \cdot B_0=0.
\end{eqnarray}
Here the nonhomgeneous source terms are
\begin{equation}\label{sec5.ggg}
\arraycolsep=1.5pt \left\{
 \begin{aligned}
 &g_{1\alpha}=-\nabla\cdot(\rho_\alpha u_\alpha):=\nabla\cdot f_{\alpha},\\
 &g_{2\alpha}=- m_{\alpha}u_{\alpha} \cdot \nabla
u_\alpha-\left(\frac{p_{\alpha}'(\rho_{\alpha}+1)}{\rho_{\alpha}+1}-\frac{p_{\alpha}'(1)}{1}\right)\nabla
\rho_{\alpha}+q_{\alpha}\frac{u_{\alpha}}{c}\times B,\\
 &g_{3}=-4\pi\sum_{\alpha=i,e}q_{\alpha}\rho_{\alpha}u_{\alpha}.
\end{aligned}\right.
\end{equation}
Notice in the isothermal case that $p_{\alpha}'(n)=T_{\alpha} $ for any $n>0$.

\subsection{Linear decay structure}

In this section, for brevity of presentation we still use $U=[\rho_{\alpha},u_{\alpha},E,B]$ to denote
the solution to  the linearized homogeneous system
\begin{equation}\label{Linear}
\left\{
  \begin{aligned}
  &\partial_t \rho_\alpha+\nabla\cdot u_\alpha=0,\\
  &m_{\alpha}\partial_t u_{\alpha}+T_{\alpha}\nabla \rho_{\alpha}-q_{\alpha}E+m_{\alpha}\nu_\alpha u_\alpha=0,\\
  &\partial_t E-c\nabla\times B+4\pi\sum_{\alpha=i,e}q_{\alpha}u_{\alpha}=0,\\
  &\partial_t B+c\nabla \times E=0,\\
  &\nabla \cdot E=4\pi\sum_{\alpha=i,e}q_{\alpha}\rho_{\alpha}, \ \  \nabla
\cdot B=0,
\end{aligned}\right.
\end{equation}
with given initial data
\begin{eqnarray}\label{NLI}
[\rho_{\alpha},u_{\alpha},E,B]|_{t=0}=[\rho_{\al 0},u_{\alpha
0},E_{0},B_0],
\end{eqnarray}
satisfying the compatible condition
\begin{eqnarray}\label{LNC}
\nabla \cdot E_0=4\pi\sum_{\alpha=i,e}q_{\alpha}\rho_{\alpha 0}, \ \
\nabla \cdot B_0=0.
\end{eqnarray}

The goal of this section is to apply the Fourier energy method  to the Cauchy problem
\eqref{Linear}, \eqref{NLI}, \eqref{LNC} to show that there exists a
time-frequency Lyapunov functional which is equivalent with
$|\hat{U}(t,k)|^{2}$ and moreover its dissipation rate can also be
characterized by the functional itself.
Let us state the main result
of this section as follows.

\begin{theorem}\label{lyapunov} Let $U(t,x)$, $t>0$,
$x\in\mathbb{R}^{3}$, be a well-defined solution to the system
$\eqref{Linear}$-$\eqref{LNC}$. There is a time-frequency Lyapunov
functional $\mathcal{E}(\hat{U}(t,k))$ with
\begin{eqnarray}\label{equver}
\mathcal {E}(\hat{U}(t,k))\sim
|\hat{U}^{2}|:=\sum_{\alpha=i,e}|[\hat{\rho}_{\alpha},\hat{u}_{\alpha}]|^{2}
+|\hat{E}|^{2}+|\hat{B}|^{2},
\end{eqnarray}
such that, for some $\lambda>0$, the Lyapunov
inequality
\begin{eqnarray}\label{Lyainequlity1}
 \frac{d}{dt}\mathcal {E}(\hat{U}(t,k))+\dfrac{\lambda|k|^{2}}{(1+|k|^{2})^{2}}\mathcal
 {E}(\hat{U}(t,k))\leq 0
\end{eqnarray}
holds for any $ t>0$ and $k\in\mathbb{R}^{3}$.
\end{theorem}

\begin{proof}
As in \cite{Duan}, we use the following notations. For an integrable function $f:\mathbb{R}^{3} \rightarrow
\mathbb{R}$, its Fourier transform is defined by
\begin{eqnarray*}
   \hat{f}(k)=\int_{\mathbb{R}^{3}}{\exp({-ix\cdot k})}f(x)dx,\ \ \
   x\cdot k:=\sum_{j=1}^{3}x_{j}k_{j},\ \ k\in \mathbb{R}^{3},
\end{eqnarray*}
where $i=\sqrt{-1}\in\mathbb{C}$ is  the imaginary unit. For two
complex numbers or vectors $a$ and $b$, $(a|b)$ denotes the dot
product of $a$ with the complex conjugate of $b$.
Taking the Fourier
transform in $x$ for $\eqref{Linear}$,
$\hat{U}=[\hat{\rho}_{\alpha},\hat{u}_{\alpha},\hat{E},\hat{B}]$
satisfies
\begin{equation}\label{LinearF}
\left\{
  \begin{aligned}
  &\partial_t \hat{\rho}_{\alpha}+ik\cdot \hat{u}_{\alpha}=0,\\
  &m_{\alpha}\partial_t \hat{u}_{\alpha} +T_{\alpha} ik\hat{\rho}_{\alpha}-q_{\alpha}\hat{ E}+m_{\alpha}\nu_{\alpha}\hat{u}_{\alpha}=0,\\
  &\partial_t \hat{E}-cik\times\hat{ B}+4\pi\sum_{\alpha=i,e}q_{\alpha}\hat{u}_{\alpha}=0,\\
  &\partial_t \hat{B}+cik\times \hat{E}=0,\\
  &ik\cdot \hat{E}=4\pi\sum_{\alpha=i,e}q_{\alpha}\hat{\rho}_{\alpha}, \ \ ik \cdot\hat{ B}=0, \ \ \ t>0,\
  k\in\mathbb{R}^{3}.
\end{aligned}\right.
\end{equation}
First of all, it is straightforward to obtain from the first four
equations of $\eqref{LinearF}$ that
\begin{eqnarray}\label{LinearF1}
   \dfrac{1}{2}\dfrac{d}{dt}\sum_{\alpha=i,e}\left|\left[\sqrt{T_{\alpha}}\hat{\rho}_{\alpha},
   \sqrt{m_{\alpha}}\hat{u}_{\alpha}\right]\right|^{2}+\dfrac{1}{2}\dfrac{1}{4\pi}\dfrac{d}{dt}\left|\left[\hat{E},\hat{B}\right]\right|^{2}+\sum_{\alpha=i,e}m_{\alpha}\nu_{\alpha}|\hat{u}_{\alpha}|^{2}=0.
\end{eqnarray}
By taking the complex dot product of the second equation of
 $\eqref{LinearF}$ with $ik\hat{\rho}_{\alpha}$, replacing
 $\partial_{t}\hat{\rho}_{\alpha}$ by the first equation of
 $\eqref{LinearF}$, taking the real part, and taking summation for $\alpha=i,e$, one has
\begin{multline*}
   \partial_{t}\sum_{\alpha=i,e}\mathfrak{R}(m_{\alpha}\hat{u}_{\alpha}|ik\hat{\rho}_{\alpha})
   +\sum_{\alpha=i,e}T_{\alpha}|k|^{2}|\hat{\rho}_{\alpha}|^{2}+4\pi\left|\sum_{\alpha=i,e}q_{\alpha}\hat{\rho}_{\alpha}\right|^{2}\\
   =\sum_{\alpha=i,e}m_{\alpha}|k\cdot
   \hat{u}_{\alpha}|^{2}-\sum_{\alpha=i,e}m_{\alpha}\nu_{\alpha}\mathfrak{R}(\hat{u}_{\alpha}|ik\hat{\rho}_{\alpha}),
\end{multline*}
which by using the Cauchy-Schwarz inequality, implies
\begin{eqnarray*}
\partial_{t}\sum_{\alpha=i,e}\mathfrak{R}(m_{\alpha}\hat{u}_{\alpha}|ik\hat{\rho}_{\alpha})
   +\lambda\sum_{\alpha=i,e}|k|^{2}|\hat{\rho}_{\alpha}|^{2}+4\pi\left|\sum_{\alpha=i,e}q_{\alpha}\hat{\rho}_{\alpha}\right|^{2}
  \leq C(1+|k|^{2})\sum_{\alpha=i,e}|\hat{u}_{\alpha}|^{2}.
\end{eqnarray*}
Dividing it by $1+|k|^{2}$ gives
\begin{equation}\label{LinearF2}
\partial_{t}\dfrac{\sum\limits_{\alpha=i,e}\mathfrak{R}(m_{\alpha}\hat{u}_{\alpha}|ik\hat{\rho}_{\alpha})}{1+|k|^{2}}
  +\lambda\frac{|k|^{2}}{1+|k|^{2}}\sum_{\alpha=i,e}|\hat{\rho}_{\alpha}|^{2}
  +\frac{4\pi}{1+|k|^{2}}\left|\sum_{\alpha=i,e}q_{\alpha}\hat{\rho}_{\alpha}\right|^{2}\leq C
  \sum_{\alpha=i,e}|\hat{u}_{\alpha}|^{2}.
\end{equation}
In a similar way, by taking the complex dot product of the second
equation of $\eqref{LinearF}$ with $-4\pi
q_{\alpha}\hat{E}/T_\al$, replacing
 $\partial_{t}\hat{E}$ by the third equation of
 $\eqref{LinearF}$, and taking summation for $\alpha=i,e$,  one has
\begin{multline}\label{Linear3}
  -\partial_{t}\sum_{\alpha=i,e}\dfrac{4\pi m_{\alpha}q_{\alpha}}{T_{\alpha}}(\hat{u}_{\alpha}|\hat{E})+|k\cdot\hat{E}|^{2}
  +\sum_{\alpha=i,e}\dfrac{4\pi q_{\alpha}^{2}}{T_{\alpha}}|\hat{E}|^{2}
=\sum_{\alpha=i,e}\dfrac{4\pi
m_{\alpha}q_{\alpha}}{T_{\alpha}}(\hat{u}_{\alpha}|-c
ik\times\hat{B})\\
   +\sum_{\alpha=i,e}\dfrac{4\pi m_{\alpha}q_{\alpha}}{T_{\alpha}}\left(\hat{u}_{\alpha}|4\pi \sum_{\alpha=i,e}q_{\alpha}\hat{u}_{\alpha}\right)
   +\sum_{\alpha=i,e}\dfrac{4\pi q_{\alpha}}{T_{\alpha}}(m_{\alpha}\nu_{\alpha}\hat{u}_{\alpha}|\hat{E}),
\end{multline}
where we have used $ik\cdot
\hat{E}=4\pi\sum\limits_{\alpha=i,e}q_{\alpha}\hat{\rho}_{\alpha}$.
 Taking the
real part of $\eqref{Linear3}$ and using the Cauchy-Schwarz
inequality imply
\begin{multline*}
  -\partial_{t}\sum_{\alpha=i,e}\dfrac{4\pi m_{\alpha}q_{\alpha}}{T_{\alpha}}\mathfrak{R}(\hat{u}_{\alpha}|\hat{E})+|k\cdot\hat{E}|^{2} +\sum_{\alpha=i,e}\dfrac{4\pi q_{\alpha}^{2}}{2T_{\alpha}}|\hat{E}|^{2}\\
\leq \sum_{\alpha=i,e}\dfrac{4\pi
m_{\alpha}q_{\alpha}}{T_{\alpha}}\mathfrak{R}(\hat{u}_{\alpha}|-c
ik\times\hat{B})+C\sum_{\alpha=i,e}|\hat{u}_{\alpha}|^{2},
\end{multline*}
which further multiplying it by $|k|^{2}/(1+|k|^{2})^{2}$ gives
\begin{multline}\label{Linear4}
  -\partial_{t}\sum_{\alpha=i,e}\dfrac{4\pi m_{\alpha}q_{\alpha}}{T_{\alpha}}\dfrac{|k|^{2}\mathfrak{R}(\hat{u}_{\alpha}|\hat{E})}{(1+|k|^{2})^{2}}
  +\dfrac{|k|^{2}|k\cdot\hat{E}|^{2}}{(1+|k|^{2})^{2}}
  +\sum\limits_{\alpha=i,e}\dfrac{4\pi q_{\alpha}^{2}}{2T_{\alpha}}\dfrac{|k|^{2}|\hat{E}|^{2}}{(1+|k|^{2})^{2}}\\
\leq \sum_{\alpha=i,e}\dfrac{4\pi
m_{\alpha}q_{\alpha}}{T_{\alpha}}\dfrac{|k|^{2}\mathfrak{R}(\hat{u}_{\alpha}|-cik\times\hat{B})}{(1+|k|^{2})^{2}}
   +C\sum_{\alpha=i,e}|\hat{u}_{\alpha}|^{2}.
\end{multline}
Similarly, it follows from equations of the electromagnetic field in
 $\eqref{LinearF}$ that
\begin{eqnarray*}
  && \partial_{t}(\hat{E}|-ik\times\hat{B})+c|k\times\hat{B}|^{2}=c|k\times\hat{E}|^{2}
  +4\pi\sum_{\alpha=i,e}(q_{\alpha}\hat{u}_{\alpha}|ik\times\hat{B}),
\end{eqnarray*}
which after using Cauchy-Schwarz and dividing it by
$(1+|k|^{2})^{2}$, implies
\begin{eqnarray}\label{Linear5}
  && \partial_{t}\dfrac{\mathfrak{R}(\hat{E}|-ik\times\hat{B})}{(1+|k|^{2})^{2}}
  +\dfrac{\lambda|k\times\hat{B}|^{2}}{(1+|k|^{2})^{2}}
  \leq \dfrac{c|k|^{2}|\hat{E}|^{2}}{(1+|k|^{2})^{2}}
  +C\sum_{\alpha=i,e}|\hat{u}_{\alpha}|^{2}.
\end{eqnarray}
 Finally, let's define
 \begin{equation*}
\arraycolsep=1.5pt
\begin{array}{rcl}
   \mathcal{E}(\hat{U}(t,k))&=&\sum\limits_{\alpha=i,e}\left|\left[\sqrt{T_{\alpha}}\hat{\rho}_{\alpha},
   \sqrt{m_{\alpha}}\hat{u}_{\alpha}\right]\right|^{2}+|[\hat{E},\hat{B}]|^{2}+
   \kappa_{1}\dfrac{\sum\limits_{\alpha=i,e}\mathfrak{R}( m_{\alpha}\hat{u}_{\alpha}|ik\hat{\rho}_{\alpha})}{1+|k|^{2}}\\[3mm]
   &&-\kappa_{1}\sum\limits_{\alpha=i,e}\dfrac{4\pi m_{\alpha}q_{\alpha}}{T_{\alpha}}\dfrac{|k|^{2}\mathfrak{R}(\hat{u}_{\alpha}|\hat{E})}{(1+|k|^{2})^{2}}
   +\kappa_{1}\kappa_{2}\dfrac{\mathfrak{R}(\hat{E}|-ik\times\hat{B})}{(1+|k|^{2})^{2}},
\end{array}
\end{equation*}
for constants $0<\kappa_{2},\kappa_{1}\ll 1$ to be determined.
Notice that as long as $ 0<\kappa_{i}\ll 1$ is small enough for
$i=1,2$, then $\mathcal{E}(\hat{U}(t,k))\sim |\hat{U}(t)|^{2}$ holds
true and \eqref{equver} is proved. The sum of $\eqref{LinearF1}$, $\eqref{LinearF2}\times
\kappa_{1}$, $\eqref{Linear4}\times\kappa_{1}$ and
$\eqref{Linear5}\times \kappa_{1}\kappa_{2}$ gives
\begin{eqnarray}\label{DJF6}
&&\partial_{t}\mathcal{E}(\hat{U}(t,k))+\lambda\sum\limits_{\alpha=i,e}|\hat{u}_{\alpha}|^{2}
   +\lambda\dfrac{|k|^{2}}{1+|k|^{2}}\sum\limits_{\alpha=i,e}|\hat{\rho}_{\alpha}|
   +\dfrac{\lambda|k|^{2}}{(1+|k|^{2})^{2}}|[\hat{E},\hat{B}]|^{2}
   \leq 0,
\end{eqnarray}
where we have used the identity $|k\times
\hat{B}|^{2}=|k|^{2}|\hat{B}|^{2} $ due to $ k\cdot \hat{B}=0$ and
also used the following Cauchy-Schwarz inequality
\begin{eqnarray*}
\kappa_{1}\sum_{\alpha=i,e}\dfrac{4\pi
m_{\alpha}q_{\alpha}}{T_{\alpha}}\dfrac{|k|^{2}\mathfrak{R}(\hat{u}_{\alpha}|-ik\times\hat{B})}{(1+|k|^{2})^{2}}\leq
\sum_{\alpha=i,e}\dfrac{4\pi
m_{\alpha}q_{\alpha}}{T_{\alpha}}\dfrac{\kappa_{1}
|k|^{4}|\hat{u}_{\alpha}|^{2}}{4\epsilon\kappa_{2}(1+|k|^{2})^{2}}
+\sum_{\alpha=i,e}\dfrac{4\pi
m_{\alpha}q_{\alpha}}{T_{\alpha}}\dfrac{\epsilon\kappa_{1}\kappa_{2}|k|^{2}|\hat{B}|^{2}}{(1+|k|^{2})^{2}}.
\end{eqnarray*}
First, we chose $\epsilon>0$ such that
$$
\epsilon4\pi
\sum\limits_{\alpha=i,e}\dfrac{m_{\alpha}q_{\alpha}}{T_{\alpha}}\leq
\lambda
$$
for $\lambda$ appearing on the left of \eqref{Linear5}, and then let  $\kappa_{2}>0$ be fixed and let
$\kappa_{1}>0$ be further chosen small enough. Therefore,
$\eqref{Lyainequlity1}$ follows from $\eqref{DJF6}$ by noticing
\begin{eqnarray*}
\sum\limits_{\alpha=i,e}|\hat{u}_{\alpha}|^{2}
   +\dfrac{\lambda|k|^{2}}{1+|k|^{2}}\sum\limits_{\alpha=i,e}|\rho_{\alpha}|^2
   +\dfrac{\lambda|k|^{2}}{(1+|k|^{2})^{2}}|[\hat{E},\hat{B}]|^{2}
   \geq \dfrac{\lambda|k|^{2}}{(1+|k|^{2})^{2}}|\hat{U}|^{2}.
\end{eqnarray*}
This completes the proof of Theorem \ref{lyapunov}.
\end{proof}

Theorem \ref{lyapunov} directly leads to the pointwise
time-frequency estimate on the modular $|\hat{U}(t,k)|$ in terms of
initial data modular $|\hat{U}_{0}(k)|$, which is the same as \cite[Corollary 4.1]{Duan}.

\begin{corollary}
Let $U(t,x)$, $t\geq0$, $x\in\mathbb{R}^{3}$ be a well-defined
solution to the system $\eqref{Linear}$-$\eqref{LNC}$. Then, there
are $\lambda>0$, $ C>0$  such that
\begin{eqnarray}\label{s4.j}
   |\hat{U}(t,k)|\leq C {\exp\left({-\tfrac{\lambda|k|^{2}t}{(1+|k|^{2})^{2}}}\right)}|\hat{U_{0}}(k)|
\end{eqnarray}
holds for any $t\geq 0$ and $k\in \mathbb{R}^{3}$.
\end{corollary}

 Based on the pointwise
time-frequency estimate $\eqref{s4.j}$,  it is also straightforward to obtain the
$L^{p}$-$L^{q}$ time-decay property to the Cauchy problem
$\eqref{Linear}$-$\eqref{NLI}$. Formally, the solution to
 the Cauchy problem $\eqref{Linear}$-$\eqref{NLI}$ is denoted by
\begin{eqnarray*}
\begin{aligned}
U(t)=\left[\rho_{\alpha},u_{\alpha},E,B\right]={e^{tL}}U_{0}, 
\end{aligned}
\end{eqnarray*}
where ${e^{tL}} $ for $t\geq 0$ is said to be the linearized
solution operator corresponding to the linearized Euler-Maxwell
system.

\begin{corollary}[see \cite{Duan} for instance]
Let $1\leq p,r\leq 2\leq q\leq \infty$, $\ell\geq 0$ and let $m\geq
1$ be an integer. Define
\begin{equation}\label{thm.decay.1}
    \left[\ell+3\left(\frac{1}{r}-\frac{1}{q}\right)\right]_{+}
    =\left\{\begin{array}{ll}
      \ell, & \  \text{if $\ell$ is integer and $r=q=2$},\\[3mm]
      {[}\ell+3(\frac{1}{r}-\frac{1}{q}){]}_{-}+1,  &\ \text{otherwise},
     \end{array}\right.
\end{equation}
where $[\cdot]_{-}$ denotes the integer part of  the argument.
Suppose $U_{0}$ satisfying $\eqref{LNC}$. Then $e^{tL}$ satisfies
the following time-decay property:
\begin{eqnarray*}
&&\begin{aligned}
 \|\nabla^{m}e^{Lt} U_0\|_{L^q}\leq C
 (1+t)^{-\tfrac{3}{2}(\tfrac{1}{p}-\tfrac{1}{q})-\tfrac{m}{2}}
 &\|U_0\|_{L^p}\\
 + &C(1+t)^{-\tfrac{\ell}{2}}\|\nabla^{m+[\ell+3(\tfrac{1}{r}-\tfrac{1}{q})]_+}{U}_0\|_{L^r}
\end{aligned}
\end{eqnarray*}
for any $t\geq 0$, where  $C=C(m,p,r,q,\ell)$.
\end{corollary}

\section{Spectral representation} \label{sec4}

In order to study the more accurate large-time asymptotic profile, we need to carry out the spectral analysis of the linearized system.

\subsection{Preparations}
As in \cite{Duan},
the linearized system \eqref{Linear} can be written as two decoupled subsystems which
govern the time evolution of $\rho_{\alpha}$, $\nabla\cdot
u_{\alpha}$, $\nabla \cdot E$ and $\nabla\times u_{\alpha}$,
$\nabla\times E$ and $\nabla \times B$ respectively. We decompose
the solution to \eqref{Linear}-\eqref{LNC} into two parts in the form of
\begin{eqnarray}\label{decompo}
&& \left[
  \begin{array}{c}
   \rho_{\alpha}(t,x)\\
   u_{\alpha}(t,x)\\
   E(t,x)\\
   B(t,x)\\
 \end{array}\right]
 =
 \left[
  \begin{array}{c}
   \rho_{\alpha}(t,x)\\
   u_{\alpha,\parallel}(t,x)\\
   E_{\parallel}(t,x)\\
   0\\
 \end{array}\right]+
\left[
  \begin{array}{c}
   0 \\
   u_{\alpha,\perp}(t,x)\\
   E_{\perp}(t,x)\\
   B(t,x)\\
 \end{array}\right],
\end{eqnarray}
where $u_{\alpha,\parallel}$, $u_{\alpha,\perp}$ are defined by
\begin{equation*}
u_{\alpha,\parallel}=-(-\Delta)^{-1}\nabla\nabla\cdot u_{\alpha},\ \
\ u_{\alpha,\perp}=(-\Delta)^{-1}\nabla\times(\nabla\times
u_{\alpha}),
\end{equation*}
and likewise for $E_{\parallel}$, $E_{\perp}$. For brevity, the first part on the right of \eqref{decompo} is called the fluid part and the second part is called the electromagnetic part, and we also write
\begin{equation*}
U_\parallel  =  [\rho_i,\rho_e,u_{i,\parallel},u_{e,\parallel}], \quad
U_\perp  = [u_{i,\perp},u_{e,\perp},E_{\perp},B].
\end{equation*}
Notice that to the end, $E_{\parallel}$ is always given by
$$
E_{\parallel}=4\pi e \De^{-1}\nabla (\rho_i-\rho_e).
$$

We now derive the equations of $U_{\parallel}$ and $U_\perp$ and their asymptotic equations that one may expect in the large time. Taking the divergence
of the second equation of \eqref{Linear}, it follows that
\begin{equation}\label{Fluid1}
\left\{
  \begin{aligned}
  &\partial_t \rho_{\alpha}+\nabla\cdot u_{\alpha}=0,\\
  &m_{\alpha}\partial_t (\nabla\cdot u_{\alpha})-q_{\alpha}\nabla\cdot E+T_{\alpha}\Delta\rho_{\alpha}+m_{\alpha}\nu_{\alpha}\nabla\cdot u_{\alpha}=0.\\
\end{aligned}\right.
\end{equation}
Applying $\Delta^{-1}\nabla$ to  the second equation of $\eqref{Fluid1}$ and noticing
$
\nabla\cdot u_{\alpha}
=\nabla\cdot
u_{\alpha,\parallel}$,
we see that the fluid part $U_\parallel$
satisfies
\begin{equation}\label{Fluid2}
\left\{
  \begin{aligned}
  &\partial_t \rho_{\alpha}+\nabla\cdot u_{\alpha,\parallel}=0,\\
  &m_{\alpha}\partial_t u_{\alpha,\parallel}-q_{\alpha}E_{\parallel}+T_{\alpha}\nabla\rho_{\alpha}+m_{\alpha}\nu_{\alpha}u_{\alpha,\parallel}=0.\\
\end{aligned}\right.
\end{equation}
Initial data are given by
\begin{equation}\label{Fluid2I}
[\rho_{\alpha},\ u_{\alpha,\parallel}]|_{t=0}=[\rho_{\alpha0},\ u_{\alpha0,\parallel}].
\end{equation}
As seen later on and also in the sense of the Darcy's law, the expected asymptotic profile of the fluid part satisfies
\begin{equation*}
\left\{
  \begin{aligned}
   &\partial_{t}\bar{\rho}+\nabla \cdot \bar{u}_{\parallel}=0,\\
   & T_{\alpha}\nabla \bar{\rho}-q_{\alpha}\bar{E}_{\parallel}+m_{\alpha}\nu_{\alpha}\bar{u}_{\parallel}=0,\\
  \end{aligned}\right.
\end{equation*}
with initial data
\begin{equation*}
\bar{\rho}|_{t=0}=\bar{\rho}_0=\frac{m_{i}\nu_{i}}{m_{i}\nu_{i}+m_{e}\nu_{e}}\rho_{i0}+\frac{m_{e}\nu_{e}}{m_{i}\nu_{i}+m_{e}\nu_{e}}\rho_{e0}.
\end{equation*}
Therefore, $\bar{\rho}$, $\bar{u}_\parallel$ and $\bar{E}_\parallel$ are determined
according to the following equations 
\begin{eqnarray}\label{sol.rhob}
\left\{
\begin{aligned}
&\partial_{t}\bar{\rho}-\dfrac{T_{i}+T_{e}}{m_{i}\nu_{i}+m_{e}\nu_{e}}\Delta
\bar{\rho}=0,\\
&\bar{u}_{\parallel}=-\dfrac{T_{i}+T_{e}}{m_{i}\nu_{i}+m_{e}\nu_{e}}\nabla
\bar{\rho},\\
&\bar{E}_{\parallel}=\dfrac{T_{i}m_{e}\nu_{e}-T_{e}m_{i}\nu_{i}}{e(m_{i}\nu_{i}+m_{e}\nu_{e})}\nabla
\bar{\rho},
\end{aligned}\right.
 \end{eqnarray}
 where initial data $\bar{u}_{0,\parallel}$ and $\bar{E}_{0,\parallel}$ of $\bar{u}_{\parallel}$ and $\bar{E}_{\parallel}$ are determined by $\bar{\rho}_0$ in terms of the last two equations of \eqref{sol.rhob}, respectively.
For later use, let us define $\bar{P}^{1}(ik)$, $\bar{P}^{2}(ik)$, $\bar{P}^{3}(ik)$  to be three row vectors in $\R^8$  by
\begin{eqnarray*}
&&\bar{P}^{1}(ik)=:\left[
\frac{m_{i}\nu_{i}}{m_{i}\nu_{i}+m_{e}\nu_{e}},\ \frac{m_{e}\nu_{e}}{m_{i}\nu_{i}+m_{e}\nu_{e}},\ 0,\  0\right],\\
&&\bar{P}^{2}(ik)=:\left[-\frac{T_{i}+T_{e}}{m_{i}\nu_{i}+m_{e}\nu_{e}}\frac{m_{i}\nu_{i}}{m_{i}\nu_{i}+m_{e}\nu_{e}}ik,\
-\frac{T_{i}+T_{e}}{m_{i}\nu_{i}+m_{e}\nu_{e}}\frac{m_{e}\nu_{e}}{m_{i}\nu_{i}+m_{e}\nu_{e}}ik,\
0,0\right],\\
&&\bar{P}^{3}(ik)=:\left[
\dfrac{T_{i}m_{e}\nu_{e}-T_{e}m_{i}\nu_{i}}{e(m_{i}\nu_{i}+m_{e}\nu_{e})}\frac{m_{i}\nu_{i}}{m_{i}\nu_{i}+m_{e}\nu_{e}}ik,
\ \dfrac{T_{i}m_{e}\nu_{e}-T_{e}m_{i}\nu_{i}}{e(m_{i}\nu_{i}+m_{e}\nu_{e})}\frac{m_{e}\nu_{e}}{m_{i}\nu_{i}+m_{e}\nu_{e}}ik,\ 0,\  0\right].\\
\end{eqnarray*}
Then the large-time asymptotic profile can be expressed in terms of the Fourier transform by
\begin{eqnarray}
&&\label{rhobar}\hat{\bar{\rho}}=\exp{\left(-\frac{T_{i}+T_{e}}{m_{i}\nu_{i}+m_{e}\nu_{e}}|k|^{2}t\right)}\bar{P}^{1}(ik)\hat{U}_{\parallel0}^{T},\\
&&\label{ubar}\hat{\bar{u}}_{\parallel}=\exp{\left(-\frac{T_{i}+T_{e}}{m_{i}\nu_{i}+m_{e}\nu_{e}}|k|^{2}t\right)}\bar{P}^{2}(ik)\hat{U}_{\parallel0}^{T},\\
&&\label{Ebar}\hat{\bar{E}}_{\parallel}=\exp{\left(-\frac{T_{i}+T_{e}}{m_{i}\nu_{i}+m_{e}\nu_{e}}|k|^{2}t\right)}\bar{P}^{3}(ik)\hat{U}_{\parallel0}^{T}.
\end{eqnarray}

The electromagnetic part satisfies the following equations:
\begin{equation}\label{electromagnetic}
\left\{
  \begin{aligned}
  &m_{i}\partial_t u_{i,\perp}-eE_{\perp}+m_{i}\nu_{i}u_{i,\perp}=0,\\
  &m_{e}\partial_t u_{e,\perp}+eE_{\perp}+m_{e}\nu_{e}u_{e,\perp}=0,\\
  &\partial_t E_{\perp}-c\nabla\times B+4\pi(eu_{i,\perp}-eu_{e,\perp})=0,\\
  &\partial_t B+c\nabla\times E_{\perp}=0,\\
\end{aligned}\right.
\end{equation}
with initial data
\begin{equation}\label{electromagnetic2I}
[u_{\alpha,\perp},\ E_{\perp},B]|_{t=0}=[u_{\alpha0,\perp},\
E_{0,\perp},B_{0}].
\end{equation}
The expected large-time asymptotic profile for the electromagnetic part is determined by the following equations in the sense of Darcy's law again:
\begin{equation}\label{perpprofile}
\left\{
  \begin{aligned}
  &-e\bar{E}_{\perp}+m_{i}\nu_{i}\bar{u}_{i,\perp}=0,\\
  &e\bar{E}_{\perp}+m_{e}\nu_{e}\bar{u}_{e,\perp}=0,\\
  &-c\nabla\times \bar{B}+4\pi(e\bar{u}_{i,\perp}-e\bar{u}_{e,\perp})=0,\\
  &\partial_t \bar{B}+c\nabla\times\bar{ E}_{\perp}=0.\\
\end{aligned}\right.
\end{equation}
As before, it is straightforward to obtain that
\begin{equation}\label{perpsolve}
\left\{
  \begin{aligned}
  &\partial_{t} \bar{B}-\dfrac{c^{2}m_{i}\nu_{i}m_{e}\nu_{e}}{4\pi e^{2}(m_{i}\nu_{i}+m_{e}\nu_{e})}\Delta \bar{B}=0,\\
  &\bar{u}_{i,\perp}=\frac{e}{m_{i}\nu_{i}}\bar{E}_{\perp}=\frac{c}{4\pi e}\dfrac{m_{e}\nu_{e}}{m_{i}\nu_{i}+m_{e}\nu_{e}}\nabla\times \bar{B},\\
  &\bar{u}_{e,\perp}=-\frac{e}{m_{e}\nu_{e}}\bar{E}_{\perp}=-\frac{c}{4\pi e}\dfrac{ m_{i}\nu_{i}}{m_{i}\nu_{i}+m_{e}\nu_{e}}\nabla\times \bar{B},\\
  &\bar{E}_{\perp}=\frac{c}{4\pi e^{2}}\dfrac{m_{i}\nu_{i}m_{e}\nu_{e}}{m_{i}\nu_{i}+m_{e}\nu_{e}}\nabla\times \bar{B}.\\
\end{aligned}\right.
\end{equation}
with initial data
\begin{equation*}
\bar{B}|_{t=0}=B_{0},
\end{equation*}
where initial data $\bar{u}_{i0,\perp},\ \bar{u}_{e0,\perp},\ \bar{E}_{0,\perp}$
are given from $\bar{B}_{0}$ according to the last three equations of
\eqref{perpsolve}. Notice that the asymptotic profile $ \bar{B}$ of the magnetic field can be expressed in term of the Fourier transform by
\begin{equation}\label{Bsolve}
\hat{\bar{B}}(t,k)=\exp\left(-\dfrac{c^{2}m_{i}\nu_{i}m_{e}\nu_{e}}{4\pi
e^{2}(m_{i}\nu_{i}+m_{e}\nu_{e})}|k|^{2}t\right)\hat{B}_{0}(k).
\end{equation}

\subsection{Spectral representation for fluid part}

\subsubsection{Asymptotic expansions and expressions}

After taking the Fourier transformation in $x$ for \eqref{Fluid2},
replacing $\hat{E}_{\parallel}$ by
$-4\pi\frac{ik}{|k|^{2}}(e\hat{\rho}_{i}-e\hat{\rho}_{e})$, the fluid part
$\hat{U}_{\parallel}=[\hat{\rho}_{i},\hat{\rho}_{e},\hat{u}_{i,\parallel},\hat{u}_{e,\parallel}]$
satisfies the following system of 1st-order ODEs
\begin{equation}\label{FluidF}
\left\{
  \begin{aligned}
  &\partial_t \hat{\rho}_{i}+ik\cdot \hat{u}_{i,\parallel}=0,\\
  &\partial_t \hat{\rho}_{e}+ik\cdot \hat{u}_{e,\parallel}=0,\\
  &\partial_t \hat{u}_{i,\parallel}+\frac{T_{i}}{m_{i}}ik\hat{\rho}_{i}
  +\frac{4\pi e^{2}}{m_{i}}\frac{ik}{|k|^{2}}\hat{\rho}_{i}-\frac{4\pi e^{2}}{m_{i}}\frac{ik}{|k|^{2}}\hat{\rho}_{e}+\nu_{i} \hat{u}_{i,\parallel}=0,\\
  &\partial_t \hat{u}_{e,\parallel}+\frac{T_{e}}{m_{e}}ik\hat{\rho}_{e}
  +\frac{4\pi e^{2}}{m_{e}}\frac{ik}{|k|^{2}}\hat{\rho}_{e}-\frac{4\pi e^{2}}{m_{e}}\frac{ik}{|k|^{2}}\hat{\rho}_{i}+\nu_{e} \hat{u}_{e,\parallel}=0.\\
\end{aligned}\right.
\end{equation}
Initial data are given as
\begin{equation}\label{FluidFI}
\hat{U}_{\parallel}(t,k)|_{t=0}=\hat{U}_{\parallel
0}(k)=:[\hat{\rho}_{i0},\ \hat{\rho}_{e0},\ \tilde{k}\tilde{k}\cdot
\hat{u}_{i0},\ \tilde{k}\tilde{k}\cdot \hat{u}_{e0}].
\end{equation}
Then the solution to \eqref{FluidF}, \eqref{FluidFI} can be written as
\begin{equation*}
\hat{U}_{\parallel}(t,k)^T={e^{A(ik)t}}\hat{U}_{\parallel
0}(k)^{T},
\end{equation*}
with the matrix $A(ik)$ defined by
\begin{equation*}
A(ik)=:\left(\begin{array} {cccc}
0 & 0 & -\zeta &\ \ \ 0\\
0 & 0 &  0 &\ \ \  -\zeta\\
-\frac{T_{i}}{m_{i}}\zeta-\frac{4\pi e^{2}}{m_{i}}\frac{\zeta}{|\zeta|^{2}} & \frac{4\pi e^{2}}{m_{i}}\frac{\zeta}{|\zeta|^{2}} & -\nu_{i}  &\ \ \ 0\\
\frac{4\pi e^{2}}{m_{e}}\frac{\zeta}{|\zeta|^{2}}& -\frac{T_{e}}{m_{e}}\zeta-\frac{4\pi e^{2}}{m_{e}}\frac{\zeta}{|\zeta|^{2}}  & 0 &-\nu_{e}\\
\end{array} \right),
\end{equation*}
where we have denoted $\zeta=ik$ on the right. In the sequel, for brevity, with a little abuse of notation, for a positive integer $\ell$, we also use $\zeta^\ell$ to denote $|\zeta|^{\ell-1}\zeta$ if $\ell$ is odd, and $|\zeta|^\ell$ if $\ell$ is even. 


By direct computation, we see that the characteristic polynomial of
$A(\zeta)$ is
\begin{eqnarray}
\det(\lambda
I-A(ik)) &=&\lambda^{4}+(\nu_{i}+\nu_{e})\lambda^{3}+\left(\nu_{i}\nu_{e}-\left(\frac{T_{i}}{m_{i}}+\frac{T_{e}}{m_{e}}\right)\zeta^{2}
+4\pi \left(\frac{e^{2}}{m_{i}}+\frac{e^{2}}{m_{e}}\right)\right)\lambda^{2}\notag\\
&&+\left(-\left(\frac{T_{i}}{m_{i}}\nu_{e}+\frac{T_{e}}{m_{e}}\nu_{i}\right)\zeta^{2}
+4\pi
\left(\frac{e^2}{m_{i}}\nu_{e}+\frac{e^{2}}{m_{e}}\nu_{i}\right)\right)\lambda\notag\\
&&+\left(\frac{T_{i}T_{e}}{m_{i}m_{e}}\zeta^{4}
-4\pi \frac{e^2(T_{i}+T_{e})}{m_{i}m_{e}}\zeta^{2}\right).\label{eigenpo}
\end{eqnarray}
It follows from \eqref{eigenpo} that
\begin{eqnarray*}
   &&\sum_{i=1}^{4}\lambda_{j}=-(\nu_{i}+\nu_{e}),
   \\
   &&
  \sum_{1\leq i\neq j\leq 4}\lam_i\lam_j =\nu_{i}\nu_{e}-\left(\frac{T_{i}}{m_{i}}+\frac{T_{e}}{m_{e}}\right)\zeta^{2}+4\pi\left(\frac{e^{2}}{m_{i}}+\frac{e^{2}}{m_{e}}\right),\\
 && \prod_{i=1}^{4}\lambda_{j}=\frac{T_{i}T_{e}}{m_{i}m_{e}}\zeta^{4}
-\frac{4\pi e^2(T_{i}+T_{e})}{m_{i}m_{e}}\zeta^{2}.
\end{eqnarray*}

First, we analyze the roots of the above characteristic equation
\eqref{eigenpo} and their asymptotic properties as $|\zeta|\rightarrow 0$. The
perturbation theory (see \cite{Ho} or \cite{Kato1}) for one-parameter family of
matrix $A(\zeta)$ for $|\zeta|\rightarrow 0$ implies that
$\lambda_{j}(\zeta)$ has the following asymptotic expansions:
\begin{equation*}
\lambda_{j}(\zeta)=\sum_{\ell=0}^{+\infty}\lambda_{j}^{(\ell)}|\zeta|^{\ell}.
\end{equation*}
Notice that $\lambda_{j}^{(0)}$ are the roots of the following
equation:
\begin{equation*}
\lambda g(\lambda)
=0,
\end{equation*}
with
\begin{equation*}
g(\lambda)=\lambda^{3}+(\nu_{i}+\nu_{e})\lambda^{2}+\left(\nu_{i}\nu_{e}
+4\pi\left(\frac{e^{2}}{m_{i}}+\frac{e^{2}}{m_{e}}\right)\right)\lambda+4\pi\left(\frac{e^2}{m_{i}}\nu_{e}+\frac{e^{2}}{m_{e}}\nu_{i}\right).
\end{equation*}
For later use we also set
\begin{equation}\label{deglamda}
\begin{aligned}
g(\lambda)
=\lambda^{3}+c_{2}\lambda^{2}+c_{1}\lambda+c_{0}.
\end{aligned}
\end{equation}
One can list some elementary properties of the function $g(\lambda)$ as follows:
\begin{itemize}
  \item $g(0)=4\pi\left(\frac{e^2}{m_{i}}\nu_{e}+\frac{e^{2}}{m_{e}}\nu_{i}\right)>0;$
  \item $g(-(\nu_{i}+\nu_{e}))=-\nu_{i}^{2}\nu_{e}-\nu_{i}\nu_{e}^{2}-4\pi\left(\frac{e^2}{m_{i}}\nu_{i}+\frac{e^{2}}{m_{e}}\nu_{e}\right)<0;$
  \item $g'(\lambda)=3\lambda^{2}+2(\nu_{i}+\nu_{e})\lambda+\left(\nu_{i}\nu_{e}
+4\pi\left(\frac{e^{2}}{m_{i}}+\frac{e^{2}}{m_{e}}\right)\right)\geq\nu_{i}\nu_{e}+4\pi\left(\frac{e^{2}}{m_{i}}+\frac{e^{2}}{m_{e}}\right)>0
$ for $\lambda\geq 0;$
\item $g'(\lambda)=\lambda^{2}+2(\lambda+\nu_{i}+\nu_{e})\lambda+\left(\nu_{i}\nu_{e}
+4\pi\left(\frac{e^{2}}{m_{i}}+\frac{e^{2}}{m_{e}}\right)\right)\geq\nu_{i}\nu_{e}+4\pi\left(\frac{e^{2}}{m_{i}}+\frac{e^{2}}{m_{e}}\right)>0,$
for $\lambda\leq -(\nu_{i}+\nu_{e});$
\item $g(\lambda)$ is strictly increasing over $\lambda\leq
-(\nu_{i}+\nu_{e})$ or $\lambda\geq 0$.
\end{itemize}
The above properties imply that the equation $g(\lambda)=0$ has at
least one real root denoted by $\sigma$ which satisfies  $-(\nu_{i}+\nu_{e})<\sigma<0$. At this time, although we
have known that there is at least one real root, it is not clear
whether these roots are distinct or not.  We can distinguish several
possible cases using the discriminant,
\begin{equation*}
\Delta=18c_{2}c_{1}c_{0}-4c_{2}^{3}c_{0}+c_{2}^{2}c_{1}^{2}-4c_{1}^{3}-27c_{0}^{2}.
\end{equation*}
\begin{itemize}
  \item  $\Delta>0$, then $g(\lambda)=0$ has three distinct real
  roots;
  \item  $\Delta<0$, then $g(\lambda)=0$ has one real root and two nonreal
complex conjugate roots;
 \item  $\Delta=0$, then $g(\lambda)=0$ has a multiple root and all its roots are
 real.
\end{itemize}
Through the paper, we only consider the first two cases. Note that the third case is much harder to study as Puiseux expansions of the eigenvalues have to be used in that case. Under this assumption, in order to give the asymptotic expressions of ${e^{A(ik)t}}$ as
$|k|\rightarrow 0$,  we see that the solution matrix
${e^{A(ik)t}}$ has the spectral decomposition
\begin{equation*}
{e^{A(ik)t}}=\sum_{j=1}^{4}{\exp({\lambda_{j}(ik)t})}P_{j}(ik),
\end{equation*}
where $\lambda_{j}(\zeta)$ are the eigenvalues of $A(\zeta)$ and
$P_{j}(\zeta)$ are the corresponding eigenprojections. Notice that $P_j(\zeta)$ can be written as
\begin{equation*}
\displaystyle P_{j}(\zeta)=\prod_{\ell\neq
j}\frac{A(\zeta)-\lambda_{\ell}(\zeta)I}{\lambda_{j}(\zeta)-\lambda_{\ell}(\zeta)},
\end{equation*}
where we have assumed that all $\lambda_j(\zeta)$ are distinct to each other. 

%

In terms of
the graph of $g(\lambda)$, one can see when $\Delta>0$,
$g(\lambda)=0$ has three distinct negative real roots. When $\Delta<
0$, assuming that $a+bi$, $a-bi$ are two conjugate complex roots and
plugging $a+bi$ into $g(\lambda)=0$, one has the following two
equations:
\begin{equation*}
  \begin{aligned}
   & \mathbf{\mathrm{Re}}: a^{3}-3ab^{2}+(\nu_{i}+\nu_{e})(a^{2}-b^{2})+\left(\nu_{i}\nu_{e}
+4\pi\left(\frac{e^{2}}{m_{i}}+\frac{e^{2}}{m_{e}}\right)\right)a+4\pi\left(\frac{e^2}{m_{i}}\nu_{e}+\frac{e^{2}}{m_{e}}\nu_{i}\right)=0,\\
   & \mathbf{\mathrm{Im}}:
   3a^{2}b-b^{3}+2(\nu_{i}+\nu_{e})ab+\left(\nu_{i}\nu_{e}
+4\pi\left(\frac{e^{2}}{m_{i}}+\frac{e^{2}}{m_{e}}\right)\right)b=0.
  \end{aligned}
\end{equation*}
Since $b\neq 0$, substituting
$$
b^{2}=3a^{2}+2(\nu_{i}+\nu_{e})a+\left(\nu_{i}\nu_{e}
+4\pi\left(\frac{e^{2}}{m_{i}}+\frac{ e^{2}}{m_{e}}\right)\right)
$$
back into the equation of the real part above, 
we have
\begin{multline*}
(2a)^{3}+2(2a)^{2}(\nu_{i}+\nu_{e})+2a\left(\nu_{i}\nu_{e}
+4\pi \left(\frac{e^{2}}{m_{i}}+\frac{e^{2}}{m_{e}}\right)+(\nu_{i}+\nu_{e})^{2}\right)\\
+\left(\nu_{i}^{2}\nu_{e}+\nu_{e}^{2}\nu_{i}+\frac{4\pi
e^{2}}{m_{i}}\nu_{i}+\frac{4\pi e^{2}}{m_{e}}\nu_{e}\right)=0.
\end{multline*}
Then the above equation must have only one real negative root. By
straightforward computations and using \eqref{eigenpo}, we find that
\begin{equation}\label{roots}
  \begin{aligned}
   &\lambda_{1}(|k|)=\dfrac{T_{i}+T_{e}}{m_{i}\nu_{i}+m_{e}\nu_{e}}\zeta^{2}+\lambda_{1}^{(4)}\zeta^{4}+O(|\zeta|^{5}),\\
   &\lambda_{j}(|k|)=\sigma_{j}+O(|\zeta|^{2}),\ \ \ \ \mathrm{for}\ \
   j=2,3,4,
 \end{aligned}
\end{equation}
where $\sigma_{j}$ $(j=2,3,4)$ are the roots of $g(\lambda)=0$, satisfying
$$
\Re\,\sigma_{j}<0,\quad
\sum\limits_{j=2}^{4}\sigma_{j}=-(\nu_{i}+\nu_{e}),\quad
\sigma_{2}\sigma_{3}\sigma_{4}=-4\pi\left(\frac{e^{2}}{m_{i}}\nu_{e}+\frac{e^{2}}{m_{e}}\nu_{i}\right),
$$
and $\lambda_{1}^{(4)}$ is to be defined later on.

{After checking  the coefficient of $\zeta^{4}$ in \eqref{eigenpo},
we can get some information of $\lambda_{1}^{(4)}$ which is
necessary for the coefficient of $\zeta^{{2}}$ in
$\lambda_{2}\lambda_{3}\lambda_{4}$. In the case $\lambda=\lambda_{1}$ in
\eqref{eigenpo}, the coefficient of $\zeta^{4}$ is}
\begin{equation*}
  \begin{aligned}
 \left(\nu_{i}\nu_{e}
+\frac{4\pi e^{2}}{m_{i}}+\frac{4\pi
e^{2}}{m_{e}}\right)\left(\lambda_{1}^{(2)}\right)^{2}
-&\left(\frac{T_{i}}{m_{i}}\nu_{e}+\frac{T_{e}}{m_{e}}\nu_{i}\right)\lambda_{1}^{(2)}\\
+&4\pi\left(\frac{e^2}{m_{i}}\nu_{e}+\frac{e^{2}}{m_{e}}\nu_{i}\right)\lambda_{1}^{(4)}+\left(\frac{T_{i}T_{e}}{m_{i}m_{e}}\right)=0,
\end{aligned}
\end{equation*}
which implies that
\begin{eqnarray*}
\begin{aligned}
\frac{\lambda_{1}^{(4)}}{\left(\lambda_{1}^{(2)}\right)^{2}}=&-\frac{\left(\nu_{i}\nu_{e}
+4\pi\left(\frac{e^{2}}{m_{i}}+\frac{e^{2}}{m_{e}}\right)\right)-\left(\frac{T_{i}}{m_{i}}\nu_{e}
+\frac{T_{e}}{m_{e}}\nu_{i}\right)\frac{m_{i}\nu_{i}+m_{e}\nu_{e}}{T_{i}+T_{e}}
+\frac{T_{i}T_{e}}{m_{i}m_{e}}\left[\frac{m_{i}\nu_{i}+m_{e}\nu_{e}}{T_{i}+T_{e}}\right]^{2}}{\frac{4\pi e^2}{m_{i}}\nu_{e}+\frac{4\pi e^{2}}{m_{e}}\nu_{i}}\\
=&-\frac{m_{i}\nu_{i}m_{e}\nu_{e} +4\pi e^{2}(m_{i}+m_{e})
-\left(T_{i}m_{e}\nu_{e}
+T_{e}m_{i}\nu_{i}\right)\frac{m_{i}\nu_{i}+m_{e}\nu_{e}}{T_{i}+T_{e}}
+T_{i}T_{e}\left[\frac{m_{i}\nu_{i}+m_{e}\nu_{e}}{T_{i}+T_{e}}\right]^{2}}{4\pi
e^{2}(m_{i}\nu_{i}+m_{e}\nu_{e})},
\end{aligned}
\end{eqnarray*}
and
\begin{eqnarray*}
\begin{aligned}
\lambda_{2}\lambda_{3}\lambda_{4}=&\frac{
\frac{T_{i}T_{e}}{m_{i}m_{e}}\zeta^{4}
-\frac{4\pi e^2(T_{i}+T_{e})}{m_{i}m_{e}}\zeta^{2}}{\frac{T_{i}+T_{e}}{m_{i}\nu_{i}+m_{e}\nu_{e}}\zeta^{2}+\lambda_{1}^{(4)}\zeta^{4}+O(|\zeta|^{5})}\\
=&\left(\frac{T_{i}T_{e}}{m_{i}m_{e}}\zeta^{2} -\frac{4\pi
e^2(T_{i}+T_{e})}{m_{i}m_{e}}\right)
\left(\frac{m_{i}\nu_{i}+m_{e}\nu_{e}}{T_{i}+T_{e}}-\frac{\lambda_{1}^{(4)}}{\left(\lambda_{1}^{(2)}\right)^{2}}\zeta^{2}+O(|\zeta|^{3})\right)\\
=&-4\pi\left(\frac{e^{2}}{m_{i}}\nu_{e}+\frac{e^{2}}{m_{e}}\nu_{i}\right)+\left(\frac{T_{i}T_{e}}{m_{i}m_{e}}\frac{m_{i}\nu_{i}+m_{e}\nu_{e}}{T_{i}
+T_{e}}+\frac{4\pi
e^2(T_{i}+T_{e})}{m_{i}m_{e}}\frac{\lambda_{1}^{(4)}}{\left(\lambda_{1}^{(2)}\right)^{2}}\right)\zeta^{2}+O(|\zeta|^{3}).
\end{aligned}
\end{eqnarray*}

Next, we estimate $P_{1}(\zeta)$ exactly. In the following, we
denote
$$
[A(\zeta)]^{2}\triangleq \left(a_{ij}^{(2)}\right)_{4\times
4},\quad [A(\zeta)]^{3}\triangleq
\left(a_{ij}^{(3)}\right)_{4\times 4}.
$$
One can compute
\begin{equation*}
[A(\zeta)]^{2}=\left(\begin{array} {cccc}
\frac{T_{i}}{m_{i}}\zeta^{2}-\frac{4\pi e^{2}}{m_{i}} & \frac{4\pi e^{2}}{m_{i}} & \nu_{i}\zeta &\ \ \ 0\\[3mm]
\frac{4\pi e^{2}}{m_{e}} & \frac{T_{e}}{m_{e}}\zeta^{2}-\frac{4\pi e^{2}}{m_{e}} &  0 &\ \ \  \nu_{e}\zeta\\[3mm]
\frac{T_{i}}{m_{i}}\nu_{i}\zeta+\frac{4\pi
e^{2}}{m_{i}}\nu_{i}\frac{\zeta}{|\zeta|^{2}} & -\frac{4\pi
e^{2}}{m_{i}}\nu_{i}\frac{\zeta}{|\zeta|^{2}}
 & \frac{T_{i}}{m_{i}}\zeta^{2}-\frac{4\pi e^{2}}{m_{i}}+\nu_{i}^{2} &\ \ \ \frac{4\pi e^{2}}{m_{i}}\\[3mm]
-\frac{4\pi e^{2}}{m_{e}}\nu_{e}\frac{\zeta}{|\zeta|^{2}}&
\frac{T_{e}}{m_{e}}\nu_{e}\zeta+\frac{4\pi
e^{2}}{m_{e}}\nu_{e}\frac{\zeta}{|\zeta|^{2}}
 & \frac{4\pi e^{2}}{m_{e}} &\frac{T_{e}}{m_{e}}\zeta^{2}-\frac{4\pi e^{2}}{m_{e}}+\nu_{e}^{2}
\end{array} \right),
\end{equation*}
and
\begin{equation*}
[A(\zeta)]^{3}=\left(\begin{array} {cccc}
-\frac{T_{i}}{m_{i}}\nu_{i}\zeta^{2}+\frac{4\pi e^{2}}{m_{i}}\nu_{i}
& -\frac{4\pi e^{2}}{m_{i}}\nu_{i} &
-\frac{T_{i}}{m_{i}}\zeta^{3}+\frac{4\pi
e^{2}}{m_{i}}\zeta-\nu_{i}^{2}\zeta &\ \ \
-\frac{4\pi e^{2}}{m_{i}}\zeta\\[3mm]
-\frac{4\pi e^{2}}{m_{e}}\nu_{e} &
-\frac{T_{e}}{m_{e}}\nu_{e}\zeta^{2}+\frac{4\pi e^{2}}{m_{e}}\nu_{e}
& -\frac{4\pi e^{2}}{m_{e}}\zeta \ \ & -\frac{T_{e}}{m_{e}}\zeta^{3}+\frac{4\pi e^{2}}{m_{e}}\zeta-\nu_{e}^{2}\zeta\\[3mm]
 a_{31}^{(3)} & a_{32}^{(3)}
 & a_{33}^{(3)} &\ \ \ a_{34}^{(3)}\\[3mm]
a_{41}^{(3)} & a_{42}^{(3)}
 & a_{43}^{(3)} &\ \ \ a_{44}^{(3)}
\end{array} \right),
\end{equation*}
where
\begin{eqnarray*}
   &&a_{31}^{(3)}=-\left(\dfrac{T_{i}}{m_{i}}\right)^{2}\zeta^{3}
   +\left(2\frac{T_{i}}{m_{i}}\frac{4\pi e^{2}}{m_{i}}-\frac{T_{i}}{m_{i}}\nu_{i}^{2}\right)\zeta
   +\left(\left(\frac{4\pi e^{2}}{m_{i}}\right)^{2}-\frac{4\pi e^{2}}{m_{i}}\nu_{i}^{2}+\frac{4\pi e^{2}}{m_{i}}\frac{4\pi e^{2}}{m_{e}}\right)\frac{\zeta}{|\zeta|^{2}},\\
   &&a_{32}^{(3)}=-\left(\frac{T_{i}}{m_{i}}\frac{4\pi e^{2}}{m_{i}}+\frac{T_{e}}{m_{e}}\frac{4\pi e^{2}}{m_{i}}\right)\zeta
   -\left(\left(\frac{4\pi e^{2}}{m_{i}}\right)^{2}-\frac{4\pi e^{2}}{m_{i}}\nu_{i}^{2}+\frac{4\pi e^{2}}{m_{i}}\frac{4\pi e^{2}}{m_{e}}\right)\frac{\zeta}{|\zeta|^{2}},\\
   &&a_{33}^{(3)}=-2\frac{T_{i}}{m_{i}}\nu_{i}\zeta^{2}+2\frac{4\pi e^{2}}{m_{i}}\nu_{i}-\nu_{i}^{3},\\
  &&a_{34}^{(3)}=-\frac{4\pi e^{2}}{m_{i}}\nu_{i}-\frac{4\pi e^{2}}{m_{i}}\nu_{e},\\
    &&a_{41}^{(3)}=-\left(\frac{T_{i}}{m_{i}}\frac{4\pi e^{2}}{m_{e}}+\frac{T_{e}}{m_{e}}\frac{4\pi e^{2}}{m_{e}}\right)\zeta
   -\left(\left(\frac{4\pi e^{2}}{m_{e}}\right)^{2}-\frac{4\pi e^{2}}{m_{e}}\nu_{e}^{2}+\frac{4\pi e^{2}}{m_{i}}\frac{4\pi e^{2}}{m_{e}}\right)\frac{\zeta}{|\zeta|^{2}},\\
  && a_{42}^{(3)}=-\left(\dfrac{T_{e}}{m_{e}}\right)^{2}\zeta^{3}
   +\left(2\frac{T_{e}}{m_{e}}\frac{4\pi e^{2}}{m_{e}}-\frac{T_{e}}{m_{e}}\nu_{e}^{2}\right)\zeta
   +\left(\left(\frac{4\pi e^{2}}{m_{e}}\right)^{2}-\frac{4\pi e^{2}}{m_{e}}\nu_{e}^{2}+\frac{4\pi e^{2}}{m_{i}}\frac{4\pi e^{2}}{m_{e}}\right)\frac{\zeta}{|\zeta|^{2}},\\
   &&a_{43}^{(3)}=-\frac{4\pi e^{2}}{m_{e}}\nu_{e}-\frac{4\pi e^{2}}{m_{e}}\nu_{i},\\
   &&a_{44}^{(3)}=-2\frac{T_{e}}{m_{e}}\nu_{e}\zeta^{2}+2\frac{4\pi e^{2}}{m_{e}}\nu_{e}-\nu_{e}^{3}.
\end{eqnarray*}
Note that we must deal with terms  involving
$\frac{\zeta}{|\zeta|^{2}}$ carefully, since they contain
singularity as $|k|\rightarrow 0$. By using \eqref{roots}, we
estimate the numerator and denominator of $P_{1}(ik)$, respectively,
in the following way that
\begin{eqnarray*}
\begin{aligned}
&P^{\mathrm{den}}_{1}=:\sum_{\ell=0}^{+\infty}g^{(\ell)}\zeta^{\ell}
=4\pi\left(\nu_{i}\frac{e^{2}}{m_{e}}+\nu_{e}\frac{e^{2}}{m_{i}}\right)+g^{(2)}\zeta^{2}+O(|\zeta|^{3}),
\end{aligned}
\end{eqnarray*}
and
\begin{eqnarray*}
 \begin{aligned}
P^{\mathrm{num}}_{1}=&[A(\zeta)]^{3}-(\lambda_{2}+\lambda_{3}+\lambda_{4})[A(\zeta)]^{2})+(\lambda_{2}\lambda_{3}
   +\lambda_{2}\lambda_{4}+\lambda_{3}\lambda_{4})[A(\zeta)]-\lambda_{2}\lambda_{3}\lambda_{4}I\\
   =&[A(\zeta)]^{3}+(\nu_{i}+\nu_{e}+\lambda_{1})[A(\zeta)]^{2}\\
    &+
   \left(\nu_{i}\nu_{e}-\left(\frac{T_{i}}{m_{i}}+\frac{T_{e}}{m_{e}}\right)\zeta^{2}+\left(\frac{4\pi e^{2}}{m_{i}}+\frac{4\pi e^{2}}{m_{e}}\right)
  -\lambda_{1}(\lambda_{2}+\lambda_{3}+\lambda_{4})\right)[A(\zeta)]-\lambda_{2}\lambda_{3}\lambda_{4}I\\
    =&:(f_{ij})_{4\times 4}.
\end{aligned}
\end{eqnarray*}
Notice that
$$
\frac{1}{P^{\mathrm{den}}_{1}}=
\frac{1}{4\pi\left(\nu_{i}\frac{
e^{2}}{m_{e}}+\nu_{e}\frac{e^{2}}{m_{i}}\right)}-\frac{g^{(2)}}{\left[g^{(0)}\right]^{2}}\zeta^{2}+O(|\zeta|^{3}).
$$
Let us compute $f_{ij}$ ($1\leq i,j\leq 4)$ as follows. For $f_{11}$, one has
\begin{equation*}
 f_{11}(\zeta)
   =-\frac{T_{i}}{m_{i}}\nu_{i}\zeta^{2}+\frac{4\pi e^{2}}{m_{i}}\nu_{i}+(\nu_{i}+\nu_{e}+\lambda_{1})\left(\frac{T_{i}}{m_{i}}\zeta^{2}-\frac{4\pi e^{2}}{m_{i}}\right)
 -\lambda_{2}\lambda_{3}\lambda_{4}=:\sum_{\ell=0}^{+\infty}f_{11}^{(\ell)}|\zeta|^{\ell},
\end{equation*}
where
\begin{eqnarray*}
   \begin{aligned}
   &  f_{11}^{(0)}=\frac{4\pi e^{2}}{m_{i}}\nu_{i}-\frac{4\pi e^{2}}{m_{i}}(\nu_{i}+\nu_{e})
   +\nu_{i}\frac{4\pi e^{2}}{m_{e}}+\nu_{e}\frac{4\pi e^{2}}{m_{i}}=\frac{4\pi e^{2}}{m_{e}}\nu_{i},\\
   &  f_{11}^{(1)}=0,\\
   & \begin{aligned}
   f_{11}^{(2)}=&-\frac{T_{i}}{m_{i}}\nu_{i}+\frac{T_{i}}{m_{i}}(\nu_{i}+\nu_{e})-\frac{4\pi e^{2}}{m_{i}}
   \frac{T_{i}+T_{e}}{m_{i}\nu_{i}+m_{e}\nu_{e}}\\
   &-\left(\frac{T_{i}T_{e}}{m_{i}m_{e}}\frac{m_{i}\nu_{i}+m_{e}\nu_{e}}{T_{i}
+T_{e}}+\frac{4\pi e^2(T_{i}+T_{e})}{m_{i}m_{e}}\frac{\lambda_{1}^{(4)}}{\left(\lambda_{1}^{(2)}\right)^{2}}\right)\\
=&\frac{4\pi e^{2}}{m_{e}}\frac{T_{i}+T_{e}}{m_{i}\nu_{i}
+m_{e}\nu_{e}}+\frac{T_{i}m_{e}\nu_{e}-T_{e}m_{i}\nu_{i}}{m_{i}m_{e}}\frac{m_{i}\nu_{i}}{m_{i}\nu_{i}+m_{e}\nu_{e}},
 \end{aligned}
 \end{aligned}
\end{eqnarray*}
and therefore,
\begin{eqnarray*}
f_{11}(\zeta)=\frac{4\pi e^{2}}{m_{e}}\nu_{i}+ \left(\frac{4\pi
e^{2}}{m_{e}}\frac{T_{i}+T_{e}}{m_{i}\nu_{i}+m_{e}\nu_{e}}
+\frac{T_{i}m_{e}\nu_{e}-T_{e}m_{i}\nu_{i}}{m_{i}m_{e}}\frac{m_{i}\nu_{i}}{m_{i}\nu_{i}+m_{e}\nu_{e}}\right)\zeta^{2}+O(|\zeta|^{3}).
\end{eqnarray*}
 In a similar way, we can get
\begin{eqnarray*}
\begin{aligned} &f_{12}(\zeta)
   =\frac{4\pi e^{2}}{m_{i}}\nu_{e}+\frac{4\pi e^{2}}{m_{i}}\frac{T_{i}+T_{e}}{m_{i}\nu_{i}+m_{e}\nu_{e}}\zeta^{2}+O(|\zeta|^{4}),\\
& f_{13}(\zeta)=-\frac{4\pi e^{2}}{m_{e}}\zeta+O(|\zeta|^{3}),\\
& f_{14}(\zeta)=-\frac{4\pi e^{2}}{m_{i}}\zeta,
 \end{aligned}
\end{eqnarray*}
and
\begin{eqnarray*}
\begin{aligned} f_{21}(\zeta)
   =\frac{4\pi e^{2}}{m_{e}}\nu_{i}+\frac{4\pi e^{2}}{m_{e}}\frac{T_{i}+T_{e}}{m_{i}\nu_{i}+m_{e}\nu_{e}}\zeta^{2}+O(|\zeta|^{4}).
\end{aligned}
\end{eqnarray*}
For $f_{22}(\zeta)$, one has
\begin{eqnarray*}
\begin{aligned} f_{22}(\zeta)
   =-\frac{T_{e}}{m_{e}}\nu_{e}\zeta^{2}+\frac{4\pi e^{2}}{m_{e}}\nu_{e}+(\nu_{i}+\nu_{e}+\lambda_{1})\left(\frac{T_{e}}{m_{e}}\zeta^{2}-\frac{4\pi e^{2}}{m_{e}}\right)
 -\lambda_{2}\lambda_{3}\lambda_{4}=\sum_{\ell=0}^{+\infty}f_{22}^{(\ell)}\zeta^{\ell},
\end{aligned}
\end{eqnarray*}
where
\begin{eqnarray*}
   \begin{aligned}
   &  f_{22}^{(0)}=\frac{4\pi e^{2}}{m_{e}}\nu_{e}-\frac{4\pi e^{2}}{m_{e}}(\nu_{i}+\nu_{e})
   +\nu_{i}\frac{4\pi e^{2}}{m_{e}}+\nu_{e}\frac{4\pi e^{2}}{m_{i}}=\frac{4\pi e^{2}}{m_{i}}\nu_{e}, \\
   &f_{22}^{(1)}=0,\\
   & \begin{aligned}
   f_{22}^{(2)}=&-\frac{T_{e}}{m_{e}}\nu_{e}+\frac{T_{e}}{m_{e}}(\nu_{i}+\nu_{e})-\frac{4\pi e^{2}}{m_{e}}\frac{T_{i}+T_{e}}{m_{i}\nu_{i}+m_{e}\nu_{e}}\\
   &-\left(\frac{T_{i}T_{e}}{m_{i}m_{e}}\frac{m_{i}\nu_{i}+m_{e}\nu_{e}}{T_{i}
+T_{e}}+\frac{4\pi e^2(T_{i}+T_{e})}{m_{i}m_{e}}\frac{\lambda_{1}^{(4)}}{\left(\lambda_{1}^{(2)}\right)^{2}}\right),\\
=&\frac{4\pi e^{2}}{m_{i}}\frac{T_{i}+T_{e}}{m_{i}\nu_{i}
+m_{e}\nu_{e}}-\frac{T_{i}m_{e}\nu_{e}-T_{e}m_{i}\nu_{i}}{m_{i}m_{e}}\frac{m_{e}\nu_{e}}{m_{i}\nu_{i}+m_{e}\nu_{e}}.
 \end{aligned}
 \end{aligned}
\end{eqnarray*}
Therefore,
\begin{eqnarray*}
f_{22}(\zeta)=\frac{4\pi e^{2}}{m_{i}}\nu_{e}+ \left(\frac{4\pi
e^{2}}{m_{i}}\frac{T_{i}+T_{e}}{m_{i}\nu_{i}+m_{e}\nu_{e}}-\frac{T_{i}m_{e}\nu_{e}-T_{e}m_{i}\nu_{i}}{m_{i}m_{e}}\frac{m_{e}\nu_{e}}{m_{i}\nu_{i}+m_{e}\nu_{e}}\right)\zeta^{2}+O(|\zeta|^{3}).
\end{eqnarray*}
Similarly, one has
\begin{eqnarray*}
f_{23}(\zeta)=-\frac{4\pi e^{2}}{m_{e}}\zeta,\quad
f_{24}(\zeta)=-\frac{4\pi e^{2}}{m_{i}}\zeta+O(|\zeta|^{3}).
\end{eqnarray*}
Moreover, it holds that
\begin{eqnarray*}
\begin{aligned} f_{31}(\zeta)
   =&-\left(\dfrac{T_{i}}{m_{i}}\right)^{2}\zeta^{3}
   +\left(2\frac{T_{i}}{m_{i}}\frac{4\pi e^{2}}{m_{i}}-\frac{T_{i}}{m_{i}}\nu_{i}^{2}\right)\zeta
   +\left(\left(\frac{4\pi e^{2}}{m_{i}}\right)^{2}-\frac{4\pi e^{2}}{m_{i}}\nu_{i}^{2}+\frac{4\pi e^{2}}{m_{i}}\frac{4\pi e^{2}}{m_{e}}\right)\frac{\zeta}{|\zeta|^{2}}\\
   &+(\nu_{i}+\nu_{e}+\lambda_{1})\left(\frac{T_{i}}{m_{i}}\nu_{i}\zeta+\frac{4\pi e^{2}}{m_{i}}\nu_{i}\frac{\zeta}{|\zeta|^{2}}\right)\\
&+
\left(\nu_{i}\nu_{e}-\left(\frac{T_{i}}{m_{i}}+\frac{T_{e}}{m_{e}}\right)\zeta^{2}+\left(\frac{4\pi
e^{2}}{m_{i}}+\frac{4\pi e^{2}}{m_{e}}\right)
-\lambda_{1}(\lambda_{2}+\lambda_{3}+\lambda_{4})\right)\left(-\frac{T_{i}}{m_{i}}\zeta-\frac{4\pi
e^{2}}{m_{i}}\frac{\zeta}{|\zeta|^{2}} \right).
\end{aligned}
\end{eqnarray*}
In the expression of  $f_{31}(\zeta)$ above, since the coefficient of $\frac{\zeta}{|\zeta|^{2}}$ is vanishing, i.e.
\begin{eqnarray*}
\left(\left(\frac{4\pi e^{2}}{m_{i}}\right)^{2}-\frac{4\pi
e^{2}}{m_{i}}\nu_{i}^{2}+\frac{4\pi e^{2}}{m_{i}}\frac{4\pi
e^{2}}{m_{e}}\right)
   +(\nu_{i}+\nu_{e})\frac{4\pi e^{2}}{m_{i}}\nu_{i}
- \left(\nu_{i}\nu_{e}+4\pi
\left(\frac{e^{2}}{m_{i}}+\frac{e^{2}}{m_{e}}\right)\right)\frac{4\pi
e^{2}}{m_{i}}=0,
\end{eqnarray*}
and the coefficient of $\zeta$ is given by
\begin{multline*}
\left(2\frac{T_{i}}{m_{i}}\frac{4\pi
e^{2}}{m_{i}}-\frac{T_{i}}{m_{i}}\nu_{i}^{2}\right)+(\nu_{i}+\nu_{e})\frac{T_{i}}{m_{i}}\nu_{i}
-\frac{T_{i}+T_{i}}{m_{i}\nu_{i}+m_{e}\nu_{e}}\frac{4\pi
e^{2}}{m_{i}}\nu_{i}\\
-\left(\nu_{i}\nu_{e}+\left(\frac{4\pi e^{2}}{m_{i}}+\frac{4\pi e^{2}}{m_{e}}\right)\right)\frac{T_{i}}{m_{i}}-\left(\frac{T_{i}}{m_{i}}+\frac{T_{e}}{m_{e}}\right)\frac{4\pi
e^{2}}{m_{i}}+(\nu_{i}+\nu_{e})\frac{T_{i}+T_{i}}{m_{i}\nu_{i}+m_{e}\nu_{e}}\frac{4\pi
e^{2}}{m_{i}} \\
=-\frac{4\pi
e^{2}(T_{i}+T_{e})}{m_{i}m_{e}}\frac{m_{i}\nu_{i}}{m_{i}\nu_{i}+m_{e}\nu_{e}},
\end{multline*}
it follows that
\begin{eqnarray*}
&&f_{31}(\zeta)=-\frac{4\pi
e^{2}(T_{i}+T_{e})}{m_{i}m_{e}}\frac{m_{i}\nu_{i}}{m_{i}\nu_{i}+m_{e}\nu_{e}}\zeta+O(|\zeta|^{3}).
\end{eqnarray*}
Similarly we can calculate $f_{32}(\zeta)$, $f_{33}(\zeta)$,
$f_{34}(\zeta)$ and $f_{4j}(\zeta)$ $(j=1,2,3,4)$ as follows:
\begin{eqnarray*}
&&f_{32}(\zeta)=-\frac{4\pi e^{2}(T_{i}+T_{e})}{m_{i}m_{e}}\frac{m_{e}\nu_{e}}{m_{i}\nu_{i}+m_{e}\nu_{e}}\zeta+O(|\zeta|^3),\\
&&f_{33}(\zeta)=O(|\zeta|^{2}),\\
&&f_{34}(\zeta)=O(|\zeta|^{2}),\\
&&f_{41}(\zeta)=-\frac{4\pi e^{2}(T_{i}+T_{e})}{m_{i}m_{e}}\frac{m_{i}\nu_{i}}{m_{i}\nu_{i}+m_{e}\nu_{e}}\zeta+O(|\zeta|^3),\\
&&f_{42}(\zeta)=-\frac{4\pi e^{2}(T_{i}+T_{e})}{m_{i}m_{e}}\frac{m_{e}\nu_{e}}{m_{i}\nu_{i}+m_{e}\nu_{e}}\zeta+O(|\zeta|^3),\\
&&f_{43}(\zeta)=O(|\zeta|^{2}),\\
&&f_{44}(\zeta)=O(|\zeta|^{2}).
\end{eqnarray*}

Let $P_{j}^{1}(ik),\ P_{j}^{2}(ik),\ P_{j}^{3}(ik),\ P_{j}^{4}(ik)$
be the four row vectors of $P_{j}(ik)$, $j=1,2,3,4$. According to the above computations, we have
\begin{equation*}
P_{1}^{1}(ik)=\frac{1}{P^{\mathrm{den}}_{1}}\left(\begin{array} {c}
\frac{4\pi e^{2}}{m_{e}}\nu_{i}+ \left(\frac{4\pi
e^{2}}{m_{e}}\frac{T_{i}+T_{e}}{m_{i}\nu_{i}+m_{e}\nu_{e}}
+\frac{T_{i}m_{e}\nu_{e}-T_{e}m_{i}\nu_{i}}{m_{i}m_{e}}\frac{m_{i}\nu_{i}}{m_{i}\nu_{i}+m_{e}\nu_{e}}\right)\zeta^{2}+O(|\zeta|^{3})\\
\frac{4\pi e^{2}}{m_{i}}\nu_{e}+\frac{4\pi e^{2}}{m_{i}}\frac{T_{i}+T_{e}}{m_{i}\nu_{i}+m_{e}\nu_{e}}\zeta^{2}+O(|\zeta|^{3})\\
-\frac{4\pi e^{2}}{m_{e}}\zeta+O(|\zeta|^{3})\\
-\frac{4\pi e^{2}}{m_{i}}\zeta
\end{array} \right)^{T},
\end{equation*}
\begin{equation*}
P_{1}^{2}(ik)=\frac{1}{P^{\mathrm{den}}_{1}}\left(\begin{array} {c}
 \frac{4\pi e^{2}}{m_{e}}\nu_{i}+\frac{4\pi e^{2}}{m_{e}}\frac{T_{i}+T_{e}}{m_{i}\nu_{i}+m_{e}\nu_{e}}\zeta^{2}+O(|\zeta|^{3})\\
\frac{4\pi e^{2}}{m_{i}}\nu_{e}+ \left(\frac{4\pi
e^{2}}{m_{i}}\frac{T_{i}+T_{e}}{m_{i}\nu_{i}+m_{e}\nu_{e}}
-\frac{T_{i}m_{e}\nu_{e}-T_{e}m_{i}\nu_{i}}{m_{i}m_{e}}\frac{m_{e}\nu_{e}}{m_{i}\nu_{i}+m_{e}\nu_{e}}\right)\zeta^{2}+O(|\zeta|^{3})\\
-\frac{4\pi e^{2}}{m_{e}}\zeta\\
-\frac{4\pi e^{2}}{m_{i}}\zeta+O(|\zeta|^{3})
\end{array} \right)^{T},
\end{equation*}
\begin{equation*}
P_{1}^{3}(ik)=\frac{1}{P^{\mathrm{den}}_{1}}\left(\begin{array} {c}
 -\frac{4\pi e^{2}(T_{i}+T_{e})}{m_{i}m_{e}}\frac{m_{i}\nu_{i}}{m_{i}\nu_{i}+m_{e}\nu_{e}}\zeta+O(|\zeta|^{3})\\
-\frac{4\pi e^{2}(T_{i}+T_{e})}{m_{i}m_{e}}\frac{m_{e}\nu_{e}}{m_{i}\nu_{i}+m_{e}\nu_{e}}\zeta+O(|\zeta|^3)\\
O(|\zeta|^{2})\\
O(|\zeta|^{2})
\end{array} \right)^{T},
\end{equation*}
and
\begin{equation*}
P_{1}^{4}(ik)=\frac{1}{P^{\mathrm{den}}_{1}}\left(\begin{array} {c}
 -\frac{4\pi e^{2}(T_{i}+T_{e})}{m_{i}m_{e}}\frac{m_{i}\nu_{i}}{m_{i}\nu_{i}+m_{e}\nu_{e}}\zeta+O(|\zeta|^{3})\\
-\frac{4\pi e^{2}(T_{i}+T_{e})}{m_{i}m_{e}}\frac{m_{e}\nu_{e}}{m_{i}\nu_{i}+m_{e}\nu_{e}}\zeta+O(|\zeta|^3)\\
O(|\zeta|^{2})\\
O(|\zeta|^{2})
\end{array} \right)^{T}.
\end{equation*}
Then we have the expressions of $\hat{\rho}_{\alpha}(\zeta)$,
$\hat{u}_{\alpha,\parallel}(\zeta)$ and $\hat{E}_{\parallel}(\zeta)$
for $|k|\rightarrow 0$ as follows:
\begin{eqnarray*}
\begin{aligned}
 \hat{\rho}_{i}(\zeta)
=&\frac{\exp{(\lambda_{1}(ik)t)}\hat{U}_{\parallel
0}(k)}{P^{\mathrm{den}}_{1}} \left(\begin{array} {c} \frac{4\pi
e^{2}}{m_{e}}\nu_{i}+ \left(\frac{4\pi
e^{2}}{m_{e}}\frac{T_{i}+T_{e}}{m_{i}\nu_{i}+m_{e}\nu_{e}}
+\frac{T_{i}m_{e}\nu_{e}-T_{e}m_{i}\nu_{i}}{m_{i}m_{e}}\frac{m_{i}\nu_{i}}{m_{i}\nu_{i}+m_{e}\nu_{e}}\right)\zeta^{2}+O(|\zeta|^{3})\\
\frac{4\pi e^{2}}{m_{i}}\nu_{e}+\frac{4\pi e^{2}}{m_{i}}\frac{T_{i}+T_{e}}{m_{i}\nu_{i}+m_{e}\nu_{e}}\zeta^{2}+O(|\zeta|^{4})\\
-\frac{4\pi e^{2}}{m_{e}}\zeta+O(|\zeta|^{3})\\
-\frac{4\pi e^{2}}{m_{i}}\zeta
\end{array} \right)\\
&+\sum_{j=2}^{4}\exp{(\lambda_{j}(ik)t)}P_{j}^{1}(ik)\hat{U}_{\parallel
0}(k)^{T}\\
=&\exp{(\lambda_{1}(ik)t)}\bar{P}^{1}\hat{U}_{\parallel
0}(k)^{T}+O(|k|)\exp{(\lambda_{1}(ik)t)}\left|\hat{U}_{\parallel
0}(k)\right|+\sum_{j=2}^{4}\exp{(\lambda_{j}(ik)t)}P_{j}^{1}(ik)\hat{U}_{\parallel
0}(k)^{T},
\end{aligned}
\end{eqnarray*}
\begin{eqnarray*}
\begin{aligned}
 \hat{\rho}_{e}(\zeta)
=&\frac{\exp{(\lambda_{1}(ik)t)}\hat{U}_{\parallel
0}(k)}{P^{\mathrm{den}}_{1}} \left(\begin{array} {c}
 \frac{4\pi e^{2}}{m_{e}}\nu_{i}+\frac{4\pi e^{2}}{m_{e}}\frac{T_{i}+T_{e}}{m_{i}\nu_{i}+m_{e}\nu_{e}}\zeta^{2}+O(|\zeta|^{4})\\
\frac{4\pi e^{2}}{m_{i}}\nu_{e}+ \left(\frac{4\pi
e^{2}}{m_{i}}\frac{T_{i}+T_{e}}{m_{i}\nu_{i}+m_{e}\nu_{e}}
-\frac{T_{i}m_{e}\nu_{e}-T_{e}m_{i}\nu_{i}}{m_{i}m_{e}}\frac{m_{e}\nu_{e}}{m_{i}\nu_{i}+m_{e}\nu_{e}}\right)\zeta^{2}+O(|\zeta|^{3})\\
-\frac{4\pi e^{2}}{m_{e}}\zeta\\
-\frac{4\pi e^{2}}{m_{i}}\zeta+O(|\zeta|^{3})
\end{array} \right)\\
&+\sum_{j=2}^{4}\exp{(\lambda_{j}(ik)t)}P_{j}^{2}(ik)\hat{U}_{\parallel
0}(k)^T\\
=&\exp{(\lambda_{1}(ik)t)}\bar{P}^{1}\hat{U}_{\parallel
0}(k)^{T}+O(|k|)\exp{(\lambda_{1}(ik)t)}|\hat{U}_{\parallel
0}(k)|+\sum_{j=2}^{4}\exp{(\lambda_{j}(ik)t)}P_{j}^{2}(ik)\hat{U}_{\parallel
0}(k)^{T},
\end{aligned}
\end{eqnarray*}
\begin{eqnarray*}
\begin{aligned}
 \hat{u}_{i,\parallel}(\zeta)
=&\frac{\exp{(\lambda_{1}(ik)t)}\hat{U}_{\parallel
0}(k)}{P^{\mathrm{den}}_{1}} \left(\begin{array} {c}
 -\frac{4\pi e^{2}(T_{i}+T_{e})}{m_{i}m_{e}}\frac{m_{i}\nu_{i}}{m_{i}\nu_{i}+m_{e}\nu_{e}}\zeta+O(|\zeta|^{3})\\
-\frac{4\pi e^{2}(T_{i}+T_{e})}{m_{i}m_{e}}\frac{m_{e}\nu_{e}}{m_{i}\nu_{i}+m_{e}\nu_{e}}\zeta+O(|\zeta|^3)\\
O(|\zeta|^{2})\\
O(|\zeta|^{2})
\end{array} \right)\\
&+\sum_{j=2}^{4}\exp{(\lambda_{j}(ik)t)}P_{j}^{3}(ik)\hat{U}_{\parallel
0}(k)^T\\
=&\exp{(\lambda_{1}(ik)t)}\bar{P}^{2}\hat{U}_{\parallel
0}(k)^{T}+O(|k|^{2})\exp{(\lambda_{1}(ik)t)}|\hat{U}_{\parallel
0}(k)|+\sum_{j=2}^{4}\exp{(\lambda_{j}(ik)t)}P_{j}^{3}(ik)\hat{U}_{\parallel
0}(k)^T,
\end{aligned}
\end{eqnarray*}
\begin{eqnarray*}
\begin{aligned}
 \hat{u}_{e,\parallel}(\zeta)
=&\frac{\exp{(\lambda_{1}(ik)t)}\hat{U}_{\parallel
0}(k)}{P^{\mathrm{den}}_{1}} \left(\begin{array} {c}
 -\frac{4\pi e^{2}(T_{i}+T_{e})}{m_{i}m_{e}}\frac{m_{i}\nu_{i}}{m_{i}\nu_{i}+m_{e}\nu_{e}}\zeta+O(|\zeta|^{3})\\
-\frac{4\pi e^{2}(T_{i}+T_{e})}{m_{i}m_{e}}\frac{m_{e}\nu_{e}}{m_{i}\nu_{i}+m_{e}\nu_{e}}\zeta+O(|\zeta|^3)\\
O(|\zeta|^{2})\\
O(|\zeta|^{2})
\end{array} \right)\\
&+\sum_{j=2}^{4}\exp{(\lambda_{j}(ik)t)}P_{j}^{4}(ik)\hat{U}_{\parallel
0}(k)^T\\
=&\exp{(\lambda_{1}(ik)t)}\bar{P}^{2}\hat{U}_{\parallel
0}(k)^{T}+O(|k|^{2})\exp{(\lambda_{1}(ik)t)}|\hat{U}_{\parallel
0}(k)|+\sum_{j=2}^{4}\exp{(\lambda_{j}(ik)t)}P_{j}^{4}(ik)\hat{U}_{\parallel
0}(k)^T.
\end{aligned}
\end{eqnarray*}
Finally, noticing
\begin{equation*}
P_{1}^{1}(ik)-P_{1}^{2}(ik)=\frac{1}{P^{\mathrm{den}}_{1}}\left(\begin{array}
{c}
 \frac{T_{i}m_{e}\nu_{e}-T_{e}m_{i}\nu_{i}}{m_{i}m_{e}}\frac{m_{i}\nu_{i}}{m_{i}\nu_{i}+m_{e}\nu_{e}}\zeta^{2}+O(|\zeta|^{4})\\
\frac{T_{i}m_{e}\nu_{e}-T_{e}m_{i}\nu_{i}}{m_{i}m_{e}}\frac{m_{e}\nu_{e}}{m_{i}\nu_{i}+m_{e}\nu_{e}}\zeta^{2}+O(|\zeta|^{3})\\
O(|\zeta|^{3})\\
O(|\zeta|^{3})
\end{array} \right)^{T},
\end{equation*}
we also have
\begin{eqnarray}\label{Eparallel}
\begin{aligned}
\hat{E}_{\parallel}=&-4\pi \frac{ik}{|k|^{2}}\sum_{\alpha}q_{\alpha}\rho_{\alpha}=-4\pi \frac{ik}{|k|^{2}}(e\hat{\rho}_{i}-e\hat{\rho}_{e})\\
                   =&-4\pi e\frac{ik}{|k|^{2}}\exp(\lambda_{1}(ik) t)\left(P_{1}^{1}(ik)-P_{1}^{2}(ik)\right)\hat{U}_{\parallel
0}(k)^{T}\\
&-4\pi e
\frac{ik}{|k|^{2}}\sum_{j=2}^{4}\exp{(\lambda_{j}(ik)t)}\left(P_{j}^{1}(ik)-P_{j}^{2}(ik)\right)\hat{U}_{\parallel
0}(k)^{T}\\
=&\exp(\lambda_{1}(ik) t)\bar{P}^{3}\hat{U}_{\parallel0}^{T}+O(|k|^{2})\exp(\lambda_{1}(ik) t)\left|\hat{U}_{\parallel
0}(k)\right|\\
&-4\pi e
\frac{ik}{|k|^{2}}\sum_{j=2}^{4}\exp{(\lambda_{j}(ik)t)}\left(P_{j}^{1}(ik)-P_{j}^{2}(ik)\right)\hat{U}_{\parallel
0}(k)^{T}.
\end{aligned}
\end{eqnarray}

\subsubsection{Error estimates}
 \begin{lemma}\label{errorL}
 There is $r_{0}>0$ such that for $|k|\leq r_{0}$ and $t\geq 0$, the error term $|U_{\parallel}-\overline{U}_{\parallel}|$ can be bounded as
\begin{eqnarray}
&&|\hat{\rho}_{\alpha}(t,k)-\hat{\bar{\rho}}(t,k)|\leq C
|k|\exp{\left(-\lambda|k|^{2}t\right)}\left|\hat{U}_{\parallel
0}(k)\right|+C\exp{\left(-\lambda t\right)}\left|\hat{U}_{\parallel
0}(k)\right|,\label{error1}\\
&&|\hat{u}_{\alpha,\parallel}(t,k)-\hat{\bar{u}}_{\parallel}(t,k)|\leq C
|k|^{2}\exp{\left(-\lambda|k|^{2}t\right)}\left|\hat{U}_{\parallel
0}(k)\right|\notag\\
&&\qquad\qquad\qquad\qquad\qquad\qquad\qquad\qquad+C\exp{\left(-\lambda
t\right)}\left(\left|\hat{U}_{\parallel
0}(k)\right|+\left|\hat{E}_{\parallel 0}(k)\right|\right),\label{error2}\\
&&|\hat{E}_{\parallel}(t,k)-\hat{\bar{E}}_{\parallel}(t,k)|\leq C
|k|^{2}\exp{\left(-\lambda|k|^{2}t\right)}\left|\hat{U}_{\parallel
0}(k)\right|+C\exp{\left(-\lambda
t\right)}\left(\left|\hat{U}_{\parallel
0}(k)\right|+\left|\hat{E}_{\parallel 0}(k)\right|\right),\label{error3}
\end{eqnarray}
where $C$ and $\lambda$ are positive constants.
\end{lemma}

\begin{proof} It follows from the expressions of
$\hat{\rho}_{\alpha}(\zeta)$ and $\hat{\bar{\rho}}(\zeta)$ that
\begin{eqnarray*}
&&\begin{aligned}
  &\hat{\rho}_{i}(\zeta)-\hat{\bar{\rho}}(\zeta)\\
= &\exp{(\lambda_{1}(ik)t)}\bar{P}^{1}\hat{U}_{\parallel 0}(k)^{T}-
\exp{\left(-\frac{T_{i}+T_{e}}{m_{i}\nu_{i}+m_{e}\nu_{e}}|k|^{2}t\right)}\bar{P}^{1}\hat{U}_{\parallel
0}(k)^{T}\\
&+O(|k|)\exp{(\lambda_{1}(ik)t)}\left|\hat{U}_{\parallel
0}(k)\right|+\sum_{j=2}^{4}\exp{(\lambda_{j}(ik)t)}P_{j}^{1}(ik)\hat{U}_{\parallel
0}(k)^{T}\\
               :=&\hat{R}_{11}(ik)+\hat{R}_{12}(ik)+\hat{R}_{13}(ik).
\end{aligned}
\end{eqnarray*}
We have from \eqref{roots} that
\begin{equation*}
\lambda_{1}(ik)+\frac{T_{i}+T_{e}}{m_{i}\nu_{i}+m_{e}\nu_{e}}|k|^{2}=O(|k|^{4}),
\end{equation*}
and
\begin{eqnarray*}
\begin{aligned}
&\left|\exp{(\lambda_{1}(ik)t)}-\exp{\left(-\frac{T_{i}+T_{e}}{m_{i}\nu_{i}+m_{e}\nu_{e}}|k|^{2}t\right)}\right|\\
&=\exp{\left(-\frac{T_{i}+T_{e}}{m_{i}\nu_{i}+m_{e}\nu_{e}}|k|^{2}t\right)}
\left|\exp{\left(\lambda_{1}(ik)t+\frac{T_{i}+T_{e}}{m_{i}\nu_{i}+m_{e}\nu_{e}}|k|^{2}t\right)}-1\right|\\
&\leq
C\exp{\left(-\frac{T_{i}+T_{e}}{m_{i}\nu_{i}+m_{e}\nu_{e}}|k|^{2}t\right)}
|k|^{4}t \exp{(C|k|^{4}t)}\\
&\leq  C|k|^{2}\exp{\left(-\lambda|k|^{2}t\right)},
\end{aligned}
\end{eqnarray*}
as $|k|\to 0$.
 Therefore, we obtain that
\begin{eqnarray*}
\left|\hat{R}_{11}(ik)\right|\leq
C|k|^{2}\exp{\left(-\lambda|k|^{2}t\right)}\left|\hat{U}_{\parallel
0}(k)\right|\ \ \ \mathrm{as}\ \ |k|\rightarrow 0.
\end{eqnarray*}
Note that $\Re\, \lambda_{1}(ik)\leq -\lambda|k|^2$ and
$|{\exp({\lambda_{1}(ik)t})}|\leq {\exp({-\lambda|k|^2t})}$
as $|k|\rightarrow 0$. Consequently, we find that
\begin{eqnarray*}
\left|\hat{R}_{12}(ik)\right|\leq C|k|
{\exp({-\lambda|k|^2t})}\quad \mathrm{as}\ \ |k|\rightarrow 0.
\end{eqnarray*}
Now it suffices to estimate $\left|\hat{R}_{13}(ik)\right|$. Recall
$\mathrm{Re}\ \sigma_{j}<0$ for $j=2,3,4$. This together with
\eqref{roots} give ${\exp({\lambda_{j}(ik)t})}\leq
{\exp({-\lambda t})}$ as $|k|\rightarrow 0$. Also notice
$P_{j}^{1}(ik)=O(1)$. Thus we have
\begin{eqnarray*}
\left|\hat{R}_{13}(ik)\right|\leq C{\exp({-\lambda
t})}\left|\hat{U}_{\parallel 0}(k)\right|\ \ \ \mathrm{as}\ \
|k|\rightarrow 0.
\end{eqnarray*}
In a similar way, we can get
\begin{equation*}
|\hat{\rho}_{e}(k)-\hat{\bar{\rho}}(k)|\leq C
|k|\exp{\left(-\lambda|k|^{2}t\right)}\left|\hat{U}_{\parallel
0}(k)\right|+C\exp{\left(-\lambda t\right)}\left|\hat{U}_{\parallel
0}(k)\right|.
\end{equation*}
This then proves the desired estimate $\eqref{error1}$.

To consider the rest estimates, one has to prove that
$$
P_{j}^{3}{U}_{\parallel
0}^T(k),\quad
P_{j}^{4}{U}_{\parallel
0}^T(k),\quad
\frac{ik}{|k|^{2}}\sum_{j=2}^{4}\left(P_{j}^{1}(ik)-P_{j}^{2}(ik)\right){U}_{\parallel
0}^T(k),
$$
are all bounded. Notice that those terms include
$\frac{\zeta}{|\zeta|^{2}}$ which  is singular as $|k|\rightarrow
0$.
For $j=2,3,4$, by using \eqref{roots}, we see that
$$
P^{\mathrm{den}}_{j}=\prod\limits_{\ell \neq j
}(\lambda_{j}(\zeta)-\lambda_{l}(\zeta))=O(1),
$$
and
\begin{eqnarray}\label{Pj}
&&\begin{aligned} P^{\mathrm{num}}_{j}
   =&[A(\zeta)]^{3}+(\nu_{i}+\nu_{e}+\lambda_{j})[A(\zeta)]^{2}\\
    &+
   \left(\nu_{i}\nu_{e}-\left(\frac{T_{i}}{m_{i}}+\frac{T_{e}}{m_{e}}\right)\zeta^{2}+4\pi \left(\frac{e^{2}}{m_{i}}+\frac{e^{2}}{m_{e}}\right)
  +\lambda_{j}(\nu_{i}+\nu_{e}+\lambda_{j})\right)A(\zeta)\\
  &-\lambda_{2}\lambda_{3}\lambda_{4}I\\
  :=&(g_{il}^{j})_{4\times4}.
\end{aligned}
\end{eqnarray}
We have to be careful to  treat the third row and the fourth row involving $\frac{\zeta}{|\zeta|^{2}}$. It is straightforward to compute
$g^{j}_{31}$ as
\begin{eqnarray*}
\begin{aligned}g^{j}_{31}(\zeta)
   =&-\left(\dfrac{T_{i}}{m_{i}}\right)^{2}\zeta^{3}
   +\left(2\frac{T_{i}}{m_{i}}\frac{4\pi e^{2}}{m_{i}}-\frac{T_{i}}{m_{i}}\nu_{i}^{2}\right)\zeta
   +\left(\left(\frac{4\pi e^{2}}{m_{i}}\right)^{2}-\frac{4\pi e^{2}}{m_{i}}\nu_{i}^{2}+\frac{4\pi e^{2}}{m_{i}}\frac{4\pi e^{2}}{m_{e}}\right)\frac{\zeta}{|\zeta|^{2}}\\
   &+(\nu_{i}+\nu_{e}+\lambda_{j})\left(\frac{T_{i}}{m_{i}}\nu_{i}\zeta+\frac{4\pi e^{2}}{m_{i}}\nu_{i}\frac{\zeta}{|\zeta|^{2}}\right)\\
&+
\left(\nu_{i}\nu_{e}-\left(\frac{T_{i}}{m_{i}}+\frac{T_{e}}{m_{e}}\right)\zeta^{2}+\left(\frac{4\pi
e^{2}}{m_{i}}+\frac{4\pi e^{2}}{m_{e}}\right)
+\lambda_{j}(\nu_{i}+\nu_{e}+\lambda_{j})\right)\left(-\frac{T_{i}}{m_{i}}\zeta-\frac{4\pi
e^{2}}{m_{i}}\frac{\zeta}{|\zeta|^{2}} \right).
\end{aligned}
\end{eqnarray*}
The coefficient of $\frac{\zeta}{|\zeta|^{2}}$ in the above expression of
$g^{j}_{31}(\zeta)$ is further simplified as
\begin{eqnarray*}
\frac{4\pi
e^{2}}{m_{i}}\nu_{i}\sigma_{j}-\sigma_{j}(\nu_{i}+\nu_{e}+\sigma_{j})\frac{4\pi
e^{2}}{m_{i}}=-\frac{4\pi
e^{2}}{m_{i}}\sigma_{j}(\nu_{e}+\sigma_{j}),
\end{eqnarray*}
where we recall that $\sigma_{j}$ $(j=2,3,4)$ are the three roots of $g(\lambda)=0$. Therefore,
\begin{eqnarray*}
 g^{j}_{31}=-\frac{4\pi e^{2}}{m_{i}}\sigma_{j}(\nu_{e}+\sigma_{j})\frac{\zeta}{|\zeta|^{2}}+O(|\zeta|)
       =-\frac{4\pi e^{2}}{m_{i}}\sigma_{j}(\nu_{e}+\sigma_{j})\frac{ik}{|k|^{2}}+O(|k|).
\end{eqnarray*}
We now turn to estimate $g^{j}_{32}$. It follows that
\begin{eqnarray*}
\begin{aligned}g^{j}_{32}(\zeta)
   =&-\left(\frac{T_{i}}{m_{i}}\frac{4\pi e^{2}}{m_{i}}+\frac{T_{e}}{m_{e}}\frac{4\pi e^{2}}{m_{i}}\right)\zeta
   -\left(\left(\frac{4\pi e^{2}}{m_{i}}\right)^{2}-\frac{4\pi e^{2}}{m_{i}}\nu_{i}^{2}+\frac{4\pi e^{2}}{m_{i}}\frac{4\pi e^{2}}{m_{e}}\right)\frac{\zeta}{|\zeta|^{2}}\\
&+(\nu_{i}+\nu_{e}+\lambda_{j})\left(-\frac{4\pi e^{2}}{m_{i}}\nu_{i}\frac{\zeta}{|\zeta|^{2}}\right)\\
&+
\left(\nu_{i}\nu_{e}-\left(\frac{T_{i}}{m_{i}}+\frac{T_{e}}{m_{e}}\right)\zeta^{2}+\left(\frac{4\pi
e^{2}}{m_{i}}+\frac{4\pi e^{2}}{m_{e}}\right)
+\lambda_{j}(\nu_{i}+\nu_{e}+\lambda_{j})\right)\frac{4\pi
e^{2}}{m_{i}}\frac{\zeta}{|\zeta|^{2}}.
\end{aligned}
\end{eqnarray*}
The coefficient of $\frac{\zeta}{|\zeta|^{2}}$ in the above expression is further simplified as
\begin{eqnarray*}
-\frac{4\pi
 e^{2}}{m_{i}}\nu_{i}\sigma_{j}+\sigma_{j}(\nu_{i}+\nu_{e}+\sigma_{j})\frac{4\pi  e^{2}}{m_{i}}=\frac{4\pi
 e^{2}}{m_{i}}\sigma_{j}(\nu_{e}+\sigma_{j}),
\end{eqnarray*}
which hence implies that
\begin{eqnarray*}
 g^{j}_{32}=\frac{4\pi e^{2}}{m_{i}}\sigma_{j}(\nu_{e}+\sigma_{j})\frac{\zeta}{|\zeta|^{2}}+O(|\zeta|)
       =\frac{4\pi e^{2}}{m_{i}}\sigma_{j}(\nu_{e}+\sigma_{j})\frac{ik}{|k|^{2}}+O(|k|).
\end{eqnarray*}
Checking the third row  of $A(ik)$, $[A(ik)]^{2}$ and
$[A(ik)]^{3}$, we can obtain that
\begin{eqnarray*}
 g^{j}_{33}=O(1),\ \ \  g^{j}_{34}=O(1),
\end{eqnarray*}
as $|k|\to 0$.
It is direct to verify that
\begin{eqnarray*}
\begin{aligned}
P_{j}^{3}(ik)\hat{U}_{\parallel 0}(k)^{T}=&
g^{j}_{31}\hat{\rho}_{i0}+g^{j}_{32}\hat{\rho}_{e0}+g^{j}_{33}\hat{u}_{i0,\parallel}+g^{j}_{34}\hat{u}_{e0,\parallel}\\
=&-\frac{4\pi
 e^{2}}{m_{i}}\sigma_{j}(\nu_{e}+\sigma_{j})\frac{ik}{|k|^{2}}(\hat{\rho}_{i0}-\hat{\rho}_{e0})+O(1)\left|\hat{U}_{\parallel
0}(k)\right|\\
=& \frac{
 e}{m_{i}}\sigma_{j}(\nu_{e}+\sigma_{j})\hat{E}_{\parallel
0}(k)+O(1)\left|\hat{U}_{\parallel 0}(k)\right|,
\end{aligned}
\end{eqnarray*}
where we have used the compatible condition $\hat{E}_{\parallel
0}=-4\pi\frac{ik}{|k|^{2}}(e\hat{\rho}_{i0}-e\hat{\rho}_{e0})$. Then
the expressions of $\hat{u}_{\alpha,\parallel}(\zeta)$ and
$\hat{\bar{u}}_{\parallel}(\zeta)$ imply that
\begin{eqnarray*}
\begin{aligned}
  |\hat{u}_{i,\parallel}(\zeta)-\hat{\bar{u}}_{\parallel}(\zeta)|
= &\exp{(\lambda_{1}(ik)t)}\bar{P}^{2}\hat{U}_{\parallel 0}(k)^{T}-
\exp{\left(-\frac{T_{i}+T_{e}}{m_{i}\nu_{i}+m_{e}\nu_{e}}|k|^{2}t\right)}\bar{P}^{2}\hat{U}_{\parallel
0}(k)^{T}\\
&+O(|k|^{2})\exp{(\lambda_{1}(ik)t)}|\hat{U}_{\parallel
0}(k)|+\sum_{j=2}^{4}\exp{(\lambda_{j}(ik)t)}P_{j}^{3}(ik)\hat{U}_{\parallel
0}(k)\\
\leq &C|k|^{2}\exp{\left(-\lambda|k|^{2}t\right)}|\hat{U}_{\parallel
0}(k)|+C{\exp({-\lambda t})}\left(\left|\hat{U}_{\parallel
0}(k)\right|+\left|\hat{E}_{\parallel 0}(k)\right|\right).
\end{aligned}
\end{eqnarray*}
In a similar way, we  can get
\begin{equation*}
|\hat{u}_{e,\parallel}(k)-\hat{\bar{u}}(k)|\leq C
|k|^{2}\exp{\left(-\lambda|k|^{2}t\right)}\left|\hat{U}_{\parallel
0}(k)\right|+C\exp{\left(-\lambda
t\right)}\left(\left|\hat{U}_{\parallel
0}(k)\right|+\left|\hat{E}_{\parallel 0}(k)\right|\right).
\end{equation*}
This proves \eqref{error2}.

It now remains to estimate
\begin{equation}
\label{du.p01}
\frac{ik}{|k|^{2}}\sum_{j=2}^{4}\left(P_{j}^{1}(ik)-P_{j}^{2}(ik)\right)\hat{U}_{\parallel
0}(k)^{T},
\end{equation}
appearing in \eqref{Eparallel}.
Since the first row minus the second row of
$I$, $A(ik)$, $[A(ik)]^{2}$ and $[A(ik)]^{3}$ are respectively given by
\begin{eqnarray*}
&&(1,-1,0,0),\\
&&(0,0,-\zeta,\zeta),\\
&&\left(-\frac{T_{i}}{m_{i}}\zeta^{2}-\frac{4\pi
e^{2}}{m_{i}}-\frac{4\pi  e^{2}}{m_{e}},
-\frac{T_{e}}{m_{e}}\zeta^{2}+\frac{4\pi
 e^{2}}{m_{i}}+\frac{4\pi  e^{2}}{m_{e}},\nu_{i}\zeta,-\nu_{e}\zeta\right),
\end{eqnarray*}
and
\begin{multline*}
\left(-\frac{T_{i}}{m_{i}}\nu_{i}\zeta^{2}+\frac{4\pi
 e^{2}}{m_{i}}\nu_{i}+\frac{4\pi  e^{2}}{m_{e}}\nu_{e},
\frac{T_{e}}{m_{e}}\nu_{e}\zeta^{2}-\frac{4\pi  e^{2}}{m_{i}}\nu_{i}-\frac{4\pi  e^{2}}{m_{e}}\nu_{e},\right.\\
\left. -\frac{T_{i}}{m_{i}}\zeta^{3}+\left(\frac{4\pi
e^{2}}{m_{i}}+\frac{4\pi e^{2}}{m_{e}}\right)\zeta-\nu_{i}^{2}\zeta,
\frac{T_{e}}{m_{e}}\zeta^{3}-\left(\frac{4\pi
 e^{2}}{m_{i}}+\frac{4\pi
 e^{2}}{m_{e}}\right)\zeta+\nu_{e}^{2}\zeta\right),
\end{multline*}
one can compute \eqref{du.p01} by  \eqref{Pj}  as
\begin{equation*}
\begin{aligned}
&\frac{ik}{|k|^{2}}\left(P_{j}^{1}(ik)-P_{j}^{2}(ik)\right)\hat{U}_{\parallel
0}(k)^{T}\\
=&\frac{1}{P_{j}^{\mathrm{den}}}\frac{ik}{|k|^{2}}\left((g^{j}_{11}-g^{j}_{21})\hat{\rho}_{i0}+(g^{j}_{12}-g^{j}_{22})\hat{\rho}_{e0}
+(g^{j}_{13}-g^{j}_{23})\hat{u}_{i0,\parallel}+(g^{j}_{14}-g^{j}_{24})\hat{u}_{e0,\parallel}\right)\\
=&-\frac{1}{P_{j}^{\mathrm{den}}}4\pi\left(\frac{e^{2}}{m_{i}}+\frac{e^{2}}{m_{e}}\right)\sigma_{j}\frac{ik}{|k|^{2}}(\hat{\rho}_{i0}-\hat{\rho}_{e0})+O(1)\left|\hat{U}_{\parallel
0}(k)\right|\\
=& \frac{1}{P_{j}^{\mathrm{den}}}\left(\frac{e}{m_{i}}+\frac{e
}{m_{e}}\right)\sigma_{j}\hat{E}_{\parallel
0}(k)+O(1)\left|\hat{U}_{\parallel 0}(k)\right|,
\end{aligned}
\end{equation*}
which is bounded when $|k|\rightarrow 0$. Then the  expressions of
$\hat{E}_{\parallel}(\zeta)$ and $\hat{\bar{E}}_{\parallel}(\zeta)$
imply that
\begin{eqnarray*}
&&\begin{aligned}
  |\hat{E}(\zeta)-\hat{\bar{E}}_{\parallel}(\zeta)|
=&\exp(\lambda_{1}(ik))\bar{P}^{3}\hat{U}_{\parallel0}^{T}-
\exp{\left(-\frac{T_{i}+T_{e}}{m_{i}\nu_{i}+m_{e}\nu_{e}}|k|^{2}t\right)}\bar{P}^{3}\hat{U}_{\parallel
0}(k)^{T}\\
&+O(|k|^{2})\exp(\lambda_{1}(ik))\left|\hat{U}_{\parallel
0}(k)\right|\\
&-4\pi
e\frac{ik}{|k|^{2}}\sum_{j=2}^{4}\exp{(\lambda_{j}(ik)t)}\left(P_{j}^{1}(ik)-P_{j}^{2}(ik)\right)\hat{U}_{\parallel
0}(k)^{T}\\
\leq & C
|k|^{2}\exp{\left(-\lambda|k|^{2}t\right)}\left|\hat{U}_{\parallel
0}(k)\right|+C\exp{\left(-\lambda
t\right)}\left(\left|\hat{U}_{\parallel
0}(k)\right|+\left|\hat{E}_{\parallel 0}(k)\right|\right).
\end{aligned}
\end{eqnarray*}
This proves \eqref{error3} and then completes the proof of Lemma \ref{errorL}.
\end{proof}

Next, we consider the properties of $\hat{\rho}_{\alpha}(\zeta)$,
$\hat{u}_{\alpha,\parallel}(\zeta)$ and $\hat{E}_{\parallel}(\zeta)$
as  $|k|\rightarrow \infty$. It follows from \eqref{s4.j} that
\begin{eqnarray}\label{s4.je}
   |\hat{U}(t,k)|\leq \left\{\begin{aligned}
   &C {\exp({- \lambda|k|^{2}t })}|\hat{U}_{0}(k)|,\ \ \ \ \ \ |k|\leq r_{0},\\
   &C {\exp({- \lambda|k|^{-2}t )}}|\hat{U}_{0}(k)|,\ \ \ \ \ |k|\geq
   r_{0}.
\end{aligned}
\right.
\end{eqnarray}
Here $r_{0}$ is defined in Lemma \ref{errorL}. Combining
\eqref{s4.je} with \eqref{rhobar}, \eqref{ubar} and \eqref{Ebar}, we
have the following pointwise estimate for the error terms
$\hat{\rho}_{\alpha}(k)-\hat{\bar{\rho}}(k)$,
$\hat{u}_{\alpha,\parallel}(k)-\hat{\bar{u}}(k)$ and
$\hat{E}_{\parallel}(k)-\hat{\bar{E}}(k)$ as $|k|\rightarrow
\infty$.

\begin{lemma}\label{errorL2}
Let  $r_{0}>0$ be given in Lemma \ref{errorL}. For $|k|\geq r_0$ and $t\geq 0$, the error $|U_\parallel-\bar{U}_\parallel|$ can be bounded as
\begin{eqnarray*}
&&|\hat{\rho}_{\alpha}(t,k)-\hat{\bar{\rho}}(t,k)|\leq C
 \exp{\left(-\lambda|k|^{-2}t\right)}\left|\hat{U}_{
0}(k)\right|+C\exp{\left(-\lambda t\right)}\left|[\hat{\rho}_{i
0}(k),\hat{\rho}_{e 0}(k)]\right|,
\\
&&|\hat{u}_{\alpha,\parallel}(t,k)-\hat{\bar{u}}_{\parallel}(t,k)|\leq C
\exp{\left(-\lambda|k|^{-2}t\right)}\left|\hat{U}_{
0}(k)\right|+C\exp{\left(-\lambda t\right)} |k|\left|[\hat{\rho}_{i
0}(k),\hat{\rho}_{e 0}(k)]\right|,
\\
&&|\hat{E}_{\parallel}(t,k)-\hat{\bar{E}}_{\parallel}(t,k)|\leq  C
\exp{\left(-\lambda|k|^{-2}t\right)}\left|\hat{U}_{
0}(k)\right|+C\exp{\left(-\lambda t\right)} |k|\left|[\hat{\rho}_{i
0}(k),\hat{\rho}_{e 0}(k)]\right|,
\end{eqnarray*}
where $C$ and $\lambda$ are positive constants.
\end{lemma}

Based on Lemma \ref{errorL} and \ref{errorL2} together with \cite[Theorem 4.3]{Duan}, the time-decay
properties for the difference terms $\rho_{\alpha}-\bar{\rho}$,
$u_{\alpha,\parallel}-\bar{u}_{\parallel}$ and
$E_{\parallel}-\bar{E}_{\parallel}$ are stated as follows.

\begin{theorem}\label{thm.decaypar}
Let $1\leq p,r\leq 2\leq q\leq \infty$, $ \ell \geq 0$, and let $m\geq
1$ be an integer. Suppose that
$[\rho_{\alpha},u_{\alpha,\parallel}]$ is the solution
to the Cauchy problem \eqref{Fluid2}-\eqref{Fluid2I}. Then
$U_{\parallel}=[\rho_{\alpha},u_{\alpha,\parallel}]$ and
$E_{\parallel}$ satisfy the following time-decay property:
\begin{multline*}
\|\nabla ^{m}(\rho_{\alpha}(t)-\bar{\rho}(t))\|_{L^q}\leq
C(1+t)^{-\frac{3}{2}(\frac{1}{p}-\frac{1}{q})-\frac{m+1}{2}}\|U_{0}\|_{L^{p}}+C{\exp({-\lambda
t})}\|U_{0}\|_{L^{p}}\\
+C (1+t)^{-\frac{\ell}{2}}
  \|\nabla^{m+[\ell+3(\frac{1}{r}-\frac{1}{q})]_+}U_{0}\|_{L^r}+C{\exp({-\lambda
t})}\|\nabla^{m+[3(\frac{1}{r}-\frac{1}{q})]_+}[\rho_{i0},\rho_{e0}]\|_{L^{r}},
\end{multline*}
\begin{multline*}
\|\nabla
^{m}(u_{\alpha,\parallel}(t)-\bar{u}_{\parallel}(t))\|_{L^q}\leq
C(1+t)^{-\frac{3}{2}(\frac{1}{p}-\frac{1}{q})-\frac{m+2}{2}}\|U_{0}\|_{L^{p}}+C{\exp({-\lambda
t})}\|U_{0}\|_{L^{p}}\\
+C (1+t)^{-\frac{\ell}{2}}
  \|\nabla^{m+[\ell+3(\frac{1}{r}-\frac{1}{q})]_+}U_{0}\|_{L^r}+C{\exp{(-\lambda
t})}\|\nabla^{m+1+[3(\frac{1}{r}-\frac{1}{q})]_+}[\rho_{i0},\rho_{e0}]\|_{L^{r}},
\end{multline*}
and
\begin{multline*}
\|\nabla ^{m}(E_{\parallel}(t)-\bar{E}_{\parallel}(t))\|_{L^q}\leq
C(1+t)^{-\frac{3}{2}(\frac{1}{p}-\frac{1}{q})-\frac{m+2}{2}}\|U_{0}\|_{L^{p}}+C{\exp({-\lambda
t})}\|U_{0}\|_{L^{p}}\\
+C (1+t)^{-\frac{\ell}{2}}
  \|\nabla^{m+[\ell+3(\frac{1}{r}-\frac{1}{q})]_+}U_{0}\|_{L^r}+C{\exp({-\lambda
t})}\|\nabla^{m+1+[3(\frac{1}{r}-\frac{1}{q})]_+}[\rho_{i0},\rho_{e0}]\|_{L^{r}},
\end{multline*}
for any $t\geq 0$, where  $C=C(m,p,r,q,\ell)$ and
$[\ell+3(\tfrac{1}{r}-\tfrac{1}{q})]_+$ is defined in
\eqref{thm.decay.1}.
\end{theorem}

\subsection{Spectral representation for electromagnetic part}

\subsubsection{Asymptotic expansions and expressions for $B$}

 Taking the curl for the equations of
$\partial_t u_{i,\perp}$, $\partial_t u_{e,\perp}$, $\partial_t
E_{\perp}$ in \eqref{electromagnetic} and using $\Delta
B=-\nabla\times (\nabla \times B)$, it follows that
\begin{equation}\label{electrcurl}
\left\{
  \begin{aligned}
  &m_{i}\partial_t(\nabla\times u_{i,\perp})-e\nabla\times E_{\perp}+m_{i}\nu_{i}(\nabla\times u_{i,\perp})=0,\\
  &m_{e}\partial_t (\nabla\times u_{e,\perp})+e\nabla\times E_{\perp}+m_{e}\nu_{e}(\nabla\times u_{e,\perp})=0,\\
  &\partial_t (\nabla\times E_{\perp})+c\Delta B+4\pi e(\nabla\times u_{i,\perp}-\nabla\times u_{e,\perp})=0,\\
  &\partial_t B+c\nabla\times E_{\perp}=0.\\
\end{aligned}\right.
\end{equation}
Taking the time derivative for the fourth equation of
\eqref{electrcurl} and then using the third equations to
replace $\partial_t (\nabla\times E_{\perp})$ gives
\begin{equation}\label{tB}
\partial_{tt}B-c^{2}\Delta B -4\pi c e(\nabla\times u_{i,\perp}-\nabla\times
u_{e,\perp})=0.
\end{equation}
Further taking the time derivative for \eqref{tB} and replacing
$\partial_t(\nabla\times u_{i,\perp})$ and $\partial_t (\nabla\times
u_{e,\perp})$ give
\begin{equation}\label{ttB}
\partial_{ttt}B-c^{2}\Delta B_{t}+4\pi\left(\frac{e^{2}}{m_{i}}+\frac{e^{2}}{m_{e}}\right)\partial_{t}B+4\pi c(e\nu_{i}\nabla\times u_{i,\perp}-e\nu_{e}\nabla\times
u_{e,\perp})=0.
\end{equation}
Here we have replaced $\nabla\times E_{\perp}$ by
$-\frac{1}{c}\partial_{t} B$. Further taking the time derivative for
\eqref{ttB} and replacing $\partial_t(\nabla\times u_{i,\perp})$ and
$\partial_t (\nabla\times u_{e,\perp})$ gives
\begin{multline}\label{tttB}
\partial_{tttt}B-c^{2}\Delta
B_{tt}+4\pi\left(\frac{e^{2}}{m_{i}}+\frac{e^{2}}{m_{e}}\right)\partial_{tt}B
-4\pi\left(\frac{e^{2}}{m_{i}}\nu_{i}+\frac{e^{2}}{m_{e}}\nu_{e}\right)\partial_{t}B\\-4\pi
c(e\nu_{i}^{2}\nabla\times u_{i,\perp}-e\nu_{e}^{2}\nabla\times
u_{e,\perp})=0.
\end{multline}
Taking the summation of $\eqref{tttB}$, $\eqref{ttB}\times
(\nu_{i}+\nu_{e})$ and $\eqref{tB}\times \nu_{i}\nu_{e}$ yields
\begin{multline*}
\partial_{tttt}B+(\nu_{i}+\nu_{e})\partial_{ttt}B-c^{2}\Delta
B_{tt}+4\pi\left(\frac{e^{2}}{m_{i}}+\frac{e^{2}}{m_{e}}\right)\partial_{tt}B+\nu_{i}\nu_{e}\partial_{tt}B\\
-c^{2}(\nu_{i}+\nu_{e})\Delta
B_{t}+4\pi\left(\frac{e^{2}}{m_{i}}\nu_{e}+\frac{e^{2}}{m_{e}}\nu_{i}\right)\partial_{t}B-\nu_{i}\nu_{e}c^{2}\Delta
B=0.
\end{multline*}
In terms of the Fourier transform in $x$ of the above equation, one
has
\begin{multline}\label{tttBF}
\partial_{tttt}\hat{B}+(\nu_{i}+\nu_{e})\partial_{ttt}\hat{B}+\left(c^{2}|k|^{2}
+\frac{4\pi e^{2}}{m_{i}}+\frac{4\pi e^{2}}{m_{e}}+\nu_{i}\nu_{e}\right)\partial_{tt}\hat{B}\\
+\left((\nu_{i}+\nu_{e})c^{2}|k|^{2}+\frac{4\pi
e^{2}}{m_{i}}\nu_{e}+\frac{4\pi
e^{2}}{m_{e}}\nu_{i}\right)\partial_{t}\hat{B}+\nu_{i}\nu_{e}|k|^{2}c^{2}
\hat{B}=0.
\end{multline}
Initial data are given as
\begin{equation}\label{MMM2i}
\left\{
  \begin{aligned}
   &\hat{B}|_{t=0}=\hat{B}_{0},\\
   &\partial_t \hat{B}|_{t=0} = -cik\times \hat{E}_{0,\perp},\\
   &\partial_{tt} \hat{B}|_{t=0} =-c^{2}|k|^{2}\hat{B}_{0}+4\pi c\left(eik\times \hat{u}_{i0,\perp}-eik\times \hat{u}_{e0,\perp}\right),\\
   &\partial_{ttt} \hat{B}|_{t=0} =\left(c^{2}|k|^{2}+4\pi\left(\frac{e^{2}}{m_{i}}+\frac{e^{2}}{m_{e}}\right)\right)cik\times \hat{E}_{0,\perp}
   -4\pi ce\nu_{i}ik\times \hat{u}_{i0,\perp}+4\pi ce\nu_{e}ik\times \hat{u}_{e0,\perp}.\\
 \end{aligned}\right.
\end{equation}
The characteristic equation of  $\eqref{tttBF}$ reads
\begin{equation*}
\begin{aligned}
\lambda^{4}+(\nu_{i}+\nu_{e})\lambda^{3}+&\left(c^{2}|k|^{2}
+4\pi\left(\frac{e^{2}}{m_{i}}+\frac{e^{2}}{m_{e}}\right)+\nu_{i}\nu_{e}\right)\lambda^{2}\\
&+\left(c^{2}(\nu_{i}+\nu_{e})|k|^{2}+4\pi\left(\frac{e^{2}}{m_{i}}\nu_{e}+\frac{e^{2}}{m_{e}}\nu_{i}\right)\right)\lambda+\nu_{i}\nu_{e}c^{2}|k|^{2}=0.
\end{aligned}
\end{equation*}
For the roots of the above characteristic equation and their basic
properties, one has
\begin{equation}\label{roots1}
  \begin{aligned}
   &\lambda_{1}(|k|)=-\dfrac{c^{2}m_{i}\nu_{i}m_{e}\nu_{e}}{4\pi e^{2}(m_{i}\nu_{i}+m_{e}\nu_{e})}|k|^{2}+O(|k|^{4}),\\
   &\lambda_{j}(|k|)=\sigma_{j}+O(|k|^{2}),\ \ \ \ \mathrm{for}\ \
   j=2,3,4,
 \end{aligned}
\end{equation}
as $|k|\rightarrow 0$. Here we note that $\sigma_j$ $(j=2,3,4)$ with $\Re\,\sigma_{j}<0$
are the solutions to $g(\lambda)=0$ with $g(\lambda)$ still defined in
\eqref{deglamda}. One can set the solution of $\eqref{tttBF}$ to be
\begin{equation}\label{sol.M2}
\hat{B}=\sum_{j=1}^{4}c_{j}(ik) \exp\{\lambda_{j}(ik)t\},
\end{equation}
where $c_{i}\ (1\leq i\leq 4)$ are to be determined by $
\eqref{MMM2i}$ later. In fact, $\eqref{sol.M2}$ implies
\begin{eqnarray}\label{determine.c}
&&M\left[
  \begin{array}{c}
   c_{1}\\
   c_{2}\\
    c_{3}\\
    c_{4}
 \end{array}\right]:=
 \left[
  \begin{array}{cccc}
    1         &  1    & 1 & 1\\
  \lambda_{1}         &    \lambda_{2}     &   \lambda_{3}   &   \lambda_{4}  \\
   \lambda_{1}^{2}        &    \lambda_{2}^{2}    &   \lambda_{3}^{2}   &   \lambda_{4}^{2}  \\
   \lambda_{1}^{3}        &    \lambda_{2}^{3}    &   \lambda_{3}^{3}   &   \lambda_{4}^{3}  \\
 \end{array}\right]
 \left[
  \begin{array}{c}
   c_{1}\\
   c_{2}\\
    c_{3}\\
    c_{4}
 \end{array}\right]=
 \left[
  \begin{array}{c}
    \hat{B}|_{t=0}\\
   \partial_t\hat{B}|_{t=0} \\
   \partial_{tt} \hat{B}|_{t=0} \\
   \partial_{ttt} \hat{B}|_{t=0}
 \end{array}\right],\end{eqnarray}
where the right-hand term is given in terms of \eqref{MMM2i} by
\begin{equation}\label{MMM2iJ}
\left[
  \begin{array}{c}
    \hat{B}|_{t=0}\\
   \partial_t\hat{B}|_{t=0} \\
   \partial_{tt} \hat{B}|_{t=0} \\
   \partial_{ttt} \hat{B}|_{t=0}
 \end{array}\right]
 =\left[
  \begin{array}{cccc}
  0       &  0    & 0 & 1\\
0       &    0     &   -c ik \times  &   0  \\
4\pi c eik\times       &    -4\pi c eik\times     &   0   &   -c^{2}|k|^{2}  \\
-4\pi c e\nu_{i}ik\times       &   4\pi c e\nu_{e}ik\times    &   \left(c^{2}|k|^{2}+4\pi\left(\frac{e^{2}}{m_{i}}+\frac{e^{2}}{m_{e}}\right)\right)cik\times     &   0  \\
 \end{array}\right]
 \left[
  \begin{array}{c}
  \hat{u}_{i0,\perp}\\
  \hat{u}_{e0,\perp}\\
  \hat{ E}_{0,\perp} \\
   \hat{B}_{0}
 \end{array}\right].
\end{equation}
It is straightforward to check that

\begin{equation*}
\det M=\prod_{1\leq j<i\leq 4}(\lambda_{i}-\lambda_{j})\neq 0,
\end{equation*}
as long as $\lambda_{j}(|k|)$ are distinct to each other, and
\begin{equation*}
M^{-1}=\dfrac{1}{\det M} \left[
  \begin{array}{cccc}
  M_{11}  &   M_{21} &   M_{31} &    M_{41}  \\
    M_{12}  &   M_{22}  &    M_{32}  &   M_{42}  \\
    M_{13}  &   M_{23}  &    M_{33} &    M_{43}  \\
      M_{14}  &   M_{24} &   M_{34}  &    M_{44}
 \end{array}\right],
\end{equation*}
where $M_{ij}$ is the corresponding algebraic complement of $M$.
Notice that $\eqref{determine.c}$ together with $ \eqref{MMM2iJ}$
give
\begin{eqnarray*}
\left[
  \begin{array}{c}
   c_{1}\\
   c_{2}\\
    c_{3}\\
    c_{4}
 \end{array}\right]\
 = M^{-1}
 \left[
  \begin{array}{cccc}
  0       &  0    & 0 & 1\\
0       &    0     &   -c ik \times  &   0  \\
4\pi c eik\times       &    -4\pi c eik\times     &   0   &   -c^{2}|k|^{2}  \\
-4\pi c e\nu_{i}ik\times       &   4\pi c e\nu_{e}ik\times    &   \left(c^{2}|k|^{2}+4\pi\left(\frac{e^{2}}{m_{i}}+\frac{e^{2}}{m_{e}}\right)\right)cik\times     &   0  \\
 \end{array}\right]
 \left[
  \begin{array}{c}
  \hat{u}_{i0,\perp}\\
  \hat{u}_{e0,\perp}\\
  \hat{ E}_{0,\perp} \\
   \hat{B}_{0}
 \end{array}\right],
\end{eqnarray*}
which after plugging $ M^{-1}$, implies
\begin{equation*}
\begin{aligned}
c_{1}=&\frac{1}{\prod\limits_{1\leq j<i\leq
4}(\lambda_{i}-\lambda_{j})}\left[(4\pi c eM_{31}-4\pi
ceM_{41}\nu_{i})ik\times
\hat{u}_{i0,\perp}+(-4\pi c eM_{31}+4\pi c eM_{41}\nu_{e})ik\times\hat{u}_{e0,\perp}\right.\\
&\left.+\left(-M_{21}+M_{41}\left(c^{2}|k|^{2}+4\pi
\left(\frac{e^{2}}{m_{i}}+\frac{e^{2}}{m_{e}}\right)\right)\right)cik\times
 \hat{ E}_{0,\perp}+\left(M_{11}-c^{2}|k|^{2}M_{31}\right)\hat{B}_{0}\right]\\
 =&\frac{M_{11}\hat{B}_{0}}{\prod\limits_{1\leq j<i\leq
4}(\lambda_{i}-\lambda_{j})}+O(|k|)|\hat{U}_{0,\perp}|.
\end{aligned}
\end{equation*}
We deduce that $\frac{M_{11}}{\prod\limits_{1\leq j<i\leq
4}(\lambda_{i}-\lambda_{j})}$ has the following asymptotic
expansion as $|k|\rightarrow 0$:
\begin{equation*}
\frac{M_{11}}{\prod\limits_{1\leq j<i\leq
4}(\lambda_{i}-\lambda_{j})}=\sum_{\ell=0}^{+\infty}c_{1}^{\ell}|k|^{\ell},
\end{equation*}
where
\begin{equation*}
M_{11}=\left|
  \begin{array}{ccc}
  \lambda_{2}     &   \lambda_{3}   &   \lambda_{4}  \\
   \lambda_{2}^{2}    &   \lambda_{3}^{2}   &   \lambda_{4}^{2}  \\
  \lambda_{2}^{3}    &   \lambda_{3}^{3}   &   \lambda_{4}^{3}  \\
 \end{array}\right|= \lambda_{2}  \lambda_{3}  \lambda_{4}\prod\limits_{2\leq j<i\leq
4}(\lambda_{i}-\lambda_{j}).
\end{equation*}
By straightforward computations, $c_{1}^{0}=1$ holds true and this implies that
\begin{equation}\label{c1}
c_{1}(ik)=\hat{B}_{0}+O(|k|)|\hat{U}_{0,\perp}|[1,1,1]^{T}.
\end{equation}

\subsubsection{Error estimates}
In this section, we first give the error estimates for ${B}-\bar{B}$, and then apply the energy method in the Fourier space
to the difference problem for $\eqref{electromagnetic}$ and
$\eqref{perpprofile}$ to get the error estimates for
$u_{\alpha,\perp}-\bar{u}_{\alpha,\perp}$ and $E_{\perp}-\bar{E}_{\perp}$. It should be pointed out that it is also possible to carry out the same strenuous procedure as in the previous section to obtain the error estimates on $u_{\alpha,\perp}-\bar{u}_{\alpha,\perp}$ and $E_{\alpha,\perp}-\bar{E}_{\alpha,\perp}$. The reason why we choose the Fourier energy method is just for the simplicity of representation, since the estimates on  $u_{\alpha,\perp}-\bar{u}_{\alpha,\perp}$ and $E_{\alpha,\perp}-\bar{E}_{\alpha,\perp}$ can be directly obtained basing on the estimate on ${B}-\bar{B}$.

 \begin{lemma}\label{errorB}
 There is $r_{0}>0$ such that for $|k|\leq r_{0}$ and $t\geq 0$,
\begin{eqnarray}\label{errorB1}
|\hat{B}(t,k)-\hat{\bar{B}}(t,k)|\leq
C\left(|k|{\exp({-\lambda|k|^{2}t})+\exp({-\lambda
t})}\right)|\hat{U}_{0,\perp}|,
\end{eqnarray}
where $C$ and $\lambda$ are positive constants.
\end{lemma}

\begin{proof}
It follows from \eqref{sol.M2} and \eqref{Bsolve} that
\begin{eqnarray*}
\begin{aligned}
\hat{B}(t,k)-\hat{\bar{B}}(t,k)= &
               \sum_{j=1}^{4}c_{j}(ik) \exp\{\lambda_{j}(ik)t\}-\exp\left(-\dfrac{c^{2} m_{i}\nu_{i}m_{e}\nu_{e}}{4\pi e^{2}(m_{i}\nu_{i}+m_{e}\nu_{e})}|k|^{2}t\right)\hat{B}_{0}(k)\\
               =&(c_{1}(ik)-\hat{B}_{0}){\exp\{\lambda_{1}(ik)t\}}
               +\hat{B}_{0}\left(\exp\{\lambda_{1}(ik)t\}-\exp\left\{-\dfrac{c^{2} m_{i}\nu_{i}m_{e}\nu_{e}}{4\pi e^{2}(m_{i}\nu_{i}+m_{e}\nu_{e})}|k|^{2}t\right\}\right)\\
               &+\sum_{j=2}^{4}\exp\{\lambda_{j}(ik)t\}c_{j}(ik)\\
               :=&\hat{R}_{21}(ik)+\hat{R}_{22}(ik)+\hat{R}_{23}(ik).
\end{aligned}
\end{eqnarray*}
Using \eqref{roots1} and \eqref{c1}, one has
\begin{eqnarray*}
|\hat{R}_{21}(ik)|&\leq& C|k|
{\exp({-\lambda|k|^2t})}|\hat{U}_{0,\perp}|,\\
|\hat{R}_{22}(ik)|&\leq&
C|k|^{2}{\exp({-\lambda|k|^{2}t})}|\hat{B}_{0}|,\\
|\hat{R}_{23}(ik)|&\leq& C{\exp({-\lambda t})}|\hat{U}_{0,\perp}|.
\end{eqnarray*}
This proves $\eqref{errorB1}$ and then completes the proof of Lemma \ref{errorB}.
\end{proof}

Next, in order to get the error estimates for  $u_{\alpha,\perp}-\bar{u}_{\alpha,\perp}$ and $E_{\perp}-\bar{E}_{\perp}$, we write
\begin{eqnarray*}
\tilde{u}_{\alpha}=u_{\alpha,\perp}-\bar{u}_{\alpha,\perp},\ \
\tilde{E}=E_{\perp}-\bar{E}_{\perp},\ \ \tilde{B}=B-\bar{B}.
\end{eqnarray*}
Combining \eqref{electromagnetic} with \eqref{perpprofile}, then
$[\tilde{u}_{\alpha},\tilde{E}]$
satisfies
\begin{equation}\label{sta.equ}
\left\{
  \begin{aligned}
  &m_{\alpha}\partial_t \tilde{u}_{\alpha}-q_{\alpha}\tilde{E}+m_{\alpha}\nu_{\alpha}\tilde{u}_{\alpha}=-m_{\alpha}\partial_{t}\bar{u}_{\alpha,\perp},\\
  &\partial_t \tilde{E}-c\nabla\times \tilde{B}+4\pi\sum_{\alpha=i,e}q_{\alpha}\tilde{u}_{\alpha}=-\partial_{t}\bar{E}_{\perp}.\\
\end{aligned}\right.
\end{equation}

 \begin{lemma}\label{errorUE}
 There is $r_{0}>0$ such that
\begin{eqnarray}\label{errorU}
|\hat{u}_{\alpha,\perp}(t,k)-\hat{\bar{u}}_{\alpha,\perp}(t,k)|\leq
\left\{
  \begin{aligned}
&\left(C|k|^{2}{\exp({-\lambda|k|^{2}t})}+{\exp({-\lambda
t})}\right)|\hat{U}_{0,\perp}|,\ \ \ \mathrm{for}\ |k|\leq r_{0},\\
&  \begin{aligned} &C
 \exp{\left(-\lambda|k|^{-2}t\right)}\left|\hat{U}_{
0}(k)\right|\\
&\ \ \ \ \ \ \ \ \ +C\exp{\left(-\lambda |k|^2
t\right)}|k|\left|\hat{B}_{0}(k)\right|,\ \mathrm{for}\ |k|\geq
r_{0},
\end{aligned}
\end{aligned}
\right.
\end{eqnarray}
and
\begin{eqnarray}\label{errorE}
|\hat{E}_{\perp}(t,k)-\hat{\bar{E}}_{\perp}(t,k)|\leq \left\{
  \begin{aligned}
&\left(C|k|^{2}{\exp({-\lambda|k|^{2}t})}+{\exp({-\lambda
t})}\right)|\hat{U}_{0,\perp}|,\ \ \ \mathrm{for}\ |k|\leq r_{0},\\
&  \begin{aligned} &C
 \exp{\left(-\lambda|k|^{-2}t\right)}\left|\hat{U}_{
0}(k)\right|\\
&\ \ \ \ \ \ \ \ \ +C\exp{\left(-\lambda |k|^2
t\right)}|k|\left|\hat{B}_{0}(k)\right|,\ \mathrm{for}\ |k|\geq
r_{0},
\end{aligned}
\end{aligned}
\right.
\end{eqnarray}
where $C$ and $\lambda$ are positive constants.
\end{lemma}

\begin{proof}
It is straightforward to obtain the error estimates for
$|k|\geq r_{0}$ due to
 \eqref{s4.j}, \eqref{perpsolve} and \eqref{Bsolve}. In the case $|k|\leq r_{0}$, the desired result can follow from  the
Fourier energy estimate on the system $\eqref{sta.equ}$. Indeed, after
taking the Fourier transform in $x$, $\eqref{sta.equ}$
gives
\begin{equation}\label{sta.equF}
\left\{
  \begin{aligned}
  &m_{\alpha}\partial_t \hat{\tilde{u}}_{\alpha}-q_{\alpha}\hat{\tilde{E}}+m_{\alpha}\nu_{\alpha}\hat{\tilde{u}}_{\alpha}=
  -m_{\alpha}\partial_{t}\hat{\bar{u}}_{\alpha,\perp},\\
  &\partial_t \hat{\tilde{E}}-cik\times \hat{\tilde{B}}+4\pi\sum_{\alpha=i,e}q_{\alpha}\hat{\tilde{u}}_{\alpha}=-\partial_{t}\hat{\bar{E}}_{\perp}.\\
\end{aligned}\right.
\end{equation}
By taking the complex dot product of the first equation of
 $\eqref{sta.equF}$ with $\hat{\tilde{u}}_{\alpha}$, taking the complex dot product of the second equation of
 $\eqref{sta.equF}$ with $\hat{\tilde{E}}$, and taking the real part, one has
\begin{eqnarray*}
\begin{aligned}
   &\frac{1}{2}\frac{d}{dt}\sum_{\alpha=i,e}m_{\alpha}|\hat{\tilde{u}}_{\alpha}|^{2}+\frac{1}{4\pi}\frac{1}{2}\frac{d}{dt}|\hat{\tilde{E}}|^{2}
   +\sum_{\alpha=i,e}m_{\alpha}\nu_{\alpha}|\hat{\tilde{u}}_{\alpha}|^{2}\\
   =&-\sum_{\alpha=i,e}m_{\alpha}\mathfrak{R}(\partial_{t}\hat{\bar{u}}_{\alpha,\perp}|\hat{\tilde{u}}_{\alpha})
   -\frac{1}{4\pi}\mathfrak{R}\left(\partial_{t}\hat{\bar{E}}_{\perp}|\hat{\tilde{E}}\right)+\frac{1}{4\pi}\mathfrak{R}(cik\times \hat{\tilde{B}}|\hat{\tilde{E}}),
   \end{aligned}
\end{eqnarray*}
which by using the Cauchy-Schwarz inequality with $0<\epsilon<1$,
implies
\begin{eqnarray}\label{Bstep1}
\begin{aligned}
   &\frac{1}{2}\frac{d}{dt}\left(\sum_{\alpha=i,e}m_{\alpha}|\hat{\tilde{u}}_{\alpha}|^{2}+\frac{1}{4\pi}|\hat{\tilde{E}}|^{2}\right)
   +\sum_{\alpha=i,e}m_{\alpha}\nu_{\alpha}|\hat{\tilde{u}}_{\alpha}|^{2}\\
   \leq & \epsilon \sum_{\alpha=i,e}|\hat{\tilde{u}}_{\alpha}|^{2}+\epsilon|\hat{\tilde{E}}|^{2} +C_{\epsilon}\sum_{\alpha=i,e}|\partial_{t}\hat{\bar{u}}_{\alpha,\perp}|^{2}+ C_{\epsilon}|ik\times
   \hat{\tilde{B}}|^{2}+C_{\epsilon}|\partial_{t}\hat{\bar{E}}_{\perp}|^{2}.
   \end{aligned}
\end{eqnarray}
By taking the complex dot product of the first equation of
 $\eqref{sta.equF}$ with $-q_{\alpha}\hat{\tilde{E}}$, replacing
 $\partial_{t}\hat{\tilde{E}}$ by the second equation of
 $\eqref{sta.equF}$ and taking the real part, one has
\begin{eqnarray*}
\begin{aligned}
   &\partial_{t}\sum_{\alpha=i,e}\mathfrak{R}(m_{\alpha}\hat{\tilde{u}}_{\alpha}|-q_{\alpha}\hat{\tilde{E}})+
   \sum_{\alpha=i,e}q_{\alpha}^2\left|\hat{\tilde{E}}\right|^{2}\\
   =&-\sum_{\alpha=i,e}q_{\alpha}\mathfrak{R}(
   m_{\alpha}\hat{\tilde{u}}_{\alpha}|cik\times \hat{\tilde{B}})
   +4\pi\sum_{\alpha=i,e}q_{\alpha}\mathfrak{R}\left(m_{\alpha}\hat{\tilde{u}}_{\alpha}|\sum_{\alpha=i,e}q_{\alpha}\hat{\tilde{u}}_{\alpha}\right)\\
   &+\sum_{\alpha=i,e}q_{\alpha}\mathfrak{R}\left(m_{\alpha}\hat{\tilde{u}}_{\alpha}|\partial_{t}\hat{\bar{E}}_{\perp}\right)
   +\sum_{\alpha=i,e}\mathfrak{R}(m_{\alpha}\nu_{\alpha}\hat{\tilde{u}}_{\alpha}|q_{\alpha}\hat{\tilde{E}})
   +\sum_{\alpha=i,e}\mathfrak{R}(m_{\alpha}\partial_{t}\hat{\bar{u}}_{\alpha,\perp}|q_{\alpha}\hat{\tilde{E}}),
   \end{aligned}
\end{eqnarray*}
which by using the Cauchy-Schwarz inequality with $0<\eps<1$, implies
\begin{eqnarray}\label{Bstep2}
\begin{aligned}
   &\partial_{t}\sum_{\alpha=i,e}\mathfrak{R}(m_{\alpha}\hat{\tilde{u}}_{\alpha}|-q_{\alpha}\hat{\tilde{E}})+
   \sum_{\alpha=i,e}q_{\alpha}^2\left|\hat{\tilde{E}}\right|^{2}\\
  \leq & C_\epsilon \sum_{\alpha=i,e}|\hat{\tilde{u}}_{\alpha}|^{2}+\epsilon|\hat{\tilde{E}}|^{2} +C_{\epsilon}\sum_{\alpha=i,e}|\partial_{t}\hat{\bar{u}}_{\alpha,\perp}|^{2}+ C_{\epsilon}|ik\times
   \hat{\tilde{B}}|^{2}+C_{\epsilon}|\partial_{t}\hat{\bar{E}}_{\perp}|^{2},
\end{aligned}
\end{eqnarray}
for $0<\eps<1$. We now define
 \begin{equation*}
   \mathcal{E}(t)=\sum_{\alpha=i,e}m_{\alpha}|\hat{\tilde{u}}_{\alpha}|^{2}+\frac{1}{4\pi}|\hat{\tilde{E}}|^{2}+
   \kappa \sum_{\alpha=i,e}\mathfrak{R}(m_{\alpha}\hat{\tilde{u}}_{\alpha}|-q_{\alpha}\hat{\tilde{E}}),
\end{equation*}
for a constant  $0<\kappa \ll 1$ to be determined. Notice that as long
as $ 0<\kappa\ll 1$ is small enough, then
\begin{equation}
\label{du.p02}
\mathcal{E}(t)\sim
\sum\limits_{\alpha=i,e}|\hat{\tilde{u}}_{\alpha}|^{2}+|\hat{\tilde{E}}|^{2}
\end{equation}
holds true.  On the other hand, the sum of $\eqref{Bstep1}$ and $\eqref{Bstep2}\times
\kappa$ gives
\begin{multline}\label{DJF11}
\partial_{t}\mathcal{E}(t)+\lambda\left(\sum\limits_{\alpha=i,e}|\hat{\tilde{u}}_{\alpha}|^{2}+|\hat{\tilde{E}}|^{2}\right)
   \leq C\sum_{\alpha=i,e}|\partial_{t}\hat{\bar{u}}_{\alpha,\perp}|^{2}+ C|ik\times
   \hat{\tilde{B}}|^{2}+C|\partial_{t}\hat{\bar{E}}_{\perp}|^{2}\\
   \leq   C|k|^{6}{\exp\{-2\lambda
   |k|^{2}t\}}\left|\hat{\bar{B}}_{0}\right|^{2}+C|k|^{2}\left(C|k|^{2}{\exp\{-2\lambda|k|^{2}t\}}+{\exp\{-2\lambda
t\}}\right)|\hat{U}_{0,\perp}|^{2}\\
\leq
C|k|^{4}{\exp\{-2\lambda|k|^{2}t\}}|\hat{U}_{0,\perp}|^{2}+C|k|^{2}{\exp\{-2\lambda
t\}}|\hat{U}_{0,\perp}|^{2},
\end{multline}
for $|k|\leq r_{0}$, where we have used the expressions of
$\bar{u}_{\alpha,\perp}$, $ \bar{E}_{\alpha,\perp}$ in
\eqref{perpsolve}, the expression of $\hat{\bar{B}}$ in
$\eqref{Bsolve}$ and Lemma \ref{errorB}. Multiplying \eqref{DJF11}
by $\exp({\lambda t})$ and integrating the resulting inequality over
$(0,t)$ yield that
\begin{multline}\label{conclule}
\mathcal{E}(t)\leq {\exp(-{\lambda
t})}\left(\sum\limits_{\alpha=i,e}|\hat{\tilde{u}}_{\alpha0}|^{2}+|\hat{\tilde{E}}_{0}|^{2}\right)\\
+C{\exp({-\lambda t})}\int_{0}^{t}{\exp({\lambda
s})}\left(|k|^{4}{\exp({-2\lambda|k|^{2}s})}|\hat{U}_{0,\perp}|^{2}+C|k|^{2}{\exp({-2\lambda
s})}|\hat{U}_{0,\perp}|^{2}\right)ds\\
\leq C{\exp({-\lambda
t})}|\hat{U}_{0,\perp}|^{2}+C|k|^{4}{\exp({-2\lambda|k|^{2}t})}|\hat{U}_{0,\perp}|^{2}.
\end{multline}
Therefore, \eqref{errorU} and $\eqref{errorE}$ follows from
\eqref{conclule} by noticing \eqref{du.p02}.
This then completes the proof of Lemma \ref{errorUE}.
\end{proof}

From Lemma \ref{errorUE} together with \cite[Theorem 4.3]{Duan}, one has

\begin{theorem}\label{thm.decayperp}
Let $1\leq p,r\leq 2\leq q\leq \infty$, $ \ell \geq 0$, and let $m\geq
1$ be an integer. Suppose that $U_{\perp}=[u_{\alpha,\perp},E_{\perp},B]$ is the
solution to the Cauchy problem
\eqref{electromagnetic}-\eqref{electromagnetic2I}. Then
one has the following
time-decay property:
\begin{multline*}
\|\nabla
^{m}(u_{\alpha,\perp}(t)-\bar{u}_{\alpha,\perp}(t))\|_{L^q}\leq
C(1+t)^{-\frac{3}{2}(\frac{1}{p}-\frac{1}{q})-\frac{m+2}{2}}\|U_{0}\|_{L^{p}}+C{\exp({-\lambda
t})}\|U_{0}\|_{L^{p}}\\
+C (1+t)^{-\frac{\ell}{2}}
  \|\nabla^{m+[\ell+3(\frac{1}{r}-\frac{1}{q})]_+}U_{0}\|_{L^r}+C{\exp({-\lambda
t})}\|\nabla^{m+1+[3(\frac{1}{r}-\frac{1}{q})]_+}B_{0}\|_{L^{r}},
\end{multline*}
\begin{multline*}
\|\nabla ^{m}(E_{\perp}(t)-\bar{E}_{\perp}(t))\|_{L^q}\leq
C(1+t)^{-\frac{3}{2}(\frac{1}{p}-\frac{1}{q})-\frac{m+2}{2}}\|U_{0}\|_{L^{p}}+C{\exp({-\lambda
t})}\|U_{0}\|_{L^{p}}\\
+C (1+t)^{-\frac{\ell}{2}}
  \|\nabla^{m+[\ell+3(\frac{1}{r}-\frac{1}{q})]_+}U_{0}\|_{L^r}+C{\exp({-\lambda
t})}\|\nabla^{m+1+[3(\frac{1}{r}-\frac{1}{q})]_+}B_{0}\|_{L^{r}},
\end{multline*}
and
\begin{multline*}
\|\nabla ^{m}(B(t)-\bar{B}(t))\|_{L^q}\leq
C(1+t)^{-\frac{3}{2}(\frac{1}{p}-\frac{1}{q})-\frac{m+1}{2}}\|U_{0}\|_{L^{p}}+C{\exp({-\lambda
t})}\|U_{0}\|_{L^{p}}\\
+C (1+t)^{-\frac{\ell}{2}}
  \|\nabla^{m+[\ell+3(\frac{1}{r}-\frac{1}{q})]_+}U_{0}\|_{L^r}+C{\exp({-\lambda
t})}\|\nabla^{m+[3(\frac{1}{r}-\frac{1}{q})]_+}B_{0}\|_{L^{r}},
\end{multline*}
for any $t\geq 0$, where  $C=C(m,p,r,q,\ell)$ and
$[\ell+3(\tfrac{1}{r}-\tfrac{1}{q})]_+$ is defined in
\eqref{thm.decay.1}.
\end{theorem}

We now define the expected time-asymptotic profile of
$[\rho_{\alpha},u_{\alpha},E,B]$ to be
$[\bar{\rho},\bar{u}_{\alpha},\bar{E},\bar{B}]$, where $\bar{\rho}$ and $\bar{B}$ are diffusion waves, and
$[\bar{u}_{\alpha},\bar{E}]$ is given by
\begin{eqnarray*}
\bar{u}_{\alpha}=\bar{u}_{\parallel}+\bar{u}_{\alpha,\perp},\ \ \ \
\ \bar{E}=\bar{E}_{\parallel}+\bar{E}_{\perp}.
\end{eqnarray*}
Combining Theorem $\ref{thm.decaypar}$ with Theorem
$\ref{thm.decayperp}$, one has

\begin{corollary}\label{corollary.decay}
Let $1\leq p,r\leq 2\leq q\leq \infty$, $ \ell \geq 0$, and let $m\geq
1$ be an integer. Suppose that $U(t)=e^{tL}U_{0}$ is the solution to the
Cauchy problem $\eqref{Linear}$-$\eqref{LNC}$ with initial data
$U_{0}=[\rho_{\alpha0},u_{\alpha0},E_{0},B_{0}] $ satisfying
\eqref{LNC}. Then $U=[\rho_{\alpha},u_{\alpha},E,B]$ satisfies the
following time-decay property:
\begin{multline*}
\|\nabla ^{m}(\rho_{\alpha}(t)-\bar{\rho}(t))\|_{L^q}\leq
C(1+t)^{-\frac{3}{2}(\frac{1}{p}-\frac{1}{q})-\frac{m+1}{2}}\|U_{0}\|_{L^{p}}+C{\exp({-\lambda
t})}\|U_{0}\|_{L^{p}}\\
+C (1+t)^{-\frac{\ell}{2}}
  \|\nabla^{m+[\ell+3(\frac{1}{r}-\frac{1}{q})]_+}U_{0}\|_{L^r}+C{\exp({-\lambda
t})}\|\nabla^{m+[3(\frac{1}{r}-\frac{1}{q})]_+}[\rho_{i0},\rho_{e0}]\|_{L^{r}},
\end{multline*}
\begin{multline*}
\|\nabla ^{m}(u_{\alpha}(t)-\bar{u}_{\alpha}(t))\|_{L^q}\leq
C(1+t)^{-\frac{3}{2}(\frac{1}{p}-\frac{1}{q})-\frac{m+2}{2}}\|U_{0}\|_{L^{p}}+C{\exp({-\lambda
t})}\|U_{0}\|_{L^{p}}\\
+C (1+t)^{-\frac{\ell}{2}}
  \|\nabla^{m+[\ell+3(\frac{1}{r}-\frac{1}{q})]_+}U_{0}\|_{L^r}+C{\exp({-\lambda
t})}\|\nabla^{m+1+[3(\frac{1}{r}-\frac{1}{q})]_+}[\rho_{i0},\rho_{e0},B_{0}]\|_{L^{r}},
\end{multline*}
\begin{multline*}
\|\nabla ^{m}(E(t)-\bar{E}(t))\|_{L^q}\leq
C(1+t)^{-\frac{3}{2}(\frac{1}{p}-\frac{1}{q})-\frac{m+2}{2}}\|U_{0}\|_{L^{p}}+C{\exp({-\lambda
t})}\|U_{0}\|_{L^{p}}\\
+C (1+t)^{-\frac{\ell}{2}}
  \|\nabla^{m+[\ell+3(\frac{1}{r}-\frac{1}{q})]_+}U_{0}\|_{L^r}+C{\exp({-\lambda
t})}\|\nabla^{m+1+[3(\frac{1}{r}-\frac{1}{q})]_+}[\rho_{i0},\rho_{e0},B_{0}]\|_{L^{r}},
\end{multline*}
and
\begin{multline*}
\|\nabla ^{m}(B(t)-\bar{B}(t))\|_{L^q}\leq
C(1+t)^{-\frac{3}{2}(\frac{1}{p}-\frac{1}{q})-\frac{m+1}{2}}\|U_{0}\|_{L^{p}}+C{\exp({-\lambda
t})}\|U_{0}\|_{L^{p}}\\
+C (1+t)^{-\frac{\ell}{2}}
  \|\nabla^{m+[\ell+3(\frac{1}{r}-\frac{1}{q})]_+}U_{0}\|_{L^r}+C{\exp({-\lambda
t})}\|\nabla^{m+[3(\frac{1}{r}-\frac{1}{q})]_+}B_{0}\|_{L^{r}},
\end{multline*}
for any $t\geq 0$, where  $C=C(m,p,r,q,\ell)$ and
$[\ell+3(\tfrac{1}{r}-\tfrac{1}{q})]_+$ is defined in
\eqref{thm.decay.1}.
\end{corollary}

\begin{corollary}\label{corollary.decayL}
Under the same assumptions of Corollary \ref{corollary.decay}, it holds that
\begin{multline*}
\|\nabla ^{m}\rho_{\alpha}(t)\|_{L^q}\leq
C(1+t)^{-\frac{3}{2}(\frac{1}{p}-\frac{1}{q})-\frac{m+1}{2}}\|U_{0}\|_{L^{p}}+C{\exp({-\lambda
t})}\|U_{0}\|_{L^{p}}\\
+C (1+t)^{-\frac{\ell}{2}}
  \|\nabla^{m+[\ell+3(\frac{1}{r}-\frac{1}{q})]_+}U_{0}\|_{L^r}+C{\exp({-\lambda
t})}\|\nabla^{m+[3(\frac{1}{r}-\frac{1}{q})]_+}[\rho_{i0},\rho_{e0}]\|_{L^{r}}\\
+C(1+t)^{-\frac{3}{2}(\frac{1}{p}-\frac{1}{q})-\frac{m}{2}}\|[\rho_{i0},\rho_{e0}]\|_{L^{p}},
\end{multline*}
\begin{multline*}
\|\nabla ^{m}u_{\alpha}(t)\|_{L^q}\leq
C(1+t)^{-\frac{3}{2}(\frac{1}{p}-\frac{1}{q})-\frac{m+2}{2}}\|U_{0}\|_{L^{p}}+C{\exp({-\lambda
t})}\|U_{0}\|_{L^{p}}\\
+C (1+t)^{-\frac{\ell}{2}}
  \|\nabla^{m+[\ell+3(\frac{1}{r}-\frac{1}{q})]_+}U_{0}\|_{L^r}+C{\exp({-\lambda
t})}\|\nabla^{m+1+[3(\frac{1}{r}-\frac{1}{q})]_+}[\rho_{i0},\rho_{e0},B_{0}]\|_{L^{r}}\\
+C(1+t)^{-\frac{3}{2}(\frac{1}{p}-\frac{1}{q})-\frac{m+1}{2}}\|[\rho_{i0},\rho_{e0},B_{0}]\|_{L^{p}},
\end{multline*}
\begin{multline*}
\|\nabla^{m}E(t)\|_{L^q}\leq
C(1+t)^{-\frac{3}{2}(\frac{1}{p}-\frac{1}{q})-\frac{m+2}{2}}\|U_{0}\|_{L^{p}}+C{\exp({-\lambda
t})}\|U_{0}\|_{L^{p}}\\
+C (1+t)^{-\frac{\ell}{2}}
  \|\nabla^{m+[\ell+3(\frac{1}{r}-\frac{1}{q})]_+}U_{0}\|_{L^r}+C{\exp({-\lambda
t})}\|\nabla^{m+1+[3(\frac{1}{r}-\frac{1}{q})]_+}[\rho_{i0},\rho_{e0},B_{0}]\|_{L^{r}}\\
+C(1+t)^{-\frac{3}{2}(\frac{1}{p}-\frac{1}{q})-\frac{m+1}{2}}\|[\rho_{i0},\rho_{e0},B_{0}]\|_{L^{p}},
\end{multline*}
and
\begin{multline*}
\|\nabla ^{m}B(t)\|_{L^q}\leq
C(1+t)^{-\frac{3}{2}(\frac{1}{p}-\frac{1}{q})-\frac{m+1}{2}}\|U_{0}\|_{L^{p}}+C{\exp({-\lambda
t})}\|U_{0}\|_{L^{p}}\\
+C (1+t)^{-\frac{\ell}{2}}
  \|\nabla^{m+[\ell+3(\frac{1}{r}-\frac{1}{q})]_+}U_{0}\|_{L^r}+C{\exp({-\lambda
t})}\|\nabla^{m+[3(\frac{1}{r}-\frac{1}{q})]_+}B_{0}\|_{L^{r}}\\
+C(1+t)^{-\frac{3}{2}(\frac{1}{p}-\frac{1}{q})-\frac{m}{2}}\|B_{0}\|_{L^{p}},
\end{multline*}
for any $t\geq 0$, where  $C=C(m,p,r,q,\ell)$ and
$[\ell+3(\tfrac{1}{r}-\tfrac{1}{q})]_+$ is defined in
\eqref{thm.decay.1}.
\end{corollary}

\subsection{Extra time-decay for special initial
data}\label{sec.spec.}

Recall that the solution
$U=[\rho_{\alpha},u_{\alpha},E,B]$ to the Cauchy problem
\eqref{2.5}-\eqref{NI} with initial data
$U_{0}=[\rho_{\alpha0},u_{\alpha0},E_{0},B_{0}]$ satisfying
$\eqref{NC}$ can be formally written as
\begin{eqnarray}\label{sec5.U}
\begin{aligned}
U(t)=&e^{tL}U_{0}+\int_{0}^{t}e^{(t-s)L}[g_{1\alpha}(s),g_{2\alpha}(s),g_{3}(s),0]d
s\\
=&e^{tL}U_{0}+\int_{0}^{t}e^{(t-s)L}[\nabla \cdot
f_{\alpha}(s),g_{2\alpha}(s),g_{3}(s),0]d s,
\end{aligned}
\end{eqnarray}
where $e^{tL}$ is the linearized solution operator. We expect that the
nonlinear Cauchy problem \eqref{2.5}-\eqref{NC} can be approximated
by the corresponding linearized problem
$\eqref{Linear}$-$\eqref{LNC}$ in large time with a faster time-rate, namely the difference $U(t)-e^{tL}U_{0}$ should decay in time faster than both $U(t)$ and $e^{tL}U_{0}$.  Therefore the nonlinear term
$$
\int_{0}^{t}e^{(t-s)L}[\nabla \cdot
f_{\alpha}(s),g_{2\alpha}(s),g_{3}(s),0]d s
$$
is expected to decay in time with an extra time rate.
For this purpose, let's consider the linearized problem
$\eqref{Linear}$ with the following initial data in the special form:
\begin{eqnarray}\label{specialI}
N_{0}:=[\nabla \cdot f_{\alpha},g_{2\alpha},g_{3},0]|_{t=0}.
\end{eqnarray}
Notice that the diffusion wave $[\bar{\rho},\bar{u}_{\parallel},
\bar{E}_{\parallel}]$ given by \eqref{sol.rhob} with the corresponding initial data
\begin{eqnarray}\label{special1I}
\begin{aligned}
\bar{\rho}|_{t=0}=&\frac{m_{i}\nu_{i}}{m_{i}\nu_{i}+m_{e}\nu_{e}}\nabla
\cdot f_{i0}+\frac{m_{e}\nu_{e}}{m_{i}\nu_{i}+m_{e}\nu_{e}}\nabla
\cdot f_{e0}\\
=&\nabla\cdot\left[\frac{m_{i}\nu_{i}}{m_{i}\nu_{i}+m_{e}\nu_{e}}
f_{i0}+\frac{m_{e}\nu_{e}}{m_{i}\nu_{i}+m_{e}\nu_{e}}
f_{e0}\right],\\
\bar{u}_{\parallel}|_{t=0}=&-\dfrac{T_{i}+T_{e}}{m_{i}\nu_{i}+m_{e}\nu_{e}}\nabla
\bar{\rho}|_{t=0}\\
=&-\dfrac{T_{i}+T_{e}}{m_{i}\nu_{i}+m_{e}\nu_{e}}\nabla\nabla\cdot\left[\frac{m_{i}\nu_{i}}{m_{i}\nu_{i}+m_{e}\nu_{e}}
f_{i0}+\frac{m_{e}\nu_{e}}{m_{i}\nu_{i}+m_{e}\nu_{e}}
f_{e0}\right],\\
\bar{E}_{\parallel}|_{t=0}=&\dfrac{T_{i}m_{e}\nu_{e}-T_{e}m_{i}\nu_{i}}{e(m_{i}\nu_{i}+m_{e}\nu_{e})}\nabla
\nabla\cdot\left[\frac{m_{i}\nu_{i}}{m_{i}\nu_{i}+m_{e}\nu_{e}}
f_{i0}+\frac{m_{e}\nu_{e}}{m_{i}\nu_{i}+m_{e}\nu_{e}} f_{e0}\right],
\end{aligned}
\end{eqnarray}
should have the following  $L^p$-$L^q$ time-decay property:
\begin{equation}\label{lplqs}
\begin{aligned} \|\bar{\rho}\|_{L^{q}}&\leq
C(1+t)^{-\frac{3}{2}(\frac{1}{p}-\frac{1}{q})-\frac{m+1}{2}}\|[f_{i0},f_{e0}]\|_{L^{p}}+C{\exp({-\lambda
t})}\|\nabla^{m+1+[3(\frac{1}{r}-\frac{1}{q})]_+}[f_{i0},f_{e0}]\|_{L^{r}},\\
\left\|[\bar{u}_{\parallel},\bar{E}_{\parallel}]\right\|_{L^{q}}&\leq
C(1+t)^{-\frac{3}{2}(\frac{1}{p}-\frac{1}{q})-\frac{m+2}{2}}\|[f_{i0},f_{e0}]\|_{L^{p}}+C{\exp({-\lambda
t})}\|\nabla^{m+2+[3(\frac{1}{r}-\frac{1}{q})]_+}[f_{i0},f_{e0}]\|_{L^{r}},
\end{aligned}
\end{equation}
where the indices are chosen as in Theorem \ref{thm.decayperp}.
On the other hand, the solution
$[\bar{u}_{\alpha,\perp},\bar{E}_{\perp},\bar{B}]$ to
\eqref{perpsolve} with special initial data $\bar{B}|_{t=0}=0$
corresponding to \eqref{specialI} must be zero. i.e.,
\begin{eqnarray*}
\bar{u}_{\alpha,\perp}=0,\ \ \ \bar{E}_{\perp}=0,\ \ \ \bar{B}=0.
\end{eqnarray*}

Based on  $L^p$-$L^q$ time-decay property \eqref{lplqs} of diffusion
wave $[\bar{\rho},\bar{u}_{\parallel}, \bar{E}_{\parallel}]$ with
special initial data  \eqref{special1I} and Corollary
\ref{corollary.decay}, we obtain the extra time-decay for the solution
to the linearized problem \eqref{Linear} with special initial data
\eqref{special1I} in the following

\begin{theorem}\label{thm.special}
Let $1\leq p,r\leq 2\leq q\leq \infty$, $ \ell \geq 0$, and let $m\geq
1$ be an integer. Suppose that $e^{tL}N_{0}$ is the solution to the
Cauchy problem $\eqref{Linear}$ with initial data $N_{0}=[\nabla
\cdot f_{\alpha},g_{2\alpha},g_{3},0]|_{t=0}$ satisfying
\eqref{LNC}. Then
one has the following time-decay property:
\begin{multline*}
\|\nabla ^{m}\FP_{1\al} e^{tL}N_{0}\|_{L^q}\leq
C(1+t)^{-\frac{3}{2}(\frac{1}{p}-\frac{1}{q})-\frac{m+1}{2}}\|N_{0}\|_{L^{p}}+C{\exp({-\lambda
t})}\|N_{0}\|_{L^{p}}\\
+C (1+t)^{-\frac{\ell}{2}}
  \|\nabla^{m+[\ell+3(\frac{1}{r}-\frac{1}{q})]_+}N_{0}\|_{L^r}+C{\exp({-\lambda
t})}\|\nabla^{m+1+[3(\frac{1}{r}-\frac{1}{q})]_+}[f_{i0},f_{e0}]\|_{L^{r}}\\
+C(1+t)^{-\frac{3}{2}(\frac{1}{p}-\frac{1}{q})-\frac{m+1}{2}}\|[f_{i0},f_{e0}]\|_{L^{p}}.
\end{multline*}
\begin{multline*}
\|\nabla ^{m} \FP_{2\al}e^{tL }N_{0}\|_{L^q}\leq
C(1+t)^{-\frac{3}{2}(\frac{1}{p}-\frac{1}{q})-\frac{m+2}{2}}\|N_{0}\|_{L^{p}}+C{\exp({-\lambda
t})}\|N_{0}\|_{L^{p}}\\
+C (1+t)^{-\frac{\ell}{2}}
  \|\nabla^{m+[\ell+3(\frac{1}{r}-\frac{1}{q})]_+}N_{0}\|_{L^r}+C{\exp({-\lambda
t})}\|\nabla^{m+2+[3(\frac{1}{r}-\frac{1}{q})]_+}[f_{i0},f_{e0}]\|_{L^{r}}\\
+C(1+t)^{-\frac{3}{2}(\frac{1}{p}-\frac{1}{q})-\frac{m+2}{2}}\|[f_{i0},f_{e0}]\|_{L^{p}},
\end{multline*}
\begin{multline*}
\|\nabla ^{m}\FP_3 e^{tL}N_{0}\|_{L^q}\leq
C(1+t)^{-\frac{3}{2}(\frac{1}{p}-\frac{1}{q})-\frac{m+2}{2}}\|N_{0}\|_{L^{p}}+C{\exp({-\lambda
t})}\|U_{0}\|_{L^{p}}\\
+C (1+t)^{-\frac{\ell}{2}}
  \|\nabla^{m+[\ell+3(\frac{1}{r}-\frac{1}{q})]_+}N_{0}\|_{L^r}+C{\exp({-\lambda
t})}\|\nabla^{m+2+[3(\frac{1}{r}-\frac{1}{q})]_+}[f_{i0},f_{e0}]\|_{L^{r}}\\
+C(1+t)^{-\frac{3}{2}(\frac{1}{p}-\frac{1}{q})-\frac{m+2}{2}}\|[f_{i0},f_{e0}]\|_{L^{p}}.
\end{multline*}
\begin{multline*}
\|\nabla ^{m}\FP_4 e^{tL}N_{0}\|_{L^q}\leq
C(1+t)^{-\frac{3}{2}(\frac{1}{p}-\frac{1}{q})-\frac{m+1}{2}}\|N_{0}\|_{L^{p}}+C{\exp({-\lambda
t})}\|N_{0}\|_{L^{p}}\\
+C (1+t)^{-\frac{\ell}{2}}
  \|\nabla^{m+[\ell+3(\frac{1}{r}-\frac{1}{q})]_+}N_{0}\|_{L^r},
\end{multline*}
for any $t\geq 0$, where  $C=C(m,p,r,q,\ell)$,
$[\ell+3(\tfrac{1}{r}-\tfrac{1}{q})]_+$ is defined in
\eqref{thm.decay.1}, and  $\FP_{1\al}$, $\FP_{2\al}$, $\FP_3$, $\FP_4$ are the projection operators along the component $\rho_\al$, $u_\al$, $E$, $B$ of the solution $e^{tL}N_0$, respectively.
\end{theorem}

\section{Asymptotic behaviour of the nonlinear system}\label{sec5}

\subsection{Global existence}
To the end, we  assume the integer $N \geq 3$. For
$U=[\rho_{\alpha},u_{\alpha},E,B]$, we define the full instant
energy functional $\mathcal {E}_{N}(U(t))$ and  the high-order instant
by
\begin{equation}\label{de.E}
\arraycolsep=1.5pt
\begin{array}{rl}
\mathcal {E}_{N}(U(t))=&\displaystyle\sum_{|l|\leq N}\sum_{
\alpha=i,e}\int_{\mathbb{R}^3}\frac{p'_{\alpha}(\rho_{\alpha}+1)}{\rho_{\alpha}+1}
   |\partial^{l}\rho_{\alpha}|^2 +m_{\alpha}(\rho_{\alpha}+1)
   |\partial^{l}u_{\alpha}|^2)dx+\frac{1}{4\pi}\|[E,B]\|_{N}^{2}\\[3mm]
&\displaystyle+\kappa_{1}\sum_{|l|\leq
N-1}\sum_{\alpha=i,e}m_{\alpha}\langle
\partial^{l}u_{\alpha},
\partial^{l}\nabla\rho_{\alpha}\rangle+\kappa_{2}\sum_{|l|\leq
N-1}\sum_{\alpha=i,e}m_{\alpha}\left\langle
\partial^{l}u_{\alpha},-\frac{q_{\alpha}}{T_{\alpha}}\partial^{l}E\right\rangle\\[3mm]
&\displaystyle-\kappa_{3}\sum_{|l|\leq N-2}\langle
\partial^{l}E,\nabla \times \partial^{l}B\rangle,
\end{array}
\end{equation}
and
\begin{equation}\label{de.Eh}
\begin{aligned}
\mathcal {E}_{N}^{h}(U(t))=&\displaystyle\sum_{1\leq |l|\leq
N}\sum_{
\alpha=i,e}\int_{\mathbb{R}^3}\frac{p'_{\alpha}(\rho_{\alpha}+1)}{\rho_{\alpha}+1}
   |\partial^{l}\rho_{\alpha}|^2 +m_{\alpha}(\rho_{\alpha}+1)
   |\partial^{l}u_{\alpha}|^2)dx+\frac{1}{4\pi}\|\nabla[E,B]\|_{N-1}^{2}\\[3mm]
&\displaystyle+\kappa_{1}\sum_{1\leq|l|\leq
N-1}\sum_{\alpha=i,e}m_{\alpha}\langle
\partial^{l}u_{\alpha},
\partial^{l}\nabla\rho_{\alpha}\rangle+\kappa_{2}\sum_{1\leq |l|\leq
N-1}\sum_{\alpha=i,e}m_{\alpha}\left\langle
\partial^{l}u_{\alpha},-\frac{q_{\alpha}}{T_{\alpha}}\partial^{l}E\right\rangle\\[3mm]
&\displaystyle-\kappa_{3}\sum_{1\leq |l|\leq N-2}\langle
\partial^{l}E,\nabla \times \partial^{l}B\rangle,
\end{aligned}
\end{equation}
respectively, where $0<\kappa_{3}\ll\kappa_{2}\ll\kappa_{1}\ll 1$
are constants to be properly chosen later on. Notice that
since all constants $\kappa_i$ $(i=1,2,3)$ are small enough, one has
\begin{equation*}
    \mathcal {E}_{N}(U(t))\sim
\|[\rho_{\alpha},u_{\alpha},E,B] \|_{N}^{2},\quad \mathcal
{E}_{N}^{h}(U(t))\sim \|\nabla [\rho_{\alpha},u_{\alpha},E,B]
\|_{N-1}^{2}.
\end{equation*}
We further define the corresponding dissipation rates $\mathcal {D}_{N}(U(t))$,
$\mathcal {D}_{N}^{h}(U(t))$ by
\begin{equation}\label{de.D}
\mathcal {D}_{N}(U(t))=\displaystyle \sum_{|l|\leq
N}\int_{\mathbb{R}^3}\sum_{\alpha=i,e} m_{\alpha}(\rho_{\alpha}+1)
   |\partial^{l}u_{\alpha}|^2dx+\sum_{\alpha=i,e}\|\nabla\rho_{\alpha}\|_{N-1}^{2}
+\|\nabla[E,B]\|_{N-2}^{2}+\|E\|^{2},
\end{equation}
and
\begin{multline}\label{de.Dh}
\mathcal {D}_{N}^{h}(U(t))=\displaystyle \displaystyle \sum_{1\leq
|l|\leq N}\int_{\mathbb{R}^3}\sum_{\alpha=i,e}
m_{\alpha}(\rho_{\alpha}+1)
   |\partial^{l}u_{\alpha}|^2dx+\sum_{\alpha=i,e}\|\nabla^{2}\rho_{\alpha}\|_{N-2}^{2}\\
+\|\nabla^{2}[E,B]\|_{N-3}^{2}+\|\nabla E\|^{2},
\end{multline}
respectively.
Then, the global existence of  the reformulated Cauchy problem
\eqref{2.5}-\eqref{sec5.ggg} with small smooth initial data can be stated as follows.

\begin{theorem}\label{pro.2.1}
There is
$\mathcal {E}_{N}(\cdot) $  in the form of
$\eqref{de.E} $ such that the following holds
true. If $\mathcal {E}_{N}(U_{0})>0$ is small enough,  the Cauchy
problem $\eqref{2.5}$-$\eqref{sec5.ggg}$ admits a unique global
solution $U=[\rho_{\alpha},u_{\alpha}, E,B]$ satisfying
\begin{eqnarray*}
U \in C([0,\infty);H^{N}(\mathbb{R}^{3}))\cap {\rm
Lip}([0,\infty);H^{N-1}(\mathbb{R}^{3})),
\end{eqnarray*}
and
\begin{eqnarray*}
\mathcal {E}_{N}(U(t))+\lambda\int_{0}^{t}\mathcal
{D}_{N}(U(s))ds\leq \mathcal {E}_{N}(U_{0}),
\end{eqnarray*}
for any $t\geq 0$.
\end{theorem}

To prove Theorem \ref{pro.2.1} it suffices to show the a priori estimates in the following lemma, cf.~\cite{Kato}. For completeness, we also give all the details of the proof.

\begin{lemma}[a priori estimates]\label{estimate}
Suppose that $U=[\rho_{\alpha},u_{\alpha},E,B]\in
C([0,T);H^{N}(\mathbb{R}^{3}))$ is smooth for $T>0$ with
\begin{eqnarray*}
\sup_{0\leq t<T}\|U(t)\|_{N}\leq 1,
\end{eqnarray*}
and that $U$ solves the system (\ref{2.5}) over $0\leq t<T$.
Then, there is $\mathcal {E}_{N}(\cdot) $  in the form $\eqref{de.E} $ such
that
\begin{eqnarray}\label{3.2}
&& \frac{d}{dt}\mathcal {E}_{N}(U(t))+\lambda\mathcal
{D}_{N}(U(t))\leq
 C[\mathcal {E}_{N}(U(t))^{\frac{1}{2}}+\mathcal {E}_{N}(U(t))]\mathcal {D}_{N}(U(t))
\end{eqnarray}
for any $0\leq t<T$.
\end{lemma}

\begin{proof} It is divided by five steps as follows.
\medskip

\noindent \textbf{ Step 1.} It holds that
\begin{equation}\label{3.3}
\begin{aligned}
&\frac{1}{2}\frac{d}{dt}\left(\sum_{|l|\leq N}\sum_{
\alpha=i,e}\int_{\mathbb{R}^3}\frac{p'_{\alpha}(\rho_{\alpha}+1)}{\rho_{\alpha}+1}
   |\partial^{l}\rho_{\alpha}|^2 +m_{\alpha}(\rho_{\alpha}+1)
   |\partial^{l}u_{\alpha}|^2)dx+\frac{1}{4\pi}\|[E,B]\|_{N}^{2}\right)\\
&+\sum_{|l|\leq N}\int_{\mathbb{R}^3}\sum_{\alpha=i,e}
m_{\alpha}(\rho_{\alpha}+1)
   |\partial^{l}u_{\alpha}|^2dx\leq  C(\|U\|_{N}+\|U\|_{N}^{2})\sum_{
\alpha=i,e}(\|u_{\alpha}\|^{2}+\|\nabla[\rho_{\alpha}u_{\alpha}]\|_{N-1}^{2}).
 \end{aligned}
\end{equation}
In fact, it is convenient  to start from the
following form of \eqref{2.5}:
\begin{equation}\label{re2.2}
\left\{
  \begin{aligned}
  &\partial_t \rho_\alpha+(\rho_{\alpha}+1)\nabla\cdot u_\alpha=-u_{\alpha}\cdot\nabla \rho_{\alpha},\\
  &m_{\alpha}\partial_t u_{\alpha}+\frac{p_{\alpha}'(\rho_{\alpha}+1)}{\rho_{\alpha}+1}\nabla \rho_{\alpha}-q_{\alpha}E+m_{\alpha}\nu_\alpha u_\alpha
  =-m_{\alpha}u_{\alpha}\cdot \nabla u_{\alpha}+q_{\alpha}\frac{u_{\alpha}}{c}\times B,\\
  &\partial_t E-c\nabla\times B+4\pi\sum_{\alpha=i,e}q_{\alpha}(\rho_{\alpha}+1)u_{\alpha}=0,\\
  &\partial_t B+c\nabla \times E=0,\\
  &\nabla \cdot E=4\pi\sum_{\alpha=i,e}q_{\alpha}\rho_{\alpha}, \ \  \nabla
\cdot B=0.
\end{aligned}\right.
\end{equation}
Applying $\partial^{l}$ to the first equation of (\ref{re2.2}) for
$0 \leq l\leq N$, multiplying it by
$\frac{p'_{\alpha}(\rho_{\alpha}+1)}{\rho_{\alpha}+1}
\partial^{l}\rho_{\alpha}$, and taking integration in $x$ gives
\begin{eqnarray}\label{4.17}
  && \begin{aligned}
   &\frac{1}{2}\frac{d}{dt}\left\langle\frac{p'_{\alpha}(\rho_{\alpha}+1)}{\rho_{\alpha}+1},
   |\partial^{l}\rho_{\alpha}|^2 \right\rangle +\left\langle p'_{\alpha}(\rho_{\alpha}+1)\partial^{l}\nabla\cdot u_\alpha,
 \partial^{l}\rho_{\alpha} \right\rangle\\
 =  &\frac{1}{2}\left\langle\left(\frac{p'_{\alpha}(\rho_{\alpha}+1)}{\rho_{\alpha}+1}\right)_{t}, |\partial^{l}\rho_{\alpha}|^2 \right\rangle
   -\sum_{k< l}C^{k}_{l}\left\langle \partial^{l-k}(\rho_{\alpha}+1)
   \partial^{k}\nabla\cdot u_\alpha,\frac{p'_{\alpha}(\rho_{\alpha}+1)}{\rho_{\alpha}+1}\partial^{l}\rho_{\alpha} \right\rangle\\
   &-\left\langle u_{\alpha}\cdot\partial^{l}\nabla\rho_{\alpha},\frac{p'_{\alpha}(\rho_{\alpha}+1)}{\rho_{\alpha}+1}\partial^{l}\rho_{\alpha} \right\rangle
   -\sum_{k< l}C^{k}_{l}\left\langle \partial^{l-k}u_{\alpha}\cdot
   \partial^{k}\nabla\rho_{\alpha},\frac{p'_{\alpha}(\rho_{\alpha}+1)}{\rho_{\alpha}+1}\partial^{l}\rho_{\alpha}
   \right\rangle.
\end{aligned}
\end{eqnarray}
Applying $\partial^{l}$ to the second equation of $(\ref{re2.2})$ for $0\leq l\leq N$, multiplying it by $(\rho_{\alpha}+1)\partial^{l}u_{\alpha}$, and
integrating the resulting equation with respect to $x$ give
\begin{eqnarray}\label{4.4}
  && \begin{aligned}
   &\frac{1}{2}\frac{d}{dt}\langle m_{\alpha}(\rho_{\alpha}+1),
   |\partial^{l}u_{\alpha}|^2 \rangle+\left\langle p'_{\alpha}(\rho_{\alpha}+1)\partial^{l}\nabla\rho_{\alpha},
\partial^{l}u_{\alpha} \right\rangle\\
&+m_{\alpha}\nu_\alpha\langle (\rho_{\alpha}+1),
   |\partial^{l}u_{\alpha}|^2 \rangle
  -\langle q_{\alpha}\partial^{l}E,(\rho_{\alpha}+1)\partial^{l}u_{\alpha} \rangle\\
   =&\frac{1}{2}\langle m_{\alpha}(\rho_{\alpha}+1)_{t},
   |\partial^{l}u_{\alpha}|^2 \rangle
-\sum_{k<l}C^{k}_{l}\left\langle\partial^{l-k}\left(\frac{p'_{\alpha}(\rho_{\alpha}+1)}{\rho_{\alpha}+1}\right)
\partial^{k}\nabla\rho_{\alpha},(\rho_{\alpha}+1)\partial^{l}u_{\alpha} \right\rangle\\
   &-m_{\alpha}\left\langle u_{\alpha}\cdot\partial^{l}\nabla u_{\alpha},(\rho_{\alpha}+1)\partial^{l}u_{\alpha} \right\rangle
   -\sum_{k<l}C^{k}_{l}\left\langle \partial^{l-k}u_{\alpha}\cdot\partial^{k}\nabla u_{\alpha},(\rho_{\alpha}+1)\partial^{l}u_{\alpha}\right\rangle\\
  &+q_{\alpha}\left\langle\partial^{l}\left(\frac{u_{\alpha}}{c}\times B\right),(\rho_{\alpha}+1)\partial^{l}u_{\alpha}\right\rangle.
\end{aligned}
\end{eqnarray}
Taking the summation of \eqref{4.17} and \eqref{4.4}, integrating by
parts and taking the summation of $\alpha=i,e$, we get
\begin{equation}\label{sum.2425}
  \begin{aligned}
   &\frac{1}{2}\frac{d}{dt}\sum_{\alpha=i,e}\left(\left\langle\frac{p'_{\alpha}(\rho_{\alpha}+1)}{\rho_{\alpha}+1},
   |\partial^{l}\rho_{\alpha}|^2 \right\rangle+\langle m_{\alpha}(\rho_{\alpha}+1),
   |\partial^{l}u_{\alpha}|^2 \rangle\right)
   +\sum_{\alpha=i,e}m_{\alpha}\nu_\alpha\langle (\rho_{\alpha}+1),
   |\partial^{l}u_{\alpha}|^2 \rangle\\
  &-\left\langle \partial^{l}E,\sum_{\alpha=i,e}q_{\alpha}(\rho_{\alpha}+1)\partial^{l}u_{\alpha} \right\rangle
   =\sum_{\alpha=i,e}\left(I_{1}^{(\alpha)}(t)+\sum_{k<l}C^{k}_{l}I_{k,l}^{(\alpha)}(t)\right),
\end{aligned}
\end{equation}
with
\begin{eqnarray*}
  && \begin{aligned}
  I_{1}^{(\alpha)}(t)=&\left\langle p''(\rho_{\alpha}+1)\nabla\rho_{\alpha}\partial^{l}\rho_{\alpha},
\partial^{l}u_{\alpha}\right\rangle+
\frac{1}{2}\left\langle\left(\frac{p'_{\alpha}(\rho_{\alpha}+1)}{\rho_{\alpha}+1}\right)_{t},
|\partial^{l}\rho_{\alpha}|^2 \right\rangle+\frac{1}{2}\langle
m_{\alpha}(\rho_{\alpha}+1)_{t},
   |\partial^{l}u_{\alpha}|^2 \rangle\\
&+\frac{1}{2}\left\langle\nabla\cdot\left(u_{\alpha}\frac{p'_{\alpha}(\rho_{\alpha}+1)}{\rho_{\alpha}+1}\right),|\partial^{l}\rho_{\alpha}|^{2}
\right\rangle +\frac{m_{\alpha}}{2}\left\langle
\nabla\cdot(u_{\alpha}(\rho_{\alpha}+1)),|\partial^{l}u_{\alpha}|^{2}
\right\rangle\\
&+q_{\alpha}\left\langle\partial^{l}\left(\frac{u_{\alpha}}{c}\times
B\right),(\rho_{\alpha}+1)\partial^{l}u_{\alpha}\right\rangle,
\end{aligned}
\end{eqnarray*}
and
\begin{eqnarray*}
  && \begin{aligned}
  I_{k,l}^{(\alpha)}(t)=&-\left\langle \partial^{l-k}(\rho_{\alpha}+1)
   \partial^{k}\nabla\cdot u_\alpha,\frac{p'_{\alpha}(\rho_{\alpha}+1)}{\rho_{\alpha}+1}\partial^{l}\rho_{\alpha} \right\rangle-\left\langle \partial^{l-k}u_{\alpha}\cdot
   \partial^{k}\nabla\rho_{\alpha},\frac{p'_{\alpha}(\rho_{\alpha}+1)}{\rho_{\alpha}+1}\partial^{l}\rho_{\alpha} \right\rangle\\
&-\left\langle\partial^{l-k}\left(\frac{p'_{\alpha}(\rho_{\alpha}+1)}{\rho_{\alpha}+1}\right)
\partial^{k}\nabla\rho_{\alpha},(\rho_{\alpha}+1)\partial^{l}u_{\alpha} \right\rangle
-\left\langle \partial^{l-k}u_{\alpha}\cdot\partial^{k}\nabla u_{\alpha},(\rho_{\alpha}+1)\partial^{l}u_{\alpha}\right\rangle.
\end{aligned}
\end{eqnarray*}
Recall that
\begin{equation*}
\partial_{t}\rho_{\alpha} =-(\rho_{\alpha}+1)\nabla \cdot u_{\alpha}-u_{\alpha}\cdot\nabla \rho_{\alpha}.
\end{equation*}
When $|l|=0$, it suffices to estimate $I_{1}^{(\alpha)}(t)$ by
\begin{eqnarray*}
\begin{aligned}
 I_{1}^{(\alpha)}(t)
  \leq & C \|\nabla \cdot u_{\alpha}\|(\|\rho_{\alpha}\|_{L^{6}}\|\rho_{\alpha}\|_{L^{3}}+ \|u_{\alpha}\|_{L^{6}}\|u_{\alpha}\|_{L^{3}})
  +C\|\nabla  \rho_{\alpha}\|\|\rho_{\alpha}\|_{L^{6}}\|u_{\alpha}\|_{L^{3}}\\
  &+C\|u_{\alpha}\|_{L^{\infty}}(\|\nabla
  \rho_{\alpha}\|\|\rho_{\alpha}\|_{L^{6}}\|\rho_{\alpha}\|_{L^{3}}+\|\nabla
  \rho_{\alpha}\|\|u_{\alpha}\|_{L^{6}}\|u_{\alpha}\|_{L^{3}})+C\|B\|_{L^{\infty}}\|u_{\alpha}\|^{2}\\
\leq &
C(\|[\rho_{\alpha},u_{\alpha}]\|_{H^{1}}+\|[\rho_{\alpha},u_{\alpha}]\|_{H^{2}}^{2})\|\nabla[\rho_{\alpha},u_{\alpha}]\|^{2}+C\|\nabla
B\|_{H^{1}}\|u_{\alpha}\|^{2},
 \end{aligned}
\end{eqnarray*}
which is further bounded by the r.h.s.~term of (\ref{3.3}). When
$|\alpha|\geq 1$, since each term in $I_{k,l}^{(\alpha)}(t)$ and
$I_{1}^{(\alpha)}(t)$ is at least the integration of the three-terms
product in which there is at least one term containing the
derivative, one has
\begin{eqnarray*}
  I_{k,l}^{(\alpha)}(t)+I_{1}^{(\alpha)}(t)\leq C\left(\|[\rho_{\alpha},u_{\alpha},B]\|_{N}+\|[\rho_{\alpha},u_{\alpha}]\|_{H^{3}}^{2}\right)\|\nabla
  [\rho_{\alpha},u_{\alpha}]\|_{N-1}^{2},
\end{eqnarray*}
which is also bounded by the r.h.s.~term of (\ref{3.3}). On the
other hand, from (\ref{re2.2}), energy estimates on $\partial^{l}E$
and $\partial^{l}B$ with $0\leq l\leq N$, yield
\begin{eqnarray}\label{energyEB}
  && \begin{aligned}
   &\frac{1}{4\pi}\frac{1}{2}\frac{d}{dt}\left\|\partial^{l}[E,B]\right\|^{2}+
   \sum_{\alpha=i,e}\langle q_{\alpha}(\rho_{\alpha}+1)\partial^{l}u_{\alpha},\partial^{l}E\rangle
   \\
  =&-\sum_{\alpha=i,e}\left\langle
  q_{\alpha}\sum_{k<l}C^{k}_{l}\partial^{l-k}(\rho_{\alpha}+1)\partial^{k}u_{\alpha},\partial^{l}E\right\rangle:=I_{2}(t).
\end{aligned}
\end{eqnarray}
When $|l|=1$, $I_{2}(t)$ can be  estimated as follows,
\begin{eqnarray*}
I_{2}(t)\leq C\sum_{\alpha=i,e}\|u_{\alpha}\|_{L^{\infty}}\|\nabla
\rho_{\alpha}\|\|\nabla E\|\leq \|\nabla
E\|\sum_{\alpha=i,e}\|\nabla [\rho_{\alpha},u_{\alpha}]\|_{N-1}^{2}.
\end{eqnarray*}
When $1<|l|\leq N$, $I_{2}(t)$ can be  estimated as follows,
\begin{eqnarray*}
 \begin{aligned}
I_{2}(t)\leq &
C\sum_{\alpha=i,e}\|\partial^{l}((\rho_{\alpha}+1)u_{\alpha})-(\rho_{\alpha}+1)\partial^{l}u_{\alpha}\|\|\nabla
E\|_{N-1}\\
\leq & C\sum_{\alpha=i,e}(\|\nabla
\rho_{\alpha}\|_{L^{\infty}}\|\partial^{l-1}u_{\alpha}\|+\|u_{\alpha}\|_{L^{\infty}}\|\partial^{l}\rho_{\alpha}\|)\|\nabla
E\|_{N-1}\\
\leq &\|\nabla E\|_{N-1}\sum_{\alpha=i,e}\|\nabla
[\rho_{\alpha},u_{\alpha}]\|_{N-1}^{2}.
\end{aligned}
\end{eqnarray*}
Then (\ref{3.3}) follows by taking the summation of (\ref{sum.2425}) and
(\ref{energyEB}) over $|l| \leq N$.

\medskip
\noindent{\bf Step 2.}
 It holds that
\begin{eqnarray}\label{step2}
  &&\begin{aligned}
  &\frac{d}{dt}\mathcal {E}_{N,1}^{int}(U(t))
  +\lambda\sum_{\alpha=i,e}\|\nabla\rho_{\alpha}\|_{N-1}^{2}+4\pi\left\|\sum_{\alpha=i,e}q_{\alpha}\rho_{\alpha}
\right\|_{N-1}^2\\
  \leq &
  C\sum_{\alpha=i,e}\|u_{\alpha}\|_{N}^{2}+C\left(\sum_{\alpha=i,e}\|[\rho_{\alpha},u_{\alpha},B]\|_{N}^{2}\right)\left(\sum_{\alpha=i,e}\|\nabla
[\rho_{\alpha},u_{\alpha}]\|_{N-1}^{2}\right),
  \end{aligned}
\end{eqnarray}
where $\mathcal {E}_{N,1}^{int}(\cdot)$ is defined by
\begin{eqnarray*}
\mathcal {E}_{N,1}^{int}(U(t))=\sum_{|l|\leq
N-1}\sum_{\alpha=i,e}m_{\alpha}\langle
\partial^{l}u_{\alpha},
\partial^{l}\nabla\rho_{\alpha}\rangle.
\end{eqnarray*}
In fact, recall the first two equations in \eqref{2.5}
\begin{equation}\label{step2equ}
\left\{
  \begin{aligned}
  &\partial_t \rho_\alpha+\nabla\cdot u_\alpha=g_{1\alpha},\\
  &m_{\alpha}\partial_t u_{\alpha}+T_{\alpha}\nabla \rho_{\alpha}-q_{\alpha}E+m_{\alpha}\nu_\alpha u_\alpha=g_{2\alpha},
  \end{aligned}
  \right.
  \end{equation}
with
\begin{equation*}
\arraycolsep=1.5pt \left\{
 \begin{aligned}
 &g_{1\alpha}=-\nabla\cdot(\rho_\alpha u_\alpha),\\
 &g_{2\alpha}=- m_{\alpha}u_{\alpha} \cdot \nabla
u_\alpha-\left(\frac{p_{\alpha}'(\rho_{\alpha}+1)}{\rho_{\alpha}+1}-\frac{p_{\alpha}'(1)}{1}\right)\nabla
\rho_{\alpha}+q_{\alpha}\frac{u_{\alpha}}{c}\times B.
\end{aligned}\right.
\end{equation*}
Let $0\leq l\leq N-1$, applying $\partial^{l}$ to
the second equation of $(\ref{step2equ})$, multiplying it by $\partial^{l}\nabla
\rho_{\alpha}$, taking integrations in $x$, using
integration by parts and also the final equation of (\ref{re2.2}),
replacing $\partial_{t}\rho_{\alpha}$ from $(\ref{step2equ})_{1}$, and
taking the summation of $\alpha=i,e$, one has
\begin{equation*}
 \begin{aligned}
 &
\sum_{\alpha=i,e}m_{\alpha}\frac{d}{dt}\langle
\partial^{l}u_{\alpha},
\partial^{l}\nabla\rho_{\alpha}\rangle
+\sum_{\alpha=i,e}T_{\alpha}\|\partial^{l}\nabla\rho_{\alpha}\|^{2}+4\pi\left\|\sum_{\alpha=i,e}q_{\alpha}\partial^{l}\rho_{\alpha}
\right\|^2\\
=&\sum_{\alpha=i,e}m_{\alpha}\|\nabla\cdot\partial^{l}u_{\alpha}\|^{2}-
\sum_{\alpha=i,e}m_{\alpha}\langle\nabla\cdot\partial^{l}u_{\alpha},\partial^{l}g_{1\alpha}\rangle\\
& -\sum_{\alpha=i,e}\langle
m_{\alpha}\nu_{\alpha}\partial^{l}u_{\alpha},\nabla\partial^{l}\rho_{\alpha}\rangle
 +\sum_{\alpha=i,e}\langle\partial^{l}g_{2\alpha},\nabla\partial^{l}\rho_{\alpha}\rangle.
 \end{aligned}
\end{equation*}
Then, it follows from the Cauchy-Schwarz inequality that
\begin{equation}\label{ineq.2}
 \begin{aligned}
 &
\sum_{\alpha=i,e}m_{\alpha}\frac{d}{dt}\langle
\partial^{l}u_{\alpha},
\partial^{l}\nabla\rho_{\alpha}\rangle
+\lambda\sum_{\alpha=i,e}\|\partial^{l}\nabla\rho_{\alpha}\|^{2}+4\pi\left\|\sum_{\alpha=i,e}q_{\alpha}\partial^{l}\rho_{\alpha}
\right\|^2\\
\leq& C\sum_{\alpha=i,e}(\|\nabla \cdot
\partial^{l}u_{\alpha}\|^{2}+\|\partial^{l}u_{\alpha}\|^{2})+C\sum_{\alpha=i,e}(\|\partial^{l}g_{1\alpha}\|^{2}+\|\partial^{l}g_{2\alpha}\|^{2}).
\end{aligned}
\end{equation}
Noticing that  $g_{1\alpha},\ g_{2\alpha}$ are quadratically
nonlinear, one has
\begin{equation*}
\|\partial^{l}g_{1\alpha}\|^{2}+\|\partial^{l}g_{2\alpha}\|^{2}\leq
C\left(\sum_{\alpha=i,e}\|[\rho_{\alpha},u_{\alpha},B]\|_{N}^{2}\right)\left(\sum_{\alpha=i,e}\|\nabla
[\rho_{\alpha},u_{\alpha}]\|_{N-1}^{2}\right).
\end{equation*}
Substituting this into (\ref{ineq.2}) and taking
the summation over $|l|\leq N-1$ implies (\ref{step2}).

\medskip
\noindent{\bf Step 3.}
 It holds that
\begin{multline}\label{step3}
  \dfrac{d}{dt}\mathcal {E}_{N,2}^{int}(U(t))+\dfrac{1}{4\pi}\|\nabla\cdot E\|_{N-1}^{2}+\lambda\|E\|^{2}_{N-1} \\
  \leq
  C\sum_{\alpha=i,e}\|u_{\alpha}\|_{N}^{2}+C\sum_{\alpha=i,e}\|u_{\alpha}\|_{N}\|\nabla
B\|_{N-2}\\
+C\left(\sum_{\alpha=i,e}\|[\rho_{\alpha},u_{\alpha},B]\|_{N}^{2}\right)\left(\sum_{\alpha=i,e}\|\nabla
[\rho_{\alpha},u_{\alpha}]\|_{N-1}^{2}\right),
\end{multline}
where $\mathcal {E}_{N,2}^{int}(\cdot)$ is defined by
\begin{eqnarray*}
\mathcal {E}_{N,2}^{int}(U(t))=\sum_{|l|\leq
N-1}\sum_{\alpha=i,e}m_{\alpha}\left\langle
\partial^{l}u_{\alpha},-\frac{q_{\alpha}}{T_{\alpha}}\partial^{l}E\right\rangle.
\end{eqnarray*}
In fact, for $|l|\leq N-1$, applying
$\partial^{l}$ to the second equation of $\eqref{step2equ}$, multiplying it by
$-\frac{q_{\alpha}}{T_{\alpha}}\partial^{l}E$, taking integrations
in $x$, using integration by parts, replacing $
\partial_{t}E$ from the third equation of $(\ref{2.5})$, and taking the summation for $\alpha=i,e$ give
\begin{equation*}
 \begin{aligned}
 &\frac{d}{dt}\sum_{\alpha=i,e}m_{\alpha}\langle
\partial^{l}u_{\alpha},-\frac{q_{\alpha}}{T_{\alpha}}\partial^{l}E\rangle+
\dfrac{1}{4\pi}\|\partial^{l}\nabla\cdot E\|^{2}+\sum_{\alpha=i,e}\frac{q_{\alpha}^{2}}{T_{\alpha}}\|\partial^{l} E\|^{2}\\
=&-\sum_{\alpha=i,e}m_{\alpha}\left\langle
\partial^{l}u_{\alpha},\frac{q_{\alpha}}{T_{\alpha}}c\nabla \times
\partial^{l}B\right\rangle+\sum_{\alpha=i,e}m_{\alpha} \left\langle
\partial^{l}u_{\alpha},\frac{q_{\alpha}}{T_{\alpha}}\partial^{l} \left(4\pi\sum_{\alpha=i,e}(\rho_{\alpha}+1)u_{\alpha}\right)
\right\rangle\\
&+\sum_{\alpha=i,e}\left\langle
m_{\alpha}\nu_{\alpha}\partial^{l}u_{\alpha},\frac{q_{\alpha}}{T_{\alpha}}
\partial^{l}E\right\rangle+\sum_{\alpha=i,e}\left\langle\partial^{l}g_{2\alpha},-\frac{q_{\alpha}}{T_{\alpha}}
\partial^{l}E\right\rangle,
\end{aligned}
\end{equation*}
which from the Cauchy-Schwarz inequality, further implies
\begin{multline*}
\frac{d}{dt}\sum_{\alpha=i,e}m_{\alpha}\left\langle
\partial^{l}u_{\alpha},-\frac{q_{\alpha}}{T_{\alpha}}\partial^{l}E\right\rangle+
\dfrac{1}{4\pi}\|\partial^{l}\nabla\cdot E\|^{2}+\lambda\|\partial^{l} E\|^{2}\\
\leq
C\sum_{\alpha=i,e}\|u_{\alpha}\|_{N}^{2}+C\sum_{\alpha=i,e}\|u_{\alpha}\|_{N}\|\nabla
B\|_{N-2}\\
+C\left(\sum_{\alpha=i,e}\|[\rho_{\alpha},u_{\alpha},B]\|_{N}^{2}\right)\left(\sum_{\alpha=i,e}\|\nabla
[\rho_{\alpha},u_{\alpha}]\|_{N-1}^{2}\right).
\end{multline*}
Thus $\eqref{step3}$ follows from taking the summation of the above
estimate over $|l|\leq N-1$.

\medskip
\noindent{\bf Step 4.} It holds that
\begin{equation}\label{step4}
\begin{aligned}
   \frac{d}{dt}\mathcal {E}_{N,3}^{int}(U(t))+\lambda\|\nabla B\|^{2}_{N-2}
   \leq & C\|E\|_{N-1}^{2}+\sum_{\alpha=i,e}\|u_{\alpha}\|_{N}^{2}\\
   &+C\left(\sum_{\alpha=i,e}\|[\rho_{\alpha},u_{\alpha}]\|_{N}^{2}\right)\left(\sum_{\alpha=i,e}\|\nabla
[\rho_{\alpha},u_{\alpha}]\|_{N-1}^{2}\right),
\end{aligned}
\end{equation}
where $\mathcal {E}_{N,3}^{int}(\cdot)$ is defined by
\begin{eqnarray*}
\mathcal {E}_{N,3}^{int}(U(t))=-\sum_{|l|\leq N-2}\langle
\partial^{l}E,\nabla \times \partial^{l}B\rangle.
\end{eqnarray*}
In fact, for $|l|\leq N-2$, applying $\partial^{l}$ to the third
equation of  $\eqref{2.5}$, multiplying it by $\partial^{l}\nabla
\times  B$, taking integrations in $x$ and using integration by
parts and replacing $
\partial_{t}B$  from the fourth  equation of $\eqref{2.5}$
implies
\begin{equation*}
\begin{aligned}
& -\dfrac{d}{dt}\langle
\partial^{l}E,\nabla \times \partial^{l}B\rangle+
c\|\nabla\times \partial^{l}B\|^{2}\\
=&c\|\nabla\times\partial^{l}E\|^{2} +4\pi\left\langle
\partial^{l}\left(\sum_{\alpha=i,e}q_{\alpha}(\rho_{\alpha}+1)u_{\alpha}\right),\nabla \times
\partial^{l}B\right\rangle.
\end{aligned}
\end{equation*}
The above estimate gives $\eqref{step4}$ by further using the
Cauchy-Schwarz inequality and taking the summation over $|l|\leq N-2$,
where we also have used
\begin{eqnarray*}
\|\partial^{l}\partial_{i}B\|=\|\partial_{i}\Delta^{-1}\nabla
\times(\nabla\times\partial^{l}B) \|\leq\|\nabla\times
\partial^{l}B\|
\end{eqnarray*}
for each $1\leq i\leq 3$, due to the fact  that
$\partial_{i}\Delta^{-1}\nabla$ is bounded from $L^{p}$ to itself
for $1<p<\infty$.

\medskip

\noindent\textbf{Step 5.}
 Let us define
\begin{eqnarray*}
   \mathcal {E}_{N}(U(t))=\sum_{|l|\leq N}\sum_{
\alpha=i,e}\int_{\mathbb{R}^3}\frac{p'_{\alpha}(\rho_{\alpha}+1)}{\rho_{\alpha}+1}
   |\partial^{l}\rho_{\alpha}|^2 +m_{\alpha}(\rho_{\alpha}+1)
   |\partial^{l}u_{\alpha}|^2)dx\\
   +\frac{1}{4\pi}\|[E,B]\|_{N}^{2}
   +\sum_{i=1}^{3}\kappa_{i}\mathcal {E}^{int}_{N,i}(U(t)),
\end{eqnarray*}
that is,
\begin{equation}\label{3.12}
\arraycolsep=1.5pt
\begin{array}{rl}
\mathcal {E}_{N}(U(t))=&\displaystyle\sum_{|l|\leq N}\sum_{
\alpha=i,e}\int_{\mathbb{R}^3}\frac{p'_{\alpha}(\rho_{\alpha}+1)}{\rho_{\alpha}+1}
   |\partial^{l}\rho_{\alpha}|^2 +m_{\alpha}(\rho_{\alpha}+1)
   |\partial^{l}u_{\alpha}|^2)dx+\frac{1}{4\pi}\|[E,B]\|_{N}^{2}\\[3mm]
&\displaystyle+\kappa_{1}\sum_{|l|\leq
N-1}\sum_{\alpha=i,e}m_{\alpha}\langle
\partial^{l}u_{\alpha},
\partial^{l}\nabla\rho_{\alpha}\rangle+\kappa_{2}\sum_{|l|\leq
N-1}\sum_{\alpha=i,e}m_{\alpha}\left\langle
\partial^{l}u_{\alpha},-\frac{q_{\alpha}}{T_{\alpha}}\partial^{l}E\right\rangle\\[3mm]
&\displaystyle-\kappa_{3}\sum_{|l|\leq N-2}\langle
\partial^{l}E,\nabla \times \partial^{l}B\rangle,
\end{array}
\end{equation}
for constants $0<\kappa_{3}\ll\kappa_{2}\ll\kappa_{1}\ll 1$ to be
determined. Notice that as long as $ 0<\kappa_{i}\ll 1$ is small
enough for $i=1,2,3$, then $\mathcal{E}_{N}(U(t))\sim
\|U(t)\|^{2}_{N}$ holds true. Moreover, letting
$0<\kappa_{3}\ll\kappa_{2}\ll\kappa_{1}\ll 1$ with
$\kappa_{2}^{3/2}\ll\kappa_{3}$, the sum of $\eqref{3.3}$,
$\eqref{step2}\times \kappa_{1}$, $\eqref{step3}\times \kappa_{2} $,
$\eqref{step4}\times \kappa_{3}$ implies that there are $\lambda>0$,
$C>0$ such that $\eqref{3.2}$ holds true with $\mathcal
{D}_{N}(\cdot)$ defined in $\eqref{de.D}$. Here, we have used the
following Cauchy-Schwarz inequality:
\begin{eqnarray*}
   2 \kappa_{2} \sum_{\alpha=i,e}\|u_{\alpha}\|_{N}\|\nabla
B\|_{N-2}\leq
   \kappa_{2}^{1/2}\sum_{\alpha=i,e}\|u_{\alpha}\|^{2}_{N}+\kappa_{2}^{3/2}\|\nabla B\|^{2}_{N-2}.
\end{eqnarray*}
Due to $\kappa_{2}^{3/2}\ll \kappa_{3}$, both terms on the r.h.s. of
the above inequality were absorbed. This completes the proof of
Theorem $\ref{estimate}$.
\end{proof}

\subsection{Asymptotic rate to constant states}

Moreover, the solutions obtained in Theorem $ \ref{pro.2.1}$ indeed
decay in time with some rates under some extra regularity and
integrability conditions on initial data. For that, given
$U_{0}=[\rho_{\alpha0},u_{\alpha0}, E_0,B_{0}]$, set
$\epsilon_{m}(U_0)$ as
\begin{eqnarray}\label{def.epsi}
\epsilon_{m}(U_0)=\|U_{0}\|_{m}+\|U_{0}\|_{L^{1}},
\end{eqnarray}
for the  integer $m \geq 0$. Then one has the following

\begin{theorem}\label{pro.2.2}
Under the assumptions of Proposition \ref{pro.2.1}, if
$\epsilon_{N+6}(U_{0})>0$ is small enough, then the solution
$U=[\rho_{\alpha},u_{\alpha}, E,B]$ satisfies
\begin{eqnarray}\label{V.decay}
\|U(t)\|_{N} \leq C \epsilon_{N+2}(U_{0})(1+t)^{-\frac{3}{4}},
\end{eqnarray}
and
\begin{eqnarray}\label{nablaV.decay}
\|\nabla U(t)\|_{N-1} \leq C
\epsilon_{N+6}(U_{0})(1+t)^{-\frac{5}{4}},
\end{eqnarray}
for any $t\geq 0$.
\end{theorem}

For completeness, we also give the proof of Theorem \ref{pro.2.2}.

\subsubsection{Time rate for the full instant energy functional}


Recall from the proof of Lemma \ref{estimate} that
\begin{eqnarray}\label{sec5.ENV0}
\dfrac{d}{dt}\mathcal {E}_{N}(U(t))+\lambda \mathcal {D}_{N}(U(t))
\leq 0,
\end{eqnarray}
for any $t\geq 0$. We now apply the time-weighted energy estimate and iteration to the
Lyapunov inequality $\eqref{sec5.ENV0}$. Let $\ell \geq 0$. Multiply
$\eqref{sec5.ENV0}$ by $(1+t)^{\ell}$ and taking integration over
$[0,t]$ gives
\begin{eqnarray*}
\begin{aligned}
 & (1+t)^{\ell}\mathcal {E}_{N}(U(t))+\lambda
 \int_{0}^{t}(1+s)^{\ell}\mathcal {D}_{N}(U(s))d s \\
 \leq & \mathcal {E}_{N}(U_{0})+ \ell
 \int_{0}^{t}(1+s)^{\ell-1}\mathcal {E}_{N}(U(s))d s.
\end{aligned}
\end{eqnarray*}
Noticing
\begin{eqnarray*}
  \mathcal {E}_{N}(U(t))
  \leq C (D_{N+1}(U(t))+\|
  B\|^{2}+ \|[\rho_{i},\rho_{e}]\|^{2}),
\end{eqnarray*}
it follows that
\begin{eqnarray*}
\begin{aligned}
 & (1+t)^{\ell}\mathcal {E}_{N}(U(t))+\lambda
 \int_{0}^{t}(1+s)^{\ell}\mathcal {D}_{N}(U(s))d s \\
 \leq & \mathcal {E}_{N}(U_{0})+ C \ell
 \int_{0}^{t}(1+s)^{\ell-1}(\|
  B\|^{2}+\|[\rho_{i},\rho_{e}]\|^{2})d s\\
&+ C\ell\int_{0}^{t}(1+s)^{\ell-1}\mathcal {D}_{N+1}(U(s))d s.
\end{aligned}
\end{eqnarray*}
Similarly, it holds that
\begin{eqnarray*}
\begin{aligned}
 & (1+t)^{\ell-1}\mathcal {E}_{N+1}(U(t))+\lambda
 \int_{0}^{t}(1+s)^{\ell-1}\mathcal {D}_{N+1}(U(s))d s \\
 \leq & \mathcal {E}_{N+1}(U_{0})+ C (\ell-1)
 \int_{0}^{t}(1+s)^{\ell-2}(\|
  B\|^{2}+\|[\rho_{i},\rho_{e}]\|^{2})d s\\
&+ C(\ell-1)\int_{0}^{t}(1+s)^{\ell-2}\mathcal {D}_{N+2}(U(s))d s,
\end{aligned}
\end{eqnarray*}
and
\begin{eqnarray*}
\mathcal {E}_{N+2}(U(t))+\lambda \int_{0}^{t}\mathcal
{D}_{N+2}(U(s))d s \leq \mathcal {E}_{N+2}(U_{0}).
\end{eqnarray*}
Then, for $1<\ell<2$, it follows by iterating the above estimates
that
\begin{eqnarray}\label{sec5.ED}
\begin{aligned}
 & (1+t)^{\ell}\mathcal {E}_{N}(U(t))+\lambda
 \int_{0}^{t}(1+s)^{\ell}\mathcal {D}_{N}(U(s))d s \\
 \leq & C \mathcal {E}_{N+2}(U_{0})+ C
 \int_{0}^{t}(1+s)^{\ell-1}(\|
  B\|^{2}+\|[\rho_{i},\rho_{e}]\|^{2})d s.
\end{aligned}
\end{eqnarray}
 On the other hand, to estimate the integral term on the r.h.s. of
$\eqref{sec5.ED}$, let's define
\begin{eqnarray}\label{sec5.def}
\mathcal {E}_{N,\infty}(U(t))=\sup\limits_{0\leq s \leq t}
(1+s)^{\frac{3}{2}}\mathcal {E}_{N}(U(s)).
\end{eqnarray}
\begin{lemma}\label{lem.Bsigma}
For any $t\geq0$, it holds that
\begin{eqnarray}\label{lem.tildeB}
&&\begin{aligned}
 \|B\|^{2}+\|[\rho_{i},\rho_{e}]\|^{2}
 \leq  C(1+t)^{-\frac{3}{2}}\Big(\mathcal
 {E}_{N,\infty}^{2}(U(t))+&\|[\rho_{i0},\rho_{e0},B_{0}]\|\|_{L^{1}\cap L^{2}}^2+\left.\|U_{0}\|_{L^{1}\cap \dot{H}^{2}}^{2}\right).
\end{aligned}
\end{eqnarray}
\end{lemma}

\begin{proof}
By applying  the first linear estimate on $\rho_{\alpha}$ and the
fourth linear estimate on $B$ and  letting $m=0,\ q=r=2,\ p=1,\
\ell=\frac{3}{2}$ in Corollary \ref{corollary.decayL} to the mild
form \eqref{sec5.U} respectively, one has
\begin{multline}\label{sec5.decayB}
 \|B(t)\|\leq C
(1+t)^{-\frac{3}{4}}
 \left(\|U_{0}\|_{L^1\cap \dot{H}^{2}}+\|B_{0}\|_{L^{1}\cap L^{2}}\right)\\
+C
\int_{0}^{t}(1+t-s)^{-\frac{3}{4}}\|[g_{1\alpha}(s),g_{2\alpha}(s),g_{3}(s)]\|_{L^{1}\cap\dot{H}^{2}}ds,
\end{multline}
and
\begin{multline}\label{sec5.decayrho}
 \|[\rho_{i},\rho_{e}]\|\leq C
(1+t)^{-\frac{3}{4}}
 \left(\|U_{0}\|_{L^1\cap \dot{H}^{2}}+\|[\rho_{i0},\rho_{e0}]\|_{L^{1}\cap L^{2}}\right)\\
+C
\int_{0}^{t}(1+t-s)^{-\frac{3}{4}}\left(\|[g_{1\alpha}(s),g_{2\alpha}(s),g_{3}(s)]\|_{L^{1}\cap\dot{H}^{2}}+\|g_{1\alpha}(s)\|_{L^{1}\cap
L^{2}}\right)ds.
\end{multline}
Recall the definition $\eqref{sec5.ggg}$ of $g_{1\alpha}$,
$g_{2\alpha}$ and $g_{3}$. It is straightforward to verify that for
any $0\leq s\leq t$,
\begin{eqnarray*}
\|[g_{1\alpha}(s),g_{2\alpha}(s),g_{3}(s)]\|_{L^{1}\cap\dot{H}^{2}}\leq
C \mathcal {E}_{N}(U(s))\leq (1+s)^{-\frac{3}{2}}\mathcal
{E}_{N,\infty}(U(t)),
\end{eqnarray*}
\begin{eqnarray*}
\|g_{1\alpha}(s)\|_{L^{1}\cap L^{2}}\leq C \mathcal
{E}_{N}(U(s))\leq (1+s)^{-\frac{3}{2}}\mathcal {E}_{N,\infty}(U(t)).
\end{eqnarray*}
Here we have used $\eqref{sec5.def}$. Putting the above two
inequalities  into \eqref{sec5.decayB} and $\eqref{sec5.decayrho}$
respectively gives
\begin{eqnarray*}
&\displaystyle \|B(t)\|\leq C (1+t)^{-\frac{3}{4}}
 (\|U_{0}\|_{L^{1}\cap\dot{H}^{2}}+\|B_{0}\|_{L^{1}\cap L^{2}}+\mathcal {E}_{N,\infty}(U(t))),\\
&\displaystyle  \|[\rho_{i},\rho_{e}]\|\leq C (1+t)^{-\frac{3}{4}}
 (\|U_{0}\|_{L^{1}\cap\dot{H}^{2}}+\|[\rho_{i0},\rho_{e0}]\|_{L^{1}\cap L^{2}}+\mathcal {E}_{N,\infty}(U(t))),
\end{eqnarray*}
which imply $\eqref{lem.tildeB}$.  This completes the proof of Lemma
$\ref{lem.Bsigma}$.
\end{proof}
Now, the rest is to prove the uniform-in-time bound of $\mathcal
 {E}_{N,\infty}(U(t))$ which yields the time-decay rates of the
 Lyapunov functional $\mathcal
 {E}_{N}(U(t))$ and thus $\|U(t)\|_{N}^{2}$. In fact, by taking $\ell =\frac{3}{2}+\epsilon$
in $\eqref{sec5.ED}$ with $\epsilon>0$ small enough, one has
\begin{eqnarray*}
\begin{aligned}
 & (1+t)^{\frac{3}{2}+\epsilon}\mathcal {E}_{N}(U(t))+\lambda
 \int_{0}^{t}(1+s)^{\frac{3}{2}+\epsilon}\mathcal {D}_{N}(U(s))d s \\
 \leq & C \mathcal {E}_{N+2}(U_{0})+ C
 \int_{0}^{t}(1+s)^{\frac{1}{2}+\epsilon}(\|
  B(s)\|^{2}+\|[\rho_{i}(s),\rho_{e}(s)]\|^{2})d s.
\end{aligned}
\end{eqnarray*}
Here, using $\eqref{lem.tildeB}$ and the fact that $\mathcal
 {E}_{N,\infty}(U(t))$ is non-decreasing in $t$, it further holds
 that
\begin{eqnarray*}
\begin{aligned}
 &\int_{0}^{t}(1+s)^{\frac{1}{2}+\epsilon}(\|
 B\|^{2}+\|[\rho_{i}(s),\rho_{e}(s)]\|^{2})d s\\
  \leq  & C(1+t)^{\epsilon}\Big(\mathcal
 {E}_{N,\infty}^{2}(U(t))+\|[\rho_{i0},\rho_{e0},B_{0}]\|\|_{L^{1}\cap L^{2}}^2+\left.\|U_{0}\|_{L^{1}\cap \dot{H}^{2}}^{2}\right).
\end{aligned}
\end{eqnarray*}
Therefore, it follows that
\begin{eqnarray*}
\begin{aligned}
 & (1+t)^{\frac{3}{2}+\epsilon}\mathcal {E}_{N}(V(t))+\lambda
 \int_{0}^{t}(1+s)^{\frac{3}{2}+\epsilon}\mathcal {D}_{N}(V(s))d s \\
 \leq & C \mathcal {E}_{N+2}(V_{0})+ C(1+t)^{\epsilon}\left(\mathcal
 {E}_{N,\infty}^{2}(U(t))+\|[\rho_{i0},\rho_{e0},B_{0}]\|\|_{L^{1}\cap L^{2}}^2+\|U_{0}\|_{L^{1}\cap \dot{H}^{2}}^{2}\right),
\end{aligned}
\end{eqnarray*}
which implies
\begin{eqnarray*}
\begin{aligned}
      (1+t)^{\frac{3}{2}}\mathcal {E}_{N}(U(t))
 \leq & C \left( \mathcal {E}_{N+2}(U_{0})+
\mathcal
 {E}_{N,\infty}^{2}(U(t))+\|[\rho_{i0},\rho_{e0},B_{0}]\|\|_{L^{1}\cap L^{2}}^2+\|U_{0}\|_{L^{1}\cap \dot{H}^{2}}^{2}\right) .
\end{aligned}
\end{eqnarray*}
Thus, one has
\begin{eqnarray*}
\mathcal {E}_{N,\infty}(U(t))
 \leq C \left( \epsilon_{N+2}^{2}(U_{0})+
\mathcal {E}_{N,\infty}^{2}(U(t))\right).
\end{eqnarray*}
Here, recall the definition of $\epsilon_{N+2}(U_{0})$. Since
$\epsilon_{N+2}(U_{0})>0$ is sufficiently small, $\mathcal
{E}_{N,\infty}(U(t)) \leq C \epsilon_{N+2}^{2}(U_{0})$ holds true
for any $t\geq 0$, which implies
\begin{eqnarray*}
\|U(t)\|_{N} \leq C \mathcal {E}_{N}(U(t))^{1/2}
 \leq C  \epsilon_{N+2}(U_{0})(1+t)^{-\frac{3}{4}},
\end{eqnarray*}
for any $t\geq 0$. This proves \eqref{V.decay} in Theorem
\ref{pro.2.2}.

\subsubsection{Time rate for the higher-order  instant
energy functional}
\begin{lemma}\label{estimate2}
Let $U=[\rho_{\alpha},u_{\alpha},E,B]$ be the solution to the Cauchy
problem $\eqref{2.5}$-$ \eqref{NI}$ with initial data
$U_{0}=[\rho_{\alpha0},u_{\alpha0},E_{0},B_{0}]$ satisfying
$\eqref{NC}$ in the sense of Proposition $\ref{pro.2.1}$. Then if
$\mathcal {E}_{N}(U_{0})$ is sufficiently small, there are the
high-order instant energy functional $\mathcal {E}_{N}^{h}(\cdot)$
and the corresponding dissipation rate $\mathcal {D}_{N}^{h}(\cdot)$
such that
\begin{eqnarray}\label{sec5.high}
&& \frac{d}{dt}\mathcal {E}_{N}^{h}(U(t))+\lambda\mathcal
{D}^{h}_{N}(U(t))\leq C\sum_{\alpha=i,e}\|\nabla
\rho_{\alpha}\|^{2},
\end{eqnarray}
holds for any $ t \geq 0$.
\end{lemma}

\begin{proof}
The proof can be done by modifying the proof of Theorem
$\ref{estimate}$ a little. In fact, by making the energy estimates
on the only  high-order derivatives, then corresponding to
$\eqref{3.3}$, $\eqref{step2}$, $\eqref{step3}$ and $\eqref{step4}$,
it can be re-verified that

\begin{equation*}
\begin{aligned}
&\frac{1}{2}\frac{d}{dt}\left(\sum_{1\leq |l|\leq N}\sum_{
\alpha=i,e}\int_{\mathbb{R}^3}\frac{p'_{\alpha}(\rho_{\alpha}+1)}{\rho_{\alpha}+1}
   |\partial^{l}\rho_{\alpha}|^2 +m_{\alpha}(\rho_{\alpha}+1)
   |\partial^{l}u_{\alpha}|^2)dx+\frac{1}{4\pi}\|\nabla[E,B]\|_{N-1}^{2}\right)\\
&+\sum_{1\leq |l|\leq N}\int_{\mathbb{R}^3}\sum_{\alpha=i,e}
m_{\alpha}(\rho_{\alpha}+1)
   |\partial^{l}u_{\alpha}|^2dx\leq  C(\|U\|_{N}+\|U\|_{N}^{2})\sum_{
\alpha=i,e}\|\nabla[\rho_{\alpha}u_{\alpha}]\|_{N-1}^{2}.
 \end{aligned}
\end{equation*}
\begin{eqnarray*}
&&\begin{aligned}
  &\frac{d}{dt}\sum_{1\leq |l|\leq
N-1}\sum_{\alpha=i,e}m_{\alpha}\langle
\partial^{l}u_{\alpha},
\partial^{l}\nabla\rho_{\alpha}\rangle
  +\lambda\sum_{\alpha=i,e}\|\nabla^{2}\rho_{\alpha}\|_{N-2}^{2}+4\pi\left\|\sum_{\alpha=i,e}q_{\alpha}\nabla\rho_{\alpha}
\right\|_{N-2}^2\\
  \leq &
  C\sum_{\alpha=i,e}\|\nabla u_{\alpha}\|_{N-1}^{2}+C\|U\|_{N}^{2}\left(\sum_{\alpha=i,e}\|\nabla
[\rho_{\alpha},u_{\alpha}]\|_{N-1}^{2}\right),
  \end{aligned}
\end{eqnarray*}
\begin{eqnarray*}
&&\begin{aligned}
  &\dfrac{d}{dt}\sum_{1\leq |l|\leq
N-1}\sum_{\alpha=i,e}m_{\alpha}\left\langle
\partial^{l}u_{\alpha},-\frac{q_{\alpha}}{T_{\alpha}}\partial^{l}E\right\rangle+\dfrac{1}{4\pi}\|\nabla\nabla\cdot E\|_{N-2}^{2}+\lambda\|\nabla E\|^{2}_{N-2} \\
  \leq &
  C\sum_{\alpha=i,e}\|\nabla u_{\alpha}\|_{N-1}^{2}+C\sum_{\alpha=i,e}\|\nabla
  u_{\alpha}\|_{N-1}\|\nabla^{2}
B\|_{N-3}+C\|U\|_{N}^{2}\left(\sum_{\alpha=i,e}\|\nabla
[\rho_{\alpha},u_{\alpha}]\|_{N-1}^{2}\right).
 \end{aligned}
\end{eqnarray*}
and
\begin{eqnarray*}
&&\begin{aligned}
   &-\frac{d}{dt}\sum_{|l|\leq N-2}\langle
\partial^{l}E,\nabla \times \partial^{l}B\rangle+\lambda\|\nabla^{2}
B\|^{2}_{N-3}\\
   \leq & C\|\nabla^{2}E\|_{N-3}^{2}+\sum_{\alpha=i,e}\|\nabla
   u_{\alpha}\|_{N-1}^{2}
   +C\|U\|_{N}^{2}\left(\sum_{\alpha=i,e}\|\nabla
[\rho_{\alpha},u_{\alpha}]\|_{N-1}^{2}\right),
\end{aligned}
\end{eqnarray*}
Here, the details of proof are omitted for simplicity. Now, similar
to $\eqref{3.12}$, let us define
\begin{equation}\label{def.high}
\begin{aligned}
\mathcal {E}_{N}^{h}(U(t))=&\displaystyle\sum_{1\leq |l|\leq
N}\sum_{
\alpha=i,e}\int_{\mathbb{R}^3}\frac{p'_{\alpha}(\rho_{\alpha}+1)}{\rho_{\alpha}+1}
   |\partial^{l}\rho_{\alpha}|^2 +m_{\alpha}(\rho_{\alpha}+1)
   |\partial^{l}u_{\alpha}|^2)dx+\frac{1}{4\pi}\|\nabla[E,B]\|_{N-1}^{2}\\[3mm]
&\displaystyle+\kappa_{1}\sum_{1\leq|l|\leq
N-1}\sum_{\alpha=i,e}m_{\alpha}\langle
\partial^{l}u_{\alpha},
\partial^{l}\nabla\rho_{\alpha}\rangle+\kappa_{2}\sum_{1\leq |l|\leq
N-1}\sum_{\alpha=i,e}m_{\alpha}\left\langle
\partial^{l}u_{\alpha},-\frac{q_{\alpha}}{T_{\alpha}}\partial^{l}E\right\rangle\\[3mm]
&\displaystyle-\kappa_{3}\sum_{1\leq |l|\leq N-2}\langle
\partial^{l}E,\nabla \times \partial^{l}B\rangle,
\end{aligned}
\end{equation}

Similarly, one can choose $0<\kappa_{3}\ll\kappa_{2}\ll\kappa_{1}\ll
1$ with $\kappa_{2}^{3/2}\ll\kappa_{3}$ such that $\mathcal
{E}_{N}^{h}(U(t))\sim \|\nabla U(t)\|_{N-1}^{2}$. Furthermore, the
linear combination of previously obtained four estimates with
coefficients corresponding to $\eqref{def.high}$ yields
$\eqref{sec5.high}$ with $\mathcal {D}_{N}^{h}(\cdot)$ defined in
$\eqref{de.Dh}$. This completes the proof of Lemma \ref{estimate2}.
\end{proof}

By comparing \eqref{de.Dh} with $\eqref{de.Eh}$ for the definitions
of $ \mathcal {E}_{N}^{h}(U(t))$ and  $ \mathcal {D}_{N}^{h}(U(t))$,
it follows from $\eqref{sec5.high}$ that
\begin{eqnarray*}
&& \frac{d}{dt}\mathcal {E}_{N}^{h}(U(t))+\lambda\mathcal
{E}^{h}_{N}(U(t))\leq
 C\left(\|\nabla B\|^{2}+\|\nabla^{N}[E,B]\|^{2}+\sum_{\alpha=i,e}\|\nabla
\rho_{\alpha}\|^{2}\right),
\end{eqnarray*}
which implies
\begin{eqnarray}\label{sec.ee}
&&\begin{aligned}
 &\mathcal {E}_{N}^{h}(U(t))\leq {\exp({-\lambda
t})}
\mathcal{E}_{N}^{h}(U_{0}) \\
& +C\int_{0}^{t}{\exp\{-\lambda(t-s)\}} \left(\|\nabla
B(s)\|^{2}+\|\nabla^{N}[E,B](s)\|^{2}+\sum_{\alpha=i,e}\|\nabla
\rho_{\alpha}(s)\|^{2}\right)d s.
\end{aligned}
\end{eqnarray}
To estimate the time integral term on the r.h.s. of the above
inequality, one has

\begin{lemma}\label{estimate3}
Under the assumptions of Theorem $\ref{pro.2.1}$, if $
\epsilon_{N+6}(U_{0})$ defined in $\eqref{def.epsi}$ is sufficiently small then
\begin{eqnarray}\label{sec5.highBE}
&&\begin{aligned} \|\nabla B(t)\|^{2}+\|\nabla^{N}[E(t),B(t)]\|^{2}
+&\sum_{\alpha=i,e}\|\nabla \rho_{\alpha}(t)\|^{2} \leq
C\epsilon_{N+6}^{2}(U_{0})(1+t)^{-\frac{5}{2}}
\end{aligned}
\end{eqnarray}
holds for any $ t \geq 0$.
\end{lemma}
For this time, suppose that the above lemma is true. Then by applying
$\eqref{sec5.highBE}$ to $\eqref{sec.ee}$,  it is immediate to
obtain
\begin{eqnarray*}
\mathcal {E}_{N}^{h}(U(t))\leq {\exp\{-\lambda t\}}
\mathcal{E}_{N}^{h}(U_{0})+C
\epsilon_{N+6}^{2}(U_{0})(1+t)^{-\frac{5}{2}},
\end{eqnarray*}
which proves  \eqref{nablaV.decay} in Theorem \ref{pro.2.2}.

\medskip
\noindent{\it Proof of Lemma \ref{estimate3}:} Suppose that
$\epsilon_{N+6}(U_{0})>0$ is sufficiently small. Notice that, by the
first part of Theorem \ref{pro.2.2},
\begin{eqnarray}\label{UN+4}
\|U(t)\|_{N+4}  \leq C  \epsilon_{N+6}(U_{0})(1+t)^{-\frac{3}{4}}.
\end{eqnarray}
Similar to obtaining $\eqref{sec5.decayB}$, one can apply the linear
estimate on $\rho_{\alpha},B$ and  letting $m=1,\ q=r=2,\ p=1,\
\ell=\dfrac{5}{2}$ in Corollary \ref{corollary.decayL} to the mild
form \eqref{sec5.U} respectively, and the linear estimate on $E,B$
and  letting $m=N,\ q=r=2,\ p=1,\ \ell=\dfrac{5}{2}$ so that
\begin{multline}\label{sec5.nablarho}
 \|\nabla \rho_{\alpha}(t)\|\leq C (1+t)^{-\frac{5}{4}}
 \|U_{0}\|_{ L^1\cap \dot{H}^{4}}+{\exp\{-\lambda t\}}
 \|\nabla[\rho_{i0},\rho_{e0}]\|\\
+C
\int_{0}^{t}(1+t-s)^{-\frac{5}{4}}\|[g_{1\alpha}(s),g_{2\alpha}(s),g_{3}(s)]\|_{L^{1}\cap\dot{H}^{4}}ds\\
+C
\int_{0}^{t}{\exp\{-\lambda(t-s)\}}\|\nabla[g_{1i}(s),g_{1e}(s)]\|ds,
\end{multline}
\begin{eqnarray}\label{sec5.nablaB}
&&\begin{aligned}\|\nabla B(t)\|\leq C (1+t)^{-\frac{5}{4}}
 \| & U_{0}\|_{ L^1 \cap \dot{H}^{4}}+C{\exp\{-\lambda t\}}\|\nabla B_{0}\|\\
&+C
\int_{0}^{t}(1+t-s)^{-\frac{5}{4}}\|[g_{1\alpha}(s),g_{2\alpha}(s),g_{3}(s)]\|_{L^{1}\cap\dot{H}^{4}}ds,
\end{aligned}
\end{eqnarray}
and
\begin{multline}\label{sec5.nablaNE}
 \|\nabla^{N}E(t)\|\leq
C(1+t)^{-\frac{5}{4}}
 \|U_{0}\|_{ L^1 \cap \dot{H}^{N+3}}+{\exp\{-\lambda t\}}
 \|\nabla^{N+1}[\rho_{i0},\rho_{e0},B_{0}]\|\\
+C
\int_{0}^{t}(1+t-s)^{-\frac{5}{4}}\|[g_{1\alpha}(s),g_{2\alpha}(s),g_{3}(s)]\|_{L^{1}\cap\dot{H}^{N+3}}ds\\
+C
\int_{0}^{t}{\exp\{-\lambda(t-s)\}}\|\nabla^{N+1}[g_{1i}(s),g_{1e}(s)]\|ds,
\end{multline}
\begin{multline}\label{sec5.nablaNB}
 \|\nabla^{N}B(t)\|\leq
C(1+t)^{-\frac{5}{4}}
 \|U_{0}\|_{ L^1 \cap \dot{H}^{N+3}}+{\exp\{-\lambda t\}}
 \|\nabla^{N}B_{0}\|\\
+C
\int_{0}^{t}(1+t-s)^{-\frac{5}{4}}\|[g_{1\alpha}(s),g_{2\alpha}(s),g_{3}(s)]\|_{L^{1}\cap\dot{H}^{N+3}}ds.
\end{multline}
Recalling the definition $\eqref{sec5.ggg}$, it is straightforward
to verify
\begin{eqnarray*}
\|[g_{1\alpha}(s),g_{2\alpha}(s),g_{3}(s)]\|_{L^{1}\cap\dot{H}^{4}}\leq
C\|U(t)\|_{5}^{2},
\end{eqnarray*}
\begin{eqnarray*}
\|[g_{1\alpha}(s),g_{2\alpha}(s),g_{3}(s)]\|_{L^{1}\cap\dot{H}^{N+3}}
\leq C\|U(t)\|^{2}_{N+4},
\end{eqnarray*}
\begin{eqnarray*}
\|\nabla[g_{1i}(s),g_{1e}(s)]\|\leq C\|U(t)\|_{3}^{2},\ \ \ \ \
\|\nabla^{N+1}[g_{1i}(s),g_{1e}(s)]\|\leq C\|U(t)\|^{2}_{N+2}.
\end{eqnarray*}
The above estimates together with $\eqref{UN+4}$ give
\begin{multline*}
\|[g_{1\alpha}(s),g_{2\alpha}(s),g_{3}(s)]\|_{L^{1}\cap\dot{H}^{4}}+\|[g_{1\alpha}(s),g_{2\alpha}(s),g_{3}(s)]\|_{L^{1}\cap\dot{H}^{N+3}}\\
+\|\nabla[g_{1i}(s),g_{1e}(s)]\|+\|\nabla^{N+1}[g_{1i}(s),g_{1e}(s)]\|
 \leq
C\|U(t)\|^{2}_{N+4}\leq C
\epsilon_{N+6}^{2}(U_{0})(1+s)^{-\frac{3}{2}}.
\end{multline*}
Then it follows from $\eqref{sec5.nablarho}$, $\eqref{sec5.nablaB}$,
$\eqref{sec5.nablaNB}$ and  $\eqref{sec5.nablaNE}$ that
\begin{eqnarray*}
\begin{aligned}
\|\nabla B(t)\|+\|\nabla^{N}[E(t),B (t)]\|
+\sum_{\alpha=i,e}\|\nabla \rho_{\alpha}(t)\|^{2} \leq
C\epsilon_{N+6}(U_{0})(1+t)^{-\frac{5}{4}},
\end{aligned}
\end{eqnarray*}
where the smallness of $\epsilon_{N+6}(U_{0})$ has been used. The proof
of Lemma $\ref{estimate3}$ is complete.

\subsubsection{Time rate in $L^{2}$}

Recall that Theorem \ref{pro.2.1} shows that for $N\geq 3$, if
$\epsilon_{N+2}(U_{0})$ is sufficiently small then
\begin{eqnarray}\label{UN}
\|U(t)\|_{N}  \leq C  \epsilon_{N+2}(U_{0})(1+t)^{-\frac{3}{4}},
\end{eqnarray}
and if $\epsilon_{N+6}(U_{0})$ is sufficiently small then
\begin{eqnarray*}
\|\nabla U(t)\|_{N-1}  \leq C
\epsilon_{N+6}(U_{0})(1+t)^{-\frac{5}{4}}.
\end{eqnarray*}
Now, we write down the $L^2$  time-decay rates of $[\rho_{\alpha},B]$ and
$[u_{\alpha},E]$ as follows.

\medskip
\noindent\emph{Estimate  on} $\|[\rho_{\alpha},B]\|_{L^{2}}$. It is
easy to see from $\eqref{UN}$ that
\begin{eqnarray}
\|B(t)\|+\sum_{i=i,e}\|\rho_{\alpha}\| \leq C
\epsilon_{5}(U_{0})(1+t)^{-\frac{3}{4}}.\label{rate.brho}
\end{eqnarray}

 \medskip
\noindent\emph{Estimate  on} $\|[u_{\alpha},E]\|_{L^{2}}$. Applying
the  second and the third linear estimate on $[u_{\alpha},E]$ with  $m=0,\ q=r=2,\ p=1,\ \ell=5/2$ in Corollary
\ref{corollary.decayL} to the mild form \eqref{sec5.U}, one has
\begin{multline*}
\|[u_{\alpha},E](t)\| \leq C(1+t)^{-\frac{5}{4}}
 \|U_{0}\|_{L^{1}\cap\dot{H}^{3}}+{\exp\{-\lambda t\}}\|\nabla [\rho_{i0},\rho_{e0},B_{0}]\|\\
+C
\int_{0}^{t}(1+t-s)^{-\frac{5}{4}}\|[g_{1\alpha}(s),g_{2\alpha}(s),g_{3}(s)]\|_{L^{1}\cap\dot{H}^{3}}ds\\
+C \int_{0}^{t}{\exp\{-\lambda(t-s)\}}\|\nabla
[g_{1i}(s),g_{1e}(s)]\|ds.
\end{multline*}
By $\eqref{UN}$, it follows that
\begin{equation*}
\|\nabla [g_{1i}(s),g_{1e}(s)]\|
+\|[g_{1\alpha}(s),g_{2\alpha}(s),g_{3}(s)]\|_{L^{1}\cap\dot{H}^{3}}
\leq C\|U(t)\|^{2}_{4}\leq
C\epsilon_{6}^{2}(U_{0})(1+t)^{-\frac{3}{2}}.
\end{equation*}
Therefore, one has
\begin{eqnarray}\label{uL2}
\|[u_{\alpha},E](t)\| \leq C\epsilon_{6}(U_{0})(1+t)^{-\frac{5}{4}}.
\end{eqnarray}

\subsection{Asymptotic rate to diffusion waves}
In this section  we shall prove the main Theorem \ref{thm.main} on the
large-time asymptotic behavior of the obtained solutions.

First of all, we prove in the following lemma that the solution $U(x,t)=[\rho_{\alpha},
u_{\alpha},E,B]$ to the nonlinear Cauchy problem
\eqref{2.5}-\eqref{NC} can be approximated by the one of the corresponding
linearized problem $\eqref{Linear}$-$\eqref{LNC}$ in large time.
%

\begin{lemma}\label{asy.small.lem}
Suppose that $\epsilon_{11}(U_{0})>0$ is sufficiently small, and
$U(x,t)=[\rho_{\alpha}, u_{\alpha},E,B]$ is a solution to the Cauchy
problem \eqref{2.5}-\eqref{NC} with initial data $U_0$. Then it holds that
\begin{eqnarray}
&&\left\|
\rho_{\alpha}(t)-\FP_{1\alpha}e^{tL}U_{0}\right\|\leq C
(1+t)^{-\frac{5}{4}}, \label{asy.small.dec.rho}\\
&&\left\|u_{\alpha}(t)-\FP_{2\alpha}e^{tL}U_{0}\right\|\leq
C (1+t)^{-\frac{7}{4}},\label{asy.small.dec.u}\\
&&\left\|E(t)-\FP_{3}e^{tL}U_{0}\right\|\leq C
(1+t)^{-\frac{7}{4}},\label{asy.small.dec.E}\\
&&\left\|B(t)-\FP_4e^{tL}U_{0}\right\|\leq C
(1+t)^{-\frac{5}{4}},\label{asy.small.dec.B}
\end{eqnarray}
for any $t\geq 0$.
\end{lemma}

\begin{proof}
We rewrite each component of solutions $U(x,t)=[\rho_{\alpha}, u_{\alpha},E,B]$ to \eqref{2.5} as the mild forms
by the Duhamel's principle:
\begin{equation}\label{opt.eq1}
\rho_{\alpha}(x,t)=\FP_{1\alpha}e^{tL}U_{0}+\int_0^t\FP_{1\alpha} e^{(t-s)L}[\nabla
\cdot f_{\alpha}(s),g_{2\alpha}(s),g_{3}(s),0]ds,
\end{equation}
\begin{equation}\label{opt.eq2}
u_{\alpha}(x,t)=\FP_{2\alpha}e^{tL}U_{0}+\int_0^t\FP_{2\alpha}e^{(t-s)L}[\nabla
\cdot f_{\alpha}(s),g_{2\alpha}(s),g_{3}(s),0]ds,
\end{equation}
for $\alpha=i,e$, and
\begin{equation*}
E(x,t)=\FP_3e^{tL}U_{0}+\int_0^t\FP_3e^{(t-s)L}[\nabla \cdot
f_{\alpha}(s),g_{2\alpha}(s),g_{3}(s),0]ds,
\end{equation*}
\begin{equation*}
B(x,t)=\FP_4e^{tL}U_{0}+\int_0^t\FP_4e^{(t-s)L}[\nabla \cdot
f_{\alpha}(s),g_{2\alpha}(s),g_{3}(s),0]ds.
\end{equation*}
Denote $N(s)=[\nabla \cdot f_{\alpha}(s),g_{2\alpha}(s),g_{3}(s),0]$
as in Section \ref{sec.spec.}. In what follows we only prove
\eqref{asy.small.dec.rho} and \eqref{asy.small.dec.u}, and the other two estimates \eqref{asy.small.dec.E} and \eqref{asy.small.dec.B} can be
proved in a similar way. One can apply the linear estimate on
$\FP_{1\alpha}e^{tL}N_{0}$  to the mild form
\eqref{opt.eq1}  by letting $m=0,\ q=r=2,\ p=1,\
\ell=5/2$ in Theorem \ref{thm.special}, so as to obtain
\begin{multline}\label{asy.lg.prf.eq1}
\left\|\rho_{\alpha}(t)-\FP_{1\alpha}e^{tL}U_{0}\right\|
 \leq \displaystyle \int_0^t\left\|\FP_{1\alpha}e^{(t-s)L}[\nabla
\cdot f_{\alpha}(s),g_{2\alpha}(s),g_{3}(s),0]\right\|ds\\
\leq  C\int_0^t(1+t-s)^{-\frac{5}{4}}\left(\|N(s)\|_{L^{1}\cap
\dot{H}^{3}}+\|[f_{i},f_{e}](s)\|_{L^{1}}\right)+{\exp\{\lambda
(t-s)\}}\|\nabla[f_{i},f_{e}](s)\|ds.
\end{multline}
Recalling the definition $\eqref{sec5.ggg}$, it is straightforward
to verify
\begin{eqnarray*}
\|N(s)\|_{L^{1}\cap \dot{H}^{3}}+\|[f_{i},f_{e}](s)\|_{L^{1}}\leq
C\|U(s)\|_{4}^{2}\leq C\epsilon_{6}^{2}(U_{0})(1+s)^{-\frac{3}{2}},
\end{eqnarray*}
and
\begin{eqnarray*}
\|\nabla[f_{i},f_{e}](s)\|\leq C\|U(s)\|_{4}^{2}\leq
C\epsilon_{6}^{2}(U_{0})(1+s)^{-\frac{3}{2}}.
\end{eqnarray*}
Plugging these estimates into \eqref{asy.lg.prf.eq1}, it follows that
\begin{equation*}
\left\|
\rho_{\alpha}(t)-\FP_{1\alpha}e^{tL}U_{0}\right\|\leq C
(1+t)^{-\frac{5}{4}}.
\end{equation*}
Applying the linear estimate on $\FP_{2\alpha}e^{tL}N_{0}$ to the mild form \eqref{opt.eq2} by letting
$m=0,\ q=r=2,\ p=1,\ \ell=7/2$ in Theorem \ref{thm.special}
gives
\begin{multline}\label{asy.lg.prf.eq2}
\left\|u_{\alpha}(t)-\FP_{2\alpha}e^{tL}U_{0}\right\|
 \leq \displaystyle \int_0^t\left\|\FP_{2\alpha}e^{(t-s)L}[\nabla
\cdot f_{\alpha}(s),g_{2\alpha}(s),g_{3}(s),0]\right\|ds\\
\leq  C\int_0^t(1+t-s)^{-\frac{7}{4}}\left(\|N(s)\|_{L^{1}\cap
\dot{H}^{4}}+\|[f_{i},f_{e}](s)\|_{L^{1}}\right)+{\exp\{\lambda
(t-s)\}}\|\nabla^{2}[f_{i},f_{e}](s)\|ds.
\end{multline}
As before, recall the definition $\eqref{sec5.ggg}$ and the time-decay rates \eqref{rate.brho} and \eqref{uL2}. We first estimate $L^{1}$ norms of those terms
without any derivative as
\begin{eqnarray*}
\begin{aligned}
&\|u_{\alpha}\times B\|_{L^{1}}\leq \|u_{\alpha}\| \|B\|\leq
C\epsilon_{6}^{2}(U_{0})(1+s)^{-\frac{5}{4}}(1+s)^{-\frac{3}{4}}\leq
C\epsilon_{6}^{2}(U_{0})(1+s)^{-2},\\
&\|\rho_{\alpha}u_{\alpha} \|_{L^{1}}\leq \|u_{\alpha}\|
\|\rho_{\alpha}\|\leq
C\epsilon_{6}^{2}(U_{0})(1+s)^{-\frac{5}{4}}(1+s)^{-\frac{3}{4}}\leq
C\epsilon_{6}^{2}(U_{0})(1+s)^{-2},\\
&\|f_{\alpha}(s)\|_{L^{1}}\leq \|u_{\alpha}\| \|\rho_{\alpha}\|\leq
C\epsilon_{6}^{2}(U_{0})(1+s)^{-\frac{5}{4}}(1+s)^{-\frac{3}{4}}\leq
C\epsilon_{6}^{2}(U_{0})(1+s)^{-2}.
\end{aligned}
\end{eqnarray*}
For other terms with one derivative, for $\rho_{\alpha}\nabla
\cdot u_{\alpha}$, one has
\begin{eqnarray*}
\|\rho_{\alpha}\nabla \cdot u_{\alpha}\|_{L^{1}}\leq \|\nabla
u_{\alpha}\| \|\rho_{\alpha}\|\leq
C\epsilon_{9}(U_{0})(1+s)^{-\frac{5}{4}}\epsilon_{5}
(U_{0})(1+s)^{-\frac{3}{4}}\leq C\epsilon_{9}^{2}(U_{0})(1+s)^{-2},
\end{eqnarray*}
and similarly it follows that
\begin{eqnarray*}
\|u_{\alpha} \cdot\nabla\rho_{\alpha} \|_{L^{1}}+\|u_{\alpha}
\cdot\nabla u_{\alpha} \|_{L^{1}}+\|\rho_{\alpha}\nabla
\rho_{\alpha}\|_{L^{1}}\leq C\epsilon_{9}^{2}(U_{0})(1+s)^{-2}.
\end{eqnarray*}
For $L^2$ norms,  by calculating  for $|l|=4$,
\begin{eqnarray*}
\begin{aligned}
\|\partial^{l}(u_{\alpha}\times B)\|\leq &
\|u_{\alpha} \|_{L^{\infty}}\|\partial^{l} B\|+\|
B\|_{L^{\infty}}\|\partial^{l}u_{\alpha} \|
 \leq  C\|\nabla U\|_{3}^{2}\leq
 \epsilon_{10}^{2}(U_{0})(1+s)^{-\frac{5}{2}},
\end{aligned}
\end{eqnarray*}
and
\begin{equation*}
\|\partial^{l}(u_{\alpha}\cdot \nabla \rho_{\alpha})\|\leq
\|u_{\alpha} \|_{L^{\infty}}\|\partial^{l} \nabla
\rho_{\alpha}\|+\|\nabla \rho_{\alpha}
\|_{L^{\infty}}\|\partial^{l}u_{\alpha} \|
 \leq  C\|\nabla U\|_{4}^{2}\leq
 \epsilon_{11}^{2}(U_{0})(1+s)^{-\frac{5}{2}},
\end{equation*}
it is direct to verify that
\begin{eqnarray*}
\|N(s)\|_{\dot{H}^{4}}+\|\nabla^{2}[f_{i},f_{e}](s)\| \leq C\|\nabla
U(s)\|_{4}^{2}\leq C\epsilon_{11}^{2}(U_{0})(1+s)^{-\frac{5}{2}}.
\end{eqnarray*}
Plugging the above inequalities into \eqref{asy.lg.prf.eq2} gives
\begin{equation*}
\left\|u_{\alpha}(t)-\FP_{2\alpha}e^{tL}U_{0}\right\|\leq
C (1+t)^{-\frac{7}{4}}.
\end{equation*}
This then completes the proof of Lemma \ref{asy.small.lem}.
\end{proof}

For the solution $U(x,t)=[\rho_{\alpha},u_{\alpha},E,B]$ to the
Cauchy problem \eqref{2.5}-\eqref{NC} and the desired large-time asymptotic profile
$\overline{U}(x,t)=[\overline{\rho},\overline{u}_{\alpha},\overline{E},\overline{B}]$, their difference
can be rewritten as
\begin{equation*}
U-\overline{U}=(U-e^{tL}U_0) +(e^{tL}U_0-\overline{U}), 
\end{equation*}
that is,
\begin{equation*}
\begin{aligned}
&\rho_{\alpha}-\overline{\rho}=\left(\rho_{\alpha}-\FP_{1\alpha}e^{
tL}U_0\right)+\left(\FP_{1\alpha}e^{ tL}U_{0}-\overline{\rho}\right),
\\
&u_{\alpha}-\overline{u}_\al=\left(u_{\alpha}-\FP_{2\alpha}e^{
tL}U_0\right)+\left(\FP_{2\alpha}e^{tL}U_{0}-\overline{u}_{\alpha}\right),
\\
&E-\overline{E}=\left(E-\FP_3e^{
tL}U_0\right)+\left(\FP_3e^{tL}U_{0}-\overline{E}\right),
\\
&B-\overline{B}=\left(B-\FP_4e^{
tL}U_0\right)+\left(\FP_4e^{tL}U_{0}-\overline{B}\right).
\end{aligned}
\end{equation*}
Therefore Theorem \ref{thm.main} follows from Lemma \ref{asy.small.lem}, Theorem \ref{thm.decaypar} and Theorem \ref{thm.decayperp}. 
\qed

\medskip

\noindent{\bf Acknowledgements:}\ \  RJD was supported by the
General Research Fund (Project No.~400912) from RGC of Hong Kong.
QQL and CJZ were supported by the National Natural Science
Foundation of China $\#$11331005, the Program for Changjiang
Scholars and Innovative Research Team in University $\#$IRT13066,
and the Special Fund Basic Scientific Research of Central Colleges
$\#$CCNU12C01001. QQL was also supported by excellent doctorial
dissertation cultivation grant from Central China Normal
University. The authors would like to thank anonymous referees for many helpful comments on the paper. 

\medskip

\end{document}